\documentclass[12pt, a4paper]{amsart} 

\usepackage[latin1]{inputenc}
\usepackage[T1]{fontenc}
\usepackage{subfigure}
\usepackage{amsthm}
\usepackage{amsmath}
\usepackage{amssymb}
\usepackage{dsfont}
\usepackage[english]{babel}
\usepackage{enumitem}
\usepackage{graphicx}
\usepackage{hyperref}
\usepackage{tikz}

\usetikzlibrary{automata}

\newcommand{\T}{\mathcal{T}}
\newcommand{\Ttilde}{\widetilde{\mathcal{T}}}
\newcommand{\calI}{\mathcal{I}}
\newcommand{\calP}{\mathcal{P}}
\newcommand{\calR}{\mathcal{R}}
\newcommand{\bfb}{\boldsymbol{b}}
\newcommand{\bigOh}{\mathcal{O}}
\newcommand{\bfv}{\boldsymbol{v}}
\newcommand{\bfw}{\boldsymbol{w}}
\newcommand{\bfH}{\boldsymbol{H}}
\newcommand{\bfu}{\boldsymbol{u}}
\newcommand{\bfn}{\boldsymbol{n}}
\newcommand{\bfe}{\boldsymbol{e}}
\newcommand{\bfr}{\boldsymbol{r}}

\newcommand{\bfmu}{\boldsymbol{\mu}}
\newcommand{\bfh}{\boldsymbol{h}}
\newcommand{\bfs}{\boldsymbol{s}}
\newcommand{\bfl}{\boldsymbol{l}}
\newcommand{\bfB}{\boldsymbol{B}}
\newcommand{\bfdelta}{\boldsymbol{\delta}}
\newcommand{\bflambda}{\boldsymbol{\lambda}}
\newcommand{\bfones}{\boldsymbol{1}}
\newcommand{\eps}{\varepsilon}
\newcommand{\bfbeta}{\boldsymbol{\beta}}
\newcommand{\bfeps}{\boldsymbol{\eps}}

\newcommand{\Expect}{\mathbb{E}}
\newcommand{\Var}{\mathbb{V}}

\DeclareMathOperator{\Res}{Res}
\DeclareMathOperator{\lcm}{lcm}
\newcommand{\integers}{\mathbb{Z}}
\newcommand{\lmin}{l_{\mathrm{min}}}
\newcommand{\reals}{\mathbb{R}}
\newcommand{\naturals}{\mathbb{N}}
\newcommand{\complexes}{\mathbb{C}}
\renewcommand{\MR}[1]{}
\newcommand{\floor}[1]{\lfloor #1\rfloor}
\newcommand{\fpart}[1]{\{ #1\}}

\newtheorem{theorem}{Theorem}
\newtheorem{lemma}{Lemma}[section]
\newtheorem{corollary}[lemma]{Corollary}
\newtheorem{statement}[lemma]{Statement}
\theoremstyle{definition}
\newtheorem{definition}[lemma]{Definition}
\theoremstyle{remark}
\newtheorem{example}[lemma]{Example}
\newtheorem{remark}[lemma]{Remark}

\newif{\ifdetails}
\detailsfalse
\newif{\ifoverfulboxhacks}
\overfulboxhackstrue
\ifoverfulboxhacks
\newcommand{\OBHnotag}{\notag}
\newenvironment{equationaligned}{\align}{\endalign}
\else
\newcommand{\OBHnotag}{}
\newenvironment{equationaligned}{\equation\aligned}{\endaligned\endequation}
\fi

\allowdisplaybreaks

\title[Output sum of transducers]{Output sum of transducers: Limiting distribution and periodic fluctuation}

\author{Clemens Heuberger}
\address{Institut f\"ur Mathematik, Alpen-Adria-Universit\"at Klagenfurt,
  Universit\"atsstra\ss e 65--67, 9020 Klagenfurt, Austria}
\email{clemens.heuberger@aau.at}

\author{Sara  Kropf}
\address{Institut f\"ur Mathematik, Alpen-Adria-Universit\"at Klagenfurt,
  Universit\"atsstra\ss e 65--67, 9020 Klagenfurt, Austria}
\email{sara.kropf@aau.at}
\thanks{The first two authors are supported by the Austrian Science Fund (FWF):
  P~24644-N26.}

\author{Helmut Prodinger}
\address{Department of Mathematical Sciences, Stellenbosch University, 7602 Stellenbosch,
 South Africa}
\email{hproding@sun.ac.za}
\thanks{The third author
  was supported by an incentive grant of the NRF of South Africa.\\ 
  Parts of the
  article were written while Helmut Prodinger was a visitor at
  Alpen-Adria-Universit\"at Klagenfurt and while Sara Kropf was a visitor at
  Stellenbosch University, respectively.}
\thanks{An extended abstract with less general
  Theorems~\ref{thm:asydist}, \ref{thm:fourier} and \ref{thm:recursion} and without proofs appears as~\cite{Heuberger-Kropf-Prodinger:2014:asymp}.}

\begin{document}
\begin{abstract}
As a generalization of the sum of digits function and other digital sequences,
sequences defined as the sum of the output of a
transducer
are asymptotically analyzed. The input of the transducer is a random integer in $[0, N)$. Analogues in higher
dimensions are also considered. Sequences defined by a certain class of
recursions can be written in this framework.

Depending on properties of the transducer, the
main term, the periodic fluctuation and an error term of the
expected value and the variance of this sequence are established. The periodic fluctuation of the expected
value is H\"older continuous and, in many cases, nowhere differentiable. A general formula for
the Fourier coefficients of this periodic function is derived. Furthermore, it
turns out that the sequence is asymptotically normally
distributed for many transducers. As an example, the abelian
complexity function of the paperfolding sequence is analyzed. This sequence has recently been
studied by Madill and Rampersad.
\end{abstract}

\maketitle

\section{Introduction}
\label{sec:in}
Over the last decades, asymptotic properties of digital sequences have been
studied by many authors. The simplest example is the \begin{math}q\end{math}-ary sum of digits,
 see Delange~\cite{Delange:1975:chiffres}. This has been generalized to
various other number systems
(cf.~\cite{Kirschenhofer:1983:subbl}, \cite{Kirschenhofer-Prodinger:1984},
\cite{Thuswaldner:1999},
\cite{Grabner-Thuswaldner:2000:sum-of-digits-negative},
\cite{Barat-Grabner:2001:distr},
\cite{Grabner-Heuberger-Prodinger:2003:subbl},
\cite{Grabner-Heuberger-Prodinger:2004:distr-results-pairs},
\cite{Heuberger-Muir:2007:minim-weigh}, \cite{Heuberger-Kropf:2013:analy}). Similar results have been
obtained for other digital sequences (cf.~\cite{Cateland:digital-seq} and
\cite{Bassily-Katai:1995:distr}). Frequently observed phenomena in the
asymptotic analysis of these sequences include periodic fluctuations in the
second order term and asymptotic normality (see
also~\cite{Drmota-Grabner:2010}).

The purpose of this article is to use finite state machines as a uniform
framework to derive such asymptotic results. The results mentioned above will
follow as corollaries from our main results, see the end of the introduction
for more details. As an example of a new result fitting into this framework, we
study the abelian complexity function of the paperfolding sequence (cf.\
\cite{Madill-Rampersad:2013}), see Example~\ref{ex:paperfolding}.

Our main focus lies on transducers: these finite state machines transform input
words to output words using a finite memory (see Section~\ref{sec:pre} for a more precise
definition). In our case, the input is the \begin{math}q\end{math}-ary digit expansion of a random
integer in the interval \begin{math}[0,N)\end{math}. We then asymptotically study the sum of the
output of the transducer for \begin{math}N\to\infty\end{math}.  This is also extended to higher
dimensions.

While some of the examples can easily be formulated by transducers,
other examples are more readily expressed in terms of recursions of the shape
\begin{equation}\label{eq:recursion}a(q^{\kappa}n+\lambda)=a(q^{\kappa_{\lambda}}n+r_{\lambda})+t_{\lambda}\quad\text{
    for }\quad 0\leq\lambda<q^{\kappa}\end{equation}
 with fixed \begin{math}\kappa\end{math}, \begin{math}\kappa_{\lambda}\end{math}, \begin{math}r_{\lambda}\in\integers\end{math},
 \begin{math}t_{\lambda}\in\reals\end{math} and
 \begin{math}\kappa_{\lambda}<\kappa\end{math}. We transform such a recursion into
 a transducer in Theorem~\ref{thm:recursion} in
Section~\ref{sec:rec}.

Several notions abstracting the sum-of-digits and related problems have been
studied. One of them is the notion of completely \begin{math}q\end{math}-additive functions
 \begin{math}a:\naturals_{0}\rightarrow\reals\end{math} with
\begin{equation*}
a(qn+\lambda)=a(n)+a(\lambda)
\end{equation*}
for \begin{math}0\leq \lambda<q\end{math}
(cf.~\cite{Bassily-Katai:1995:distr}). These have been generalized to digital
sequences as defined
in~\cite{Allouche-Shallit:2003:autom,Cateland:digital-seq}: A sequence \begin{math}a(n)\end{math}
is a digital sequence if it can be represented as a sum $\sum_w f(w)$ where $f$ is a given
function and $w$ runs over all windows of a fixed length \begin{math}\kappa\end{math}
of the \begin{math}q\end{math}-ary digit representation of \begin{math}n\end{math}. These digital sequences can easily be formulated by a recursion as
in \eqref{eq:recursion}.

For a transducer \begin{math}\T\end{math}, let \begin{math}\T(n)\end{math} be the sum of the output labels of \begin{math}\T\end{math} when reading the \begin{math}q\end{math}-ary
expansion of \begin{math}n\end{math}. For a positive integer \begin{math}N\end{math}, we study the behavior
of \begin{math}\T(n)\end{math} for a uniformly chosen random \begin{math}n\end{math} in \begin{math}\{0, \ldots, N-1\}\end{math}. Assuming
suitable connectivity properties of the underlying graph of the transducer, we obtain
the following results.
\begin{itemize}
\item The expected value is given by 
\begin{equation*}
\Expect(\T(n))=e_{\T}\log_{q}N+\Psi_{1}(\log_{q}N)+o(1)
\end{equation*}
 for a constant \begin{math}e_{\T}\end{math} and a periodic, continuous function \begin{math}\Psi_1\end{math} (Theorem~\ref{thm:asydist}).
\item The variance is
  \begin{equation*}
\Var(\T(n))=v_{\T}\log_{q}N-\Psi_{1}^{2}(\log_{q}N)+\Psi_{2}(\log_{q}N)+o(1)\end{equation*}
with constant \begin{math}v_{\T}\end{math} and a periodic, continuous
function \begin{math}\Psi_{2}(x)\end{math} (Theorem~\ref{thm:asydist}). 
\item After suitable renormalization, \begin{math}\T(n)\end{math} is asymptotically normally distributed (Theorem~\ref{thm:asydist}).
\item The Fourier coefficients of \begin{math}\Psi_1\end{math} are given
  explicitly in Theorem~\ref{thm:fourier} and the Fourier series converges
  absolutely and uniformly.
\item The function \begin{math}\Psi_1\end{math} is nowhere differentiable provided that \begin{math}e_{\T}\end{math} is
  not an integer (Theorem~\ref{thm:nondiff}).
\end{itemize}
The exact assumptions for the various results are given in detail in the
respective theorems. Results for higher dimensional input are available for
expectation, variance, normal distribution and Fourier coefficients.

Our theorems are generalizations of the following known results.
\begin{itemize}
\item For the sum of digits of the standard \begin{math}q\end{math}-ary digit representations
  (cf.~\cite{Delange:1975:chiffres}), we obtain an asymptotic normal
  distribution, the Fourier coefficients and the non-differentiability (for
  even\footnote{Our approach in Theorem~\ref{thm:nondiff} requires that the constant $e_{\T}$ of the main term
    of the expected
    value is not an integer. In this case, 
    $e_{\T}=\frac{q-1}{2}$, which is an integer if $q$ is odd.} \begin{math}q\end{math}). The
  error term vanishes, as stated in
  Remark~\ref{rem:vanishingerror}. Therefore, the formula is not only
  asymptotic but also exact. The formulas for the Fourier coefficients
  by Delange~\cite{Delange:1975:chiffres} also follow from our Theorem~\ref{thm:fourier}.
\item The occurrence of subblocks in standard and non-standard digit
  representations is defined by a strongly connected, aperiodic
  transducer. Thus we obtain the expected value, the variance, the limit law
  and the Fourier coefficients
  (cf.~\cite{Kirschenhofer:1983:subbl,Kirschenhofer-Prodinger:1984,Grabner-Heuberger-Prodinger:2003:subbl}
  for the expected value). For one dimensional digit representations, we also
  obtain the  non-differentiability (assuming
  \begin{math}e_{\T}\neq 0,1\end{math}) of the fluctuation in the expectation.
\item The Hamming weight is a special case of the occurrence of
  subblocks. Thus, Theorem~\ref{thm:asydist} is a generalization of the
  results about the width-\begin{math}w\end{math} non-adjacent
  form~\cite{Heuberger-Kropf:2013:analy}, the simple joint
  sparse form \cite{Grabner-Heuberger-Prodinger:2004:distr-results-pairs} and the
  asymmetric joint sparse form~\cite{Heuberger-Kropf:2013:analy}.
\item A transducer defining a completely \begin{math}q\end{math}-additive function consists of
  only one state. Therefore, we obtain an asymptotic normal distribution (as 
  in~\cite{Bassily-Katai:1995:distr}), the
  Fourier coefficients and the non-differentiability (assuming
  \begin{math}e_{\T}\not\in\integers\end{math} and integer output). Here, the
  error term vanishes, too.
\item A digital sequence is defined by a strongly connected, aperiodic
  transducer. Thus, digital sequences are asymptotically normally
  distributed or degenerate. Assuming \begin{math}e_{\T}\not\in\integers\end{math} and
  integer output, the periodic
  fluctuation \begin{math}\Psi_{1}(x)\end{math} is non-differentiable. The Fourier coefficients can be computed
  by Theorem~\ref{thm:fourier}. See also~\cite{Cateland:digital-seq} for results
  on the expected value.
\item Automatic sequences \cite{Allouche-Shallit:2003:autom} are also defined
  by transducers: The output labels of all transitions
  are \begin{math}0\end{math} and the final output labels are as in the
  definition of such sequences. Theorem~\ref{thm:asydist} gives the expected
  value with \begin{math}e_{\T}=0\end{math} (see also~\cite{Peter:2003}) and, depending on the transducer,
  also the variance with \begin{math}v_{\T}=0\end{math}. The Fourier
  coefficients of the periodic
  fluctuation of the expected value are given explicitly in Theorem~\ref{thm:fourier}.
\item In~\cite{Grabner-Thuswaldner:2000:sum-of-digits-negative}, Grabner and
Thuswaldner investigate the sum of digits function for negative bases
\begin{math}s_{-q}(n)\end{math}. They give a transducer to
  compute the function \begin{math}s_{-q}(n)-s_{-q}(-n)\end{math}. Their result about the limit law follows
  directly from our Theorem~\ref{thm:asydist}.
\end{itemize}

As an example of a new result obtained by Theorem~\ref{thm:asydist}, we give an
asymptotic estimate of the abelian complexity function of the paperfolding
sequence in Example~\ref{ex:paperfolding}. In~\cite{Madill-Rampersad:2013},
the authors prove that this sequence satisfies a recursion of
type~(\ref{eq:recursion}). As consequences of Theorem~\ref{thm:asydist}, the expected value  is \begin{math}\sim\frac 8{13}\log_{2}N\end{math},
the variance is \begin{math}\sim\frac{432}{2197}\log_{2}N\end{math} and the sequence is
asymptotically normally distributed.

In the sequel, we discuss the relation of our setting and our results with
the notion of \begin{math}q\end{math}-regular sequences introduced
in~\cite{Allouche-Shallit:2003:autom}.

A sequence is
\begin{math}q\end{math}-regular if it is the first coordinate of a vector \begin{math}\bfv(n)\end{math} and there
exist matrices \begin{math}V_{0}\end{math}, \dots, \begin{math}V_{q-1}\end{math} such that
\begin{equation}\label{eq:q-reg}
\bfv(qn+\eps)=V_{\eps}\bfv(n)\end{equation}
for \begin{math}\eps\in\{0,1,\ldots,q-1\}\end{math}.

The concept of  \begin{math}q\end{math}-regular sequences is more general than
our setting, but a broader variety of asymptotic behavior is observed which
precludes any generalization of our results to general \begin{math}q\end{math}-regular sequences.

While \begin{math}\T(n)\end{math} is a \begin{math}q\end{math}-regular
sequence for any transducer \begin{math}\T\end{math} (see Remark~\ref{rem:q-reg}), the converse
is not necessarily true: Obviously, the sum of the output of a
transducer reading the input  \begin{math}n\end{math} is always
bounded by \begin{math}\bigOh(\log n)\end{math}. However,
the \begin{math}2\end{math}-regular sequence\footnote{Use $\bfv(0)
  =(0,1)^{\top}$ (where ${}^{\top}$ denotes transposition), $V_{0}=\big(
\begin{smallmatrix}
  2&0\\0&1
\end{smallmatrix}\big)
$ and $V_{1}=
\big(\begin{smallmatrix}
  0&1\\0&0
\end{smallmatrix}\big)
$.}
\begin{equation*}
a(n)=\begin{cases}n&\text{if } n\text{ is a power of }2,\\0&\text{otherwise}\end{cases}
\end{equation*}
can clearly not be bounded by \begin{math}\bigOh(\log n)\end{math}.

Asymptotic estimates for \begin{math}q\end{math}-regular sequences are given by
Dumas~\cite{Dumas:2013:joint,Dumas:2014:asymp}. 
By restricting our attention to sequences
defined by transducers, we obtain an asymptotic estimate of the variance,
explicit expressions for the Fourier coefficients of the fluctuation in the
second term of the expected value, non-differentiability of this fluctuation as 
well as a central limit theorem.

Section~\ref{sec:pre} contains all the theorems and the required notions.
 In Section~\ref{sec:moments-limit-distr}, Theorem~\ref{thm:asydist}, 
 formulas for the first and second moment of the output sum of a transducer and
 its limiting distribution are presented. In Theorem~\ref{thm:fourier} in
Section~\ref{sec:fou}, the Fourier coefficients of the periodic
fluctuation \begin{math}\Psi_{1}(x)\end{math} of the expected value are stated. We discuss the non-differentiability of \begin{math}\Psi_{1}(x)\end{math} in
Theorem~\ref{thm:nondiff} in Section~\ref{sec:nondiff}.

Section~\ref{sec:rec} deals with sequences satisfying the
recursion~\eqref{eq:recursion} and higher dimensional analogues. We construct
a transducer computing this sequence in Theorem~\ref{thm:recursion}. Thus, from
Theorem~\ref{thm:asydist}, the expected value, the variance and the limit
distribution follow in many cases. 

This construction and the computations for the constants
\begin{math}e_{\T}\end{math}, \begin{math}v_{\T}\end{math} and the Fourier
coefficients can be done algorithmically by the
mathematical software system
Sage~\cite{Stein-others:2014:sage-mathem-6.4.1}: The general framework is
included in Sage version 6.4.1 using its
   finite state machine package described
  in~\cite{Heuberger-Krenn-Kropf:ta:finit-state}. The code for the
  Fourier coefficients and the construction from a recursion is submitted for
  inclusion in future versions of Sage, see
  \url{http://trac.sagemath.org/17222} and \url{http://trac.sagemath.org/17221}, respectively.

In Sections~\ref{sec:distribution} to~\ref{sec:recursion-proof}, we
give the proofs of all the theorems from Section~\ref{sec:pre}.

\section{Results}\label{sec:pre}

This section starts with the definition of some notions about the connectivity of a transducer. Then we
will state the theorems about the moments and the limiting distribution, the
Fourier coefficients, the non-differentiability, and the construction of a
transducer computing a sequence given by a recursion as in~\eqref{eq:recursion}. 

\subsection{Notions}
We consider complete, deterministic and subsequential
transducers (cf.~\cite[Chapter 1]{Berthe-Rigo:2010:combin}). In our case,  the
input alphabet is \begin{math}\{0,\ldots,q-1\}^{d}\end{math}
for a positive integer \begin{math}d\end{math} and the output alphabet
\begin{math}\reals\end{math}. A transducer is said
to be
\emph{deterministic} and \emph{complete} if for every state and every digit of the input alphabet, there
is exactly one transition starting in this state with this input label. A
\emph{subsequential} transducer \begin{math}\T\end{math} (cf.~\cite{Schuetzenberger:1977}) is defined to be a finite deterministic automaton with one
initial state, an output label for
every transition and a final output label for every state. 

\begin{figure}
  \centering
  \begin{tikzpicture}[auto, initial text=, >=latex, accepting text=, accepting/.style=accepting by arrow, accepting distance=5ex, every state/.style={minimum
      size=1.3em}]
    \node[state, initial] (v0) at (0.000000,
    0.000000) {};
    \path[->] (v0.270.00) edge node[rotate=450.00, anchor=south] {$0$} ++(270.00:5ex);
    \node[state] (v1) at (3.000000, 0.000000) {};
    \path[->] (v1.270.00) edge node[rotate=450.00, anchor=south] {$1$} ++(270.00:5ex);
    \node[state] (v2) at (6.000000, 0.000000) {};
    \path[->] (v2.270.00) edge node[rotate=450.00, anchor=south] {$1$} ++(270.00:5ex);
    \path[->] (v0.10.00) edge node[rotate=0.00, anchor=south] {$1\mid 0$} (v1.170.00);
    \path[->] (v1.10.00) edge node[rotate=0.00, anchor=south] {$1\mid 1$} (v2.170.00);
    \path[->] (v0) edge[loop above] node {$0\mid 0$} ();
    \path[->] (v2.190.00) edge node[rotate=360.00, anchor=north] {$0\mid 0$} (v1.350.00);
    \path[->] (v2) edge[loop above] node {$1\mid 0$} ();
    \path[->] (v1.190.00) edge node[rotate=360.00, anchor=north] {$0\mid 1$} (v0.350.00);
  \end{tikzpicture}
  \caption{Transducer computing the Hamming weight of the non-adjacent form.}
  \label{fig:naf}
\end{figure}
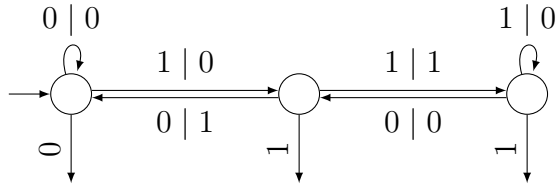

Figure~\ref{fig:naf} presents an example of a complete, deterministic, subsequential
transducer. The label of a transition with input \begin{math}\eps\end{math} and output \begin{math}\delta\end{math} is
written as \begin{math}\eps\mid\delta\end{math}.

The input of the
transducer is the standard \begin{math}q\end{math}-ary joint digit
representation of an integer
vector \begin{math}\bfn\in\naturals_{0}^{d}\end{math}, i.e.\ the standard \begin{math}q\end{math}-ary digit representation at each
coordinate of the vector \begin{math}\bfn\end{math}. The input is read from
right (least significant digit) to left (most significant digit), without
leading zeros.
Then the
output of the transducer is the sequence of the outputs of the transitions along
the unique path starting in the initial state with the given input  and the final output of the last state of this path. The element
\begin{math}\T(\bfn)\end{math} of the sequence defined by the transducer \begin{math}\T\end{math} is the sum of this
output sequence.

Using final output labels is convenient for our purposes. Clearly, it would
also be possible to model the final output labels by using an ``end-of-input'' marker
and additional transitions. In the context of digital expansions, the behavior
can usually also be obtained by reading a sufficient number of leading
zeros. But the approach using final outputs is more general as it is not
required that the final outputs are compatible with the output generated by
leading zeros.

For the various results, different properties of the complete,
deterministic, subsequential transducer and its underlying digraph are needed. All states
of the underlying digraph are
assumed to be
accessible from the initial state. Contracting each strongly connected component of
the underlying digraph gives an acyclic digraph, the so-called condensation. A strongly connected
component is said to be \emph{final strongly connected} if it corresponds to a
leaf (i.e., a vertex with outdegree \begin{math}0\end{math})
in the condensation. Let \begin{math}c\end{math} be the number of final strongly connected components. We call a transducer or a
digraph \emph{finally connected} if \begin{math}c=1\end{math}.

For the asymptotic expressions, only the final strongly connected components
are important. All other strongly connected components only influence the
error term. Thus, we are not interested in the periodicity of the whole
underlying digraph, but in the periodicity of the final strongly connected
components. The \emph{period} of a digraph is defined as the greatest common divisor
of all lengths of directed cycles of the digraph. For \begin{math}j=1,\ldots, c\end{math}, let \begin{math}p_{j}\end{math}
be the period of the final strongly connected
component \begin{math}C_{j}\end{math}. Define the \emph{final
  period} of the digraph as
\begin{equation*}
p=\lcm\{p_{j}\mid j=1,\ldots,c\}.
\end{equation*}
We call a digraph
\emph{finally aperiodic} if \begin{math}p=1\end{math}. If the underlying
digraph is strongly connected, its final period is equal to its period.

For proving the non-differentiability of the fluctuation, we not only need a
finally aperiodic, finally connected digraph (\begin{math}p=c=1\end{math}), but also a reset sequence. A \emph{reset
sequence} is an input sequence such that starting at any state and reading this
sequence leads to a specific state \begin{math}s\end{math}. If the transducer is not finally
aperiodic and finally connected, then there cannot exist a reset sequence.

\subsection{Moments and Limiting Distribution}\label{sec:moments-limit-distr}
This section contains the theorem about the moments of the output sum $\T(\bfn)$
and the limiting distribution. Further results about the periodic fluctuation
can be found in Theorems~\ref{thm:fourier} and~\ref{thm:nondiff}.

As probability space, we
use \begin{math}\Omega_{N}=\{0,1,\ldots,N-1\}^{d}\end{math} endowed with the
equidistribution measure.

Denote by $\Phi_{\mu,\sigma^{2}}$ the cumulative distribution function of
the normal distribution with mean $\mu$ and variance $\sigma^{2}\neq 0$. Thus, 
\begin{equation*}
\Phi_{\mu,\sigma^{2}}(x)=\frac{1}{\sigma\sqrt{2\pi}}\int_{-\infty}^{x}\exp\Big({-\frac{1}{2}\Big(\frac{y-\mu}{\sigma}\Big)^{2}}\Big)\,dy.
\end{equation*}

\begin{theorem}\label{thm:asydist}
Let \begin{math}d\geq 1\end{math}, \begin{math}\T\end{math} be a complete, deterministic, subsequential transducer with input alphabet
\begin{math}\{0,1,\ldots,q-1\}^{d}\end{math}, 
output alphabet \begin{math}\reals\end{math}, final
period \begin{math}p\end{math}, and $c$ final components.

Then
\begin{math}\T(\bfn)\end{math} has the expected value
\begin{equation}\Expect(\T(\bfn))=e_{\T}\log_{q}N+\Psi_{1}(\log_{q}N)+\bigOh(N^{-\xi}\log
N)\label{eq:expected}\end{equation}
where the constants \begin{math}e_{\T}\end{math} and \begin{math}\xi>0\end{math} are given
in~\eqref{eq:const-of-thm} in Section~\ref{sec:prop-trans-matr}  and \begin{math}\Psi_{1}(x)\end{math} is a \begin{math}p\end{math}-periodic,
H\"{o}lder continuous function.

If all $b_j$ given in \eqref{eq:const-of-thm} are positive,
the distribution function of $\T(\bfn)$ can be approximated by a mixture of $c$
Gaussian distributions with
weights $\lambda_{j}$, means $a_{j}\log_{q}N$ and variances $b_{j}\log_{q}N$
for some constants $a_{j}$ and $\lambda_{j}> 0$ with
$\sum_{j=1}^{c}\lambda_{j}=1$, given in~\eqref{eq:const-of-thm}. In particular,
\begin{equation*}
\mathbb P\biggl(\frac{\T(\bfn)}{\sqrt{\log_{q}N}}\le x\biggr)=\sum_{j=1}^{c}
\lambda_{j}\Phi_{a_{j}\sqrt{\log_{q}N}, b_{j}}(x)+\bigOh\bigl(\log^{-\frac12}N\bigr)
\end{equation*}
for all $x\in\reals$.

If all $a_{j}$ are equal, then \begin{math}\T(\bfn)\end{math} has the variance
\begin{equation}\label{eq:var-good}
\Var(\T(\bfn))=v_{\T}\log_{q}N-\Psi_{1}^{2}(\log_{q}N)+\Psi_{2}(\log_{q}N)+\bigOh(N^{-\xi}\log^{2}N)\end{equation}
with constant \begin{math}v_{\T}\in\reals\end{math} (given in~\eqref{eq:const-of-thm}) and a \begin{math}p\end{math}-periodic, continuous
function \begin{math}\Psi_{2}(x)\end{math}. Otherwise, the variance is
$\Var(\T(\bfn))=\Theta(\log^{2} N)$.

If all $a_{j}$ are equal, $\T(\bfn)$ converges in distribution to a mixture of
Gaussian (or degenerate) distributions with means $0$ and variances $b_j$,
weighted by $\lambda_j$. In particular, if all $b_j>0$,
\begin{equation*}\mathbb
P\biggl(\frac{\T(\bfn)-\Expect(\T(\bfn))}{\sqrt{\log_{q}N}}\le x\biggr)=\sum_{j=1}^{c}\lambda_{j}\Phi_{0,b_{j}}(x)+\bigOh\bigl(\log^{-\frac12}N\bigr)\end{equation*}
holds for all \begin{math}x\in\reals\end{math}.

If furthermore $c=1$ and \begin{math}v_{\T}\neq0\end{math}, then \begin{math}\T(\bfn)\end{math} is asymptotically
normally distributed.
\end{theorem}

We give the proof of this theorem in
Section~\ref{sec:distribution}.

\begin{remark}\label{remark:discrete-limit}
  The assumption that $b_j>0$ is essential for obtaining uniform convergence of
  the distribution function and the speed of convergence in particular. To see this,
  consider the transducer in Figure~\ref{fig:discrete-limit}. It is easily seen
  that $\T(n)=(-1)^n$. For even $N$, the distribution function of $\T(n)/\sqrt{\log_2 N}$ is
  given by
  \begin{equation*}
    \mathbb{P}\biggl(\frac{\T(n)}{\sqrt{\log_2 N}}\le x\biggr)=
    \begin{cases}
      0& \text{if }x<-1/\sqrt{\log_2 N},\\
      1/2& \text{if }-1/\sqrt{\log_2 N}\le x<1/\sqrt{\log_2 N},\\
      1& \text{if }1/\sqrt{\log_2 N}\le x,
    \end{cases}
  \end{equation*}
  which does not converge uniformly.

  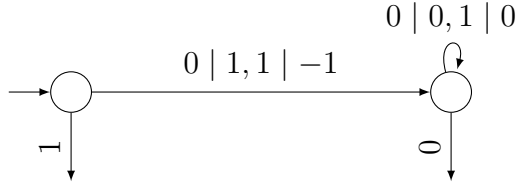
\begin{figure}
    \centering
    \begin{tikzpicture}[auto, initial text=, >=latex, accepting text=,
      accepting/.style=accepting by arrow, accepting distance=5ex, every state/.style={minimum
    size=1.3em}]
      \node[state, initial] (v0) at (0.000000, 0.000000) {};
      \path[->] (v0.270.00) edge node[rotate=450.00, anchor=south] {$1$} ++(270.00:5ex);
      \node[state] (v1) at (5.000000, 0.000000) {};
      \path[->] (v1.270.00) edge node[rotate=450.00, anchor=south] {$0$} ++(270.00:5ex);
      \path[->] (v0) edge node[rotate=0.00, anchor=south] {$0\mid 1, 1\mid -1$} (v1);
      \path[->] (v1) edge[loop above] node {$0\mid 0, 1\mid 0$} ();
    \end{tikzpicture}
    \caption{Transducer for Remark~\ref{remark:discrete-limit}.}
    \label{fig:discrete-limit}
  \end{figure}
\end{remark}

\subsection{Eigenvalues and Eigenvectors of the Transition Matrix}\label{sec:prop-trans-matr}

For the constants in Theorem~\ref{thm:asydist} and the Fourier coefficients in
Theorem~\ref{thm:fourier}, we need the notion of a transition matrix of the
transducer and properties of its eigenvalues and eigenvectors. 

We label the states of the transducer
with contiguous positive integers starting with $1$. We denote the indicator vector
of the initial state by $\bfe_{1}$.

\begin{definition}
Let $t\in\mathbb R$ be in a neighborhood of $0$.

 The transition matrix \begin{math}M_{\bfeps}\end{math} for
\begin{math}\bfeps\in\{0,\ldots,q-1\}^{d}\end{math} is the matrix whose \begin{math}(s_{1},s_{2})\end{math}-th entry
 is \begin{math}e^{it\delta}\end{math} if there is a transition from state
\begin{math}s_{1}\end{math} to state \begin{math}s_{2}\end{math} with input label \begin{math}\bfeps\end{math} and output label \begin{math}\delta\end{math},
and \begin{math}0\end{math} otherwise. 

Let \begin{math}M\end{math} be the sum of all these transition matrices.
\end{definition}

\begin{lemma}\label{lem:eigenvalues-M}There are differentiable functions
  $\mu_{j}(t)$ in a neighborhood of $t=0$ for $j=1$, \dots, $c$ such that
  the dominant eigenvalues of $M$ are
  $\mu_{j}(t)\exp(\frac{2\pi il}{p})$ in this neighborhood of $t=0$ for some
  of the $l\in\calP=\{k\in\integers\mid
  -p/2<k\leq p/2\}$. For each of these dominant eigenvalues, the algebraic and
  geometric multiplicities coincide. For $t=0$, $\mu_{j}(0)=q^{d}$.
\end{lemma}
The proof of this lemma is given in Section~\ref{sec:distribution}.

Let $l\in\mathbb Z$. Consider the (not necessarily orthogonal) projection onto the direct
sum of the left
eigenspaces of $M$ corresponding to the eigenvalues $\mu_{j}(t)\exp(\frac{2\pi
il}{p})$ for $j=1,\ldots, c$ such that the kernel is the direct sum of the
remaining generalized left eigenspaces.
Let $\bfw_{l}^{\top}(t)$ be the image of $\bfe_{1}^{\top}$ under this
projection, where ${}^{\top}$ denotes transposition. The definition of
$\bfw_{l}^{\top}(t)$ only depends on $l$ modulo $p$.

We write $\bfw_{l}^{\top}$ for $\bfw_{l}^{\top}(0)$ and $\bfw_{l}'^{\top}$ for
the
derivative of $\bfw_{l}^{\top}(t)$ at $t=0$.
 Furthermore,
$\bfw_{l}^{\top}$ is either the null vector or a left eigenvector of $M$
corresponding to the eigenvalue $q^{d}\exp(\frac{2\pi il}{p})$.

Let $C_{j}$ be a final component with corresponding indicator
  vector $\boldsymbol{c}_{j}$. Define the constants
  \begin{equation*}
    \lambda_{j}=\bfw_{0}^{\top}\boldsymbol{c}_{j}.
  \end{equation*}
In Section~\ref{sec:prel-from-lin-alg}, we will show that $\lambda_{j}> 0$ and $\sum_{j=1}^{c}\lambda_{j}=1$.

With these definitions, the constants in
Theorem~\ref{thm:asydist} can be expressed as  
\begin{equation}
\begin{aligned}
  a_{j}&=-iq^{-d}\mu_{j}'(0),\\
  e_{\T}&=\sum_{j=1}^{c}\lambda_{j}a_{j},\\
  b_{j}&=\frac{\mu_{j}'(0)^{2}-q^{d}\mu_{j}''(0)}{q^{2d}},\\
  v_{\T}&=\sum_{j=1}^{c}\lambda_{j}b_{j}.\\
\end{aligned}\label{eq:const-of-thm}
\end{equation}
Finally, $\xi>0$ is chosen such that all non-dominant eigenvalues of $M$ have modulus strictly less
than $q^{d-\xi}$ at $t=0$.

These constants can be interpreted as follows: $a_{j}\log_{q}N$ and
$b_{j}\log_{q}N$ are the main terms of the mean and the variance, respectively, of
the output sum of the final component $C_{j}$. These expressions including the
derivatives of the
eigenvalues correspond to the formulas for mean and variance given in
\cite[Theorem~IX.9]{Flajolet-Sedgewick:ta:analy}. The constants $e_{\T}$ and
$v_{\T}$ are convex combinations of the corresponding constants of the final
components $C_{j}$.

The positive weight $\lambda_{j}$ in these convex combinations turns out to be the asymptotic probability of reaching the final component
$C_{j}$. This is connected to the following interpretation of the left eigenvector $\bfw_{0}^{\top}$:
If the final period $p$ is $1$, the entries of $\bfw_{0}^{\top}$ will be shown
to be the asymptotic probabilities
of reaching the corresponding states. This corresponds to the left eigenvector
used in a steady-state analysis. If $p>1$, these probabilities depend on
the length of the input modulo $p$. Then, we will prove that $\bfw_{0}^{\top}$ gives the average 
of these probabilities taken over all residues modulo $p$. These interpretations are justified in Section~\ref{sec:prel-from-lin-alg}. 

\subsection{Fourier Coefficients}\label{sec:fou}
This section contains the formulas for the Fourier coefficients of the
periodic fluctuation \begin{math}\Psi_{1}(x)\end{math}. 
For this purpose, we need the following
definitions. 

Let \begin{math}\chi_{k}=\frac{2\pi ik}{p\log
    q}\end{math} for $k\in\mathbb Z$ and $\bfones$ be a vector whose entries
are all one.

The \begin{math}s\end{math}-th coordinate of the
vector \begin{math}\bfb(\boldsymbol{n})\end{math} is the sum of the output of the
transducer \begin{math}\T\end{math} (including the final output) if starting in
state \begin{math}s\end{math} with input the $q$-ary joint expansion of \begin{math}\boldsymbol{n}\end{math}. In
particular, the first coordinate of $\bfb(\bfn)$ is $\T(\bfn)$, and \begin{math}\bfb(0)\end{math} is the vector of final
outputs. Furthermore, define the vector-valued function $\bfH(z)$ by the
Dirichlet series
\begin{equation}\label{eq:dirichlet-H-definition}
  \bfH(z)=\sum_{\substack{\bfn\geq 0\\\bfn\neq 0}}\bfb(\bfn)\lVert \bfn\rVert_{\infty}^{-z},
\end{equation}
where the inequality in the summation index is considered coordinate-wise and
$\lVert\,\cdot\,\rVert_{\infty}$ is the maximum norm.

\begin{theorem}\label{thm:fourier}
Let \begin{math}\T\end{math} be a subsequential, complete, deterministic transducer. Then
the Fourier coefficients of the \begin{math}p\end{math}-periodic fluctuation \begin{math}\Psi_{1}(x)\end{math} are
\begin{equation}
\begin{aligned}\label{eq:gen-fourier-coeff}
c_{0}&=-\frac{e_{\T}}{d\log q}-i\bfw_{0}'^{\top}\bfones+\frac1d\Res_{z=d}\bfw_{0}^{\top}\bfH(z),\\
c_{k}&=\frac{1}{d+\chi_{k}}\Res_{z=d+\chi_{k}}\bfw_{k}^{\top}\bfH(z)
\end{aligned}
\end{equation}
for \begin{math}k\neq0\end{math}.

The Fourier series $
  \sum_{k\in\mathbb Z}c_{k}\exp(\frac{2\pi ik}{p}x)
$
 converges absolutely and uniformly.

 The
function $\bfw_{k}^{\top}\bfH(z)$ is meromorphic in $\Re z>d-1$. It has a
possible double pole at $z=d$ for $k=0$ and possible simple poles at $z=d+\chi_{k}$ for $k\neq0$.
\end{theorem}

The proof of this theorem is in
Section~\ref{sec:four-coeff-gener}. 

The infinite recursion given in Lemma~\ref{lem:inf-recursion} can be used to numerically evaluate the Dirichlet
series \begin{math}\bfH(z)\end{math} with arbitrary precision
 and to compute its residues
at \begin{math}z=d+\chi_{l}\end{math} (see Lemma~\ref{lem:dirichlet-H}
and~\cite{Grabner-Hwang:2005:digit}). For
$d=1$, the computation of the Fourier coefficients
can be done by the mathematical software system
Sage~\cite{Stein-others:2014:sage-mathem-6.4.1} (using the code submitted at
\url{http://trac.sagemath.org/17222}).

\begin{example}\label{ex:art-6p}The (artificial) transducer in Figure~\ref{fig:art-6p-trans} has two final components
  with periods $2$ and $3$, respectively. Thus the final period is $6$ and the
  function $\Psi_{1}(x)$ is $6$-periodic. The constant $e_{\T}$ of the
  expected value is $\frac{11}{8}$.
  In Figure~\ref{fig:art-6p}, the partial Fourier series with
  $2550$
  Fourier coefficients\footnote{We use $2550$ Fourier coefficients in this plot because the
    period length of the
    next summand of the Fourier series in Figure~\ref{fig:art-6p} is already less than the
    resolution of a standard printer.} is compared with the empirical values of the periodic
  fluctuation $\Psi_{1}$, i.e.,  
  \begin{equation}\label{eq:emp-values}
    \frac{1}{N}\sum_{n<N}\T(n)-\frac{11}{8}\log_{2}N
  \end{equation}
with integers $N$ and $4\leq \log_{2}N\leq 16$.

The computation of these $2550$ Fourier coefficients took less than
$6$~minutes using a standard dual-core PC.
\end{example}
\begin{figure}
    \centering
    \begin{tikzpicture}[auto, initial where=above, initial text=, >=latex,every state/.style={minimum
    size=1.3em}]
      \node[state, initial] (v0) at (0.000000, 2.000000) {};
      \node[state] (v1) at (3.000000, 2.000000) {};
      \node[state] (v2) at (3.000000, 5.000000) {};
      \node[state] (v3) at (-4.500000, 5.000000) {};
      \node[state] (v4) at (-6.000000, 2.000000) {};
      \node[state] (v5) at (-3.000000, 2.000000) {};
      \path[->] (v3) edge node[rotate=-63.43, anchor=south] {$1\mid 1$, $0\mid 3$} (v5);
      \path[->] (v4) edge node[rotate=63.43, anchor=south] {$1\mid 1$, $0\mid 1$} (v3);
      \path[->] (v5) edge node[rotate=360.00, anchor=north] {$1\mid 1$, $0\mid 2$} (v4);
      \path[->] (v0) edge node[anchor=north] {$1\mid 1$} (v5);
      \path[->] (v0) edge node[anchor=south] {$0\mid 1$} (v1);
      \path[->] (v1.95.00) edge node[rotate=90.00, anchor=south] {$0\mid 1$, $1\mid 0$} (v2.265.00);
      \path[->] (v2.-85.00) edge node[rotate=90.00, anchor=north] {$0\mid 2$, $1\mid 2$} (v1.85.00);
    \end{tikzpicture}
    \caption{Transducer of Example~\ref{ex:art-6p}: All states are final with final output $0$.}
    \label{fig:art-6p-trans}
  \end{figure}
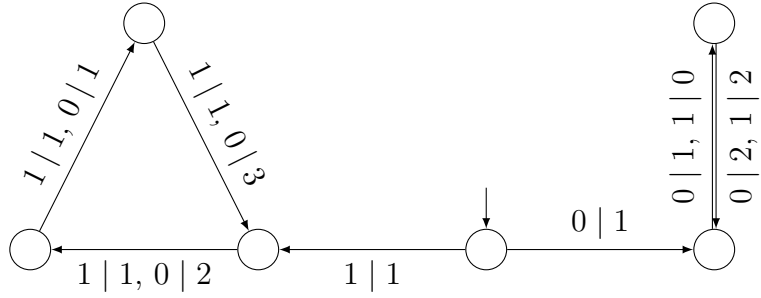
  \begin{figure}
    \centering
    \includegraphics{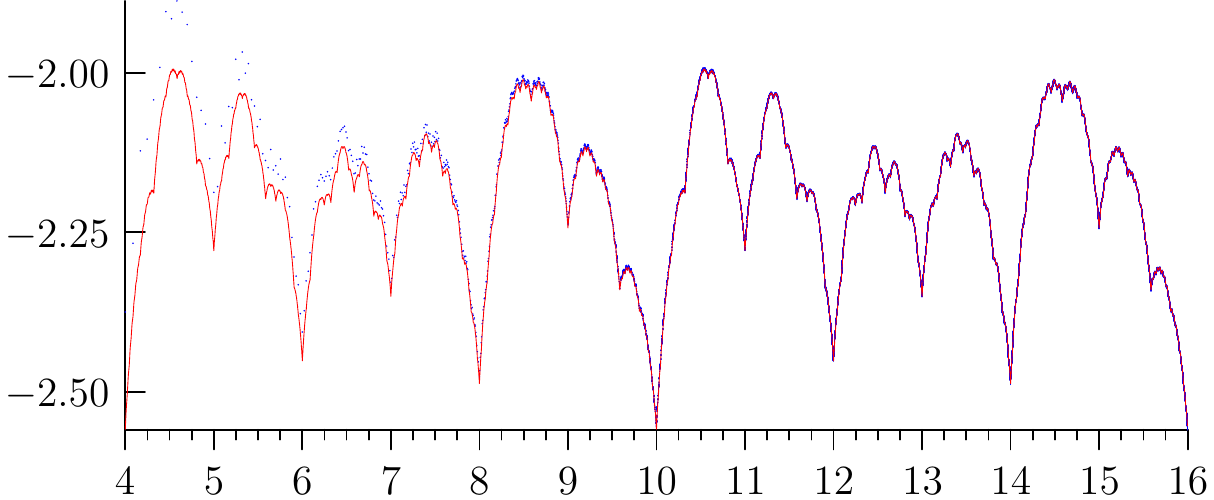}
    \caption{Partial Fourier series compared with the empirical values of the
      function $\Psi_1$ of Example~\ref{ex:art-6p}.}
    \label{fig:art-6p}
  \end{figure}

In Example~\ref{ex:paperfolding}
 we compute the first \begin{math}2550\end{math} Fourier coefficients of the abelian complexity
function of the paperfolding sequence.

As a corollary of Theorem~\ref{thm:fourier}, we obtain the following result which was already proved by Delange~\cite{Delange:1975:chiffres}.
\begin{corollary}\label{cor:delange}
  The Fourier coefficients of the periodic fluctuation
  \begin{equation*}
    \Psi_{1}(\log_{q}N)=\frac 1N\sum_{n<N}s_{q}(n)-\frac{q-1}{2}\log_{q}N
  \end{equation*}
 for the $q$-ary sum-of-digits function $s_{q}(n)$ are
 \begin{equation}\label{eq:delange-fou-coeff}
 \begin{aligned}
   c_{0}&=\frac{q-1}{2\log q}(\log (2\pi) -1)-\frac{q+1}{4} ,\\
   c_{k}&=-\frac{q-1}{\chi_{k}(1+\chi_{k})\log q}\zeta(\chi_{k})
 \end{aligned}
\end{equation}
for $k\neq 0$ and $\chi_{k}=\frac{2\pi ik}{\log q}$ where $\zeta$ denotes the Riemann $\zeta$-function.
\end{corollary}

We prove this corollary in Section~\ref{sec:four-coeff-gener}.

\subsection{Non-differentiability}\label{sec:nondiff}
In this section, we
prove that for certain transducers, the periodic fluctuation \begin{math}\Psi_{1}(x)\end{math}
of the expected value is nowhere
differentiable. 

\begin{theorem}\label{thm:nondiff}
Let \begin{math}d=1\end{math}. Assume that \begin{math}e_{\T}\not\in\integers\end{math} and that the transducer \begin{math}\T\end{math}
has a reset sequence and output alphabet \begin{math}\integers\end{math}. Then the function \begin{math}\Psi_{1}(x)\end{math} is non-differentiable
for any \begin{math}x\in\reals\end{math}.
\end{theorem}

The proof
can be found in Section~\ref{appendix:non-differentiability}.
There, we follow the method presented by
Tenenbaum~\cite{Tenenbaum:1997:non-derivabilite}, see also Grabner and
Thuswaldner~\cite{Grabner-Thuswaldner:2000:sum-of-digits-negative}. 

In~\cite{Tenenbaum:1997:non-derivabilite,Grabner-Thuswaldner:2000:sum-of-digits-negative},
the reset sequence consists only of \begin{math}0\end{math}'s. If working with digit expansions,
it is often possible to choose such a reset sequence. However, in the context
of recursions, this is not always possible, see
Example~\ref{ex:paperfolding}. There the reset sequence
is \begin{math}(00001)\end{math}. 

For a general finally aperiodic, finally connected transducer, the existence of a
reset sequence cannot be guaranteed.

\subsection{Recursions}\label{sec:rec}
In this section, we describe how to reduce a recursion to a transducer computing the given
sequence. All inequalities in this section are considered coordinate-wise.

Let  \begin{math}q\ge 2\end{math}, \begin{math}\kappa\end{math}, \begin{math}\kappa_{\bflambda}\in\integers\end{math}, \begin{math}\bfr_{\bflambda}\in\integers^{d}\end{math},
 \begin{math}t_{\bflambda}\in\reals\end{math} and \begin{math}0\leq
   \kappa_{\bflambda}<\kappa\end{math} for $0\le \bflambda<q^\kappa\bfones$. If $d\geq 2$, then additionally let
 $\bfr_{\bflambda}\geq 0$ for all $\bflambda$.

Consider the sequence \begin{math}a(\bfn)\end{math},
$\bfn\in\naturals_{0}^{d}$, defined
by the recursion
\begin{equation}\label{eq:recursions-d}
a(q^{\kappa}\bfn+\bflambda)=a(q^{\kappa_{\bflambda}}\bfn+\bfr_{\bflambda})+t_{\bflambda}\quad\text{
  for }\quad 0\leq\bflambda<q^{\kappa}\bfones
\end{equation}
 and for all integer vectors $\bfn$ such that the arguments on both sides
 are non-negative.  Furthermore, initial values $a(\bfn)$ for $\bfn\in\calI$ have to be given
 for a suitable finite set $\calI\subset \naturals_0^d$.

It must be ensured that the recursion \eqref{eq:recursions-d} does not lead to
conflicts and that the set of $\calI$ is appropriate. Additionally, we
require  that $\calI$ is minimal (with respect to
inclusion). In that case, we say that the recursion is well-posed.

In Section~\ref{sec:recursion-proof}, we construct a subsequential,
complete, deterministic transducer $\T$ (also when the recursion is not
well-posed) reading the $q$-ary joint expansion of integer vectors without
leading zeros.
We will define a distinguished subset of its states, called
\emph{simple states}. Furthermore, disjoint classes $F_1$, \ldots, $F_K$ of
integer vectors will be defined.

\begin{theorem} \label{thm:recursion}
  The recursion \eqref{eq:recursions-d} is well-posed if and only if
  \begin{enumerate}
  \item for each cycle consisting of simple states with transitions with zero
    input label, the sum of its output
    transitions vanishes and
  \item the set $\calI$ consists of one representative of each $F_j$,  $1\le j\le K$.
  \end{enumerate}
  In that case, the sum of the output of $\T$ is the sequence $a$,
  i.e., \begin{math}\T(\bfn)=a(\bfn)\end{math} for all $\bfn\ge 0$.
\end{theorem}

The proof of this theorem is in Section~\ref{sec:recursion-proof}.
Combining this result with Theorem~\ref{thm:asydist} yields an asymptotic analysis of
the  sequence \begin{math}a(\bfn)\end{math}, as in Example~\ref{ex:paperfolding}. Moreover,
this asymptotic analysis can be performed algorithmically in Sage for $d=1$
(using the code submitted at \url{http://trac.sagemath.org/17221}).
A combinatorial description of the sets $F_i$ involving an auxiliary transducer
is given in Remark~\ref{remark:combinatorial-characterization-outdegree-zero-recursion-digraph}.

\begin{remark}
  For $d\ge 2$, and $r_{\bflambda}\not\ge 0$, the sequence cannot be computed by
  a finite transducer: For every $j\ge 0$, there are non-zero integer vectors
  $\bfn\ge 0$, $\bfn'\ge 0$ with $\bfn\equiv \bfn' \pmod{q^j}$---i.e., a finite
  deterministic transducer
  cannot distinguish between $\bfn$ and $\bfn'$---such that
  the recursion \eqref{eq:recursions-d} can be applied for the argument $q^\kappa \bfn
  +\bflambda$ but cannot be applied for $q^\kappa \bfn'+\bflambda$.

  This problem does not arise in the case of dimension $d=1$: if the end of the
  input is not yet reached (this is something the transducer knows), there is a
  guaranteed forthcoming digit $\ge 1$ (instead of $\neq 0$ in the higher
  dimensional case). This information is enough to decide
  whether the recursion can be used.
\end{remark}

\begin{remark}
  Suppose that the given sequence is defined for $\bfn\ge \bfn_0$ for some
  constant $\bfn_0$. Then the sequence $b(\bfn)=a(\bfn+\bfn_0)$ fulfills
  \eqref{eq:recursions-d} with $\kappa_{\bflambda}$, $\bfr_{\bflambda}$ and
  $t_{\bflambda}$ replaced by $\kappa_{\bfmu}$,
  $q^{\kappa_{\bfmu}}\bfs+\bfr_{\bfmu}-\bfn_0$ and $t_{\bfmu}$, respectively,
  where $\bfn_0+\bflambda=q^{\kappa}\bfs + \bfmu$ for $0\le \bfmu<q^\kappa\bfones$.
  Then Theorem~\ref{thm:recursion} can be applied.
\end{remark}

  \begin{figure}
    \centering
    \begin{tikzpicture}[auto, initial text=, >=latex, every state/.style={minimum
      size=1.3em}]
      \node[state, initial, initial where=below] (v0) at (0.000000, 0.000000) {};
      \node[state] (v1) at  (-3.500000, 1.5000000)  {};
      \node[state] (v2) at  (3.000000, 1.5000000)   {};
      \node[state] (v3) at  (-5.00000, 3.000000)    {};
      \node[state] (v4) at  (-2.00000, 3.000000)    {};
      \node[state] (v5) at  (1.5000000, 3.000000)   {};
      \node[state] (v6) at  (4.5000000, 3.000000)   {};
      \node[state] (v7) at  (-5.7500000, 5.5000000) {};
      \node[state] (v8) at  (-4.2500000, 5.5000000) {};
      \node[state] (v9) at  (-1.2500000, 5.5000000) {};
      \node[state] (v10) at (-2.7500000, 5.5000000) {};
      \node[state] (v11) at (0.7500000, 5.5000000)  {};
      \node[state] (v12) at (2.2500000, 5.5000000)  {};
      \node[state] (v13) at (5.2500000, 5.5000000)  {};
      \node[state] (v14) at (3.7500000, 5.5000000)  {};
      \node[state] (v15) at (-2.00000, 9.000000)    {};
      \node[state] (v16) at (-3.50000, 9.000000)    {};
      \node[state] (v17) at (-5.00000, 9.000000)    {};

      \path[->] (v0) edge node[sloped, anchor=north] {$0\mid 0$} (v1);
      \path[->] (v0) edge node[sloped, anchor=north] {$1\mid 0$} (v2);
      \path[->] (v1) edge node[sloped, anchor=north] {$0\mid 0$} (v3);
      \path[->] (v1) edge node[sloped, anchor=north] {$1\mid 0$} (v4);
      \path[->] (v3) edge node[sloped, anchor=north] {$0\mid 0$} (v7);
      \path[->] (v3) edge node[sloped, anchor=north] {$1\mid 0$} (v8);
      \path[->] (v7) edge[loop above] node {$0\mid 0$} ();
      \path[->] (v7) edge node[sloped, anchor=south] {$1\mid 0$} (v8);
      \path[->] (v2) edge node[sloped, anchor=north, pos=0.7] {$0\mid 0$} (v5);
      \path[->] (v2) edge node[sloped, anchor=north, pos=0.7] {$1\mid 0$} (v6);
      \path[->] (v5) edge node[sloped, anchor=north, pos=0.4] {$0\mid 0$} (v11);
      \path[->] (v5.79.30) edge node[sloped, anchor=south, pos=0.6] {$1\mid 0$} (v12.247.30);
      \path[->] (v11) edge[bend right=50] node[sloped, anchor=north] {$1\mid 2$} (v2);
      \path[->] (v11) edge[in=-120, out=-150, loop] node[anchor=north east] {$0\mid 0$} ();
      \path[->] (v15) edge node[sloped, anchor=south, style=near start] {$0\mid 0$} (v4);
      \path[->] (v15) edge node[sloped, anchor=south] {$1\mid 0$} (v16);
      \path[->] (v4) edge node[sloped, anchor=north] {$0\mid 0$} (v9);
      \path[->] (v4) edge node[sloped, anchor=north] {$1\mid 0$} (v10);
      \path[->] (v9) edge node[sloped, anchor=north, pos=0.4] {$0\mid 1$} (v11);
      \path[->] (v9) edge[bend left] node[sloped, anchor=south, pos=0.4] {$1\mid 1$} (v12);
      \path[->] (v6.79.30) edge node[sloped, anchor=south, pos=0.6] {$0\mid 0$} (v13.247.30);
      \path[->] (v6) edge node[sloped, anchor=north, pos=0.4] {$1\mid 0$} (v14);
      \path[->] (v13) edge[bend left=50] node[sloped, anchor=north] {$0\mid 2$} (v2);
      \path[->] (v13.-100.70) edge node[sloped, anchor=north, pos=0.6] {$1\mid 2$} (v6.69.30);
      \path[->] (v16) edge node[sloped, anchor=north] {$0\mid 0$} (v8);
      \path[->] (v16) edge node[sloped, anchor=south] {$1\mid 0$} (v17);
      \path[->] (v8) edge[bend left] node[sloped, anchor=south, pos=0.4] {$0\mid 0$} (v9);
      \path[->] (v8) edge node[sloped, anchor=north] {$1\mid 0$} (v10);
      \path[->] (v12) edge node[sloped, anchor=north, pos=0.6] {$1\mid 2$} (v2);
      \path[->] (v12.-100.70) edge node[sloped, anchor=north, pos=0.6] {$0\mid 2$} (v5.69.30);
      \path[->] (v10) edge[bend left=60] node[sloped, anchor=south, pos=0.5] {$0\mid 1$} (v13);
      \path[->] (v10) edge[bend left=45] node[sloped, anchor=south, pos=0.6] {$1\mid 1$} (v14);
      \path[->] (v14) edge node[sloped, anchor=north, pos=0.6] {$0\mid 2$} (v2);
      \path[->] (v14) edge[in=0, out=90, relative=false] node[sloped, anchor=south, style=near end] {$1\mid 1$} (v15);
      \path[->] (v17) edge node[sloped, anchor=north] {$0\mid 0$} (v8);
      \path[->] (v17) edge[loop above] node {$1\mid 0$} ();
    \end{tikzpicture}
    \caption{Transducer computing the abelian complexity function $\rho(n)$ of
      the paperfolding sequence. For simplicity, the final output labels are omitted.}
    \label{fig:paperfoldingTrans}
  \end{figure}
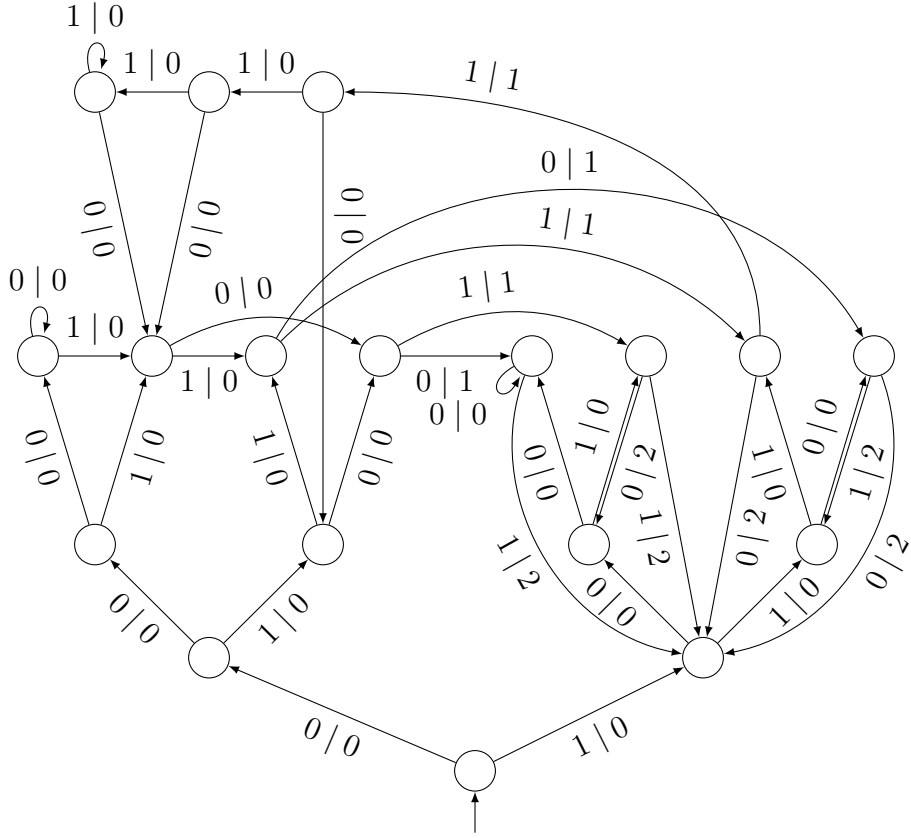
\begin{example}\label{ex:paperfolding}
Consider the abelian complexity function \begin{math}\rho(n)\end{math} of the paperfolding
sequence. The paperfolding sequence is obtained by repeatedly folding a strip
of paper in half in the same direction. Then we open the strip and encode a
right turn by \begin{math}1\end{math} and a left turn by \begin{math}0\end{math}. The abelian complexity
function \begin{math}\rho(n)\end{math} gives the number of abelian equivalence classes of
subwords of length \begin{math}n\end{math} of the paperfolding sequence. Two subwords of
length \begin{math}n\end{math} are equivalent if they are
permutations of each other. In~\cite{Madill-Rampersad:2013}, the authors prove
that this sequence satisfies the recursion
\begingroup\allowdisplaybreaks
\begin{align*}
\rho(4n)&=\rho(2n),\\
\rho(4n+2)&=\rho(2n+1)+1,\\
\rho(16n+1)&=\rho(8n+1),\\
\rho(16n+3)&=\rho(2n+1)+2,\\
\rho(16n+5)&=\rho(4n+1)+2,\\
\rho(16n+7)&=\rho(2n+1)+2,\\
\rho(16n+9)&=\rho(2n+1)+2,\\
\rho(16n+11)&=\rho(4n+3)+2,\\
\rho(16n+13)&=\rho(2n+1)+2,\\
\rho(16n+15)&=\rho(2n+2)+1
\end{align*}\endgroup
with \begin{math}\rho(1)=2\end{math} and \begin{math}\rho(0)=0\end{math}. The constructed
transducer is shown in Figure~\ref{fig:paperfoldingTrans}. For simplicity, we do not
state the final output labels in this figure. The expected value and the variance are
\begin{align*}
\Expect(\rho(n))&=\frac8{13}\log_{2}N+\Psi_{1}(\log_{2}N)+\bigOh(N^{-\xi}\log N),\\
\Var(\rho(n))&=\frac{432}{2197}\log_{2}N-\Psi_{1}^{2}(\log_{2}N)+\Psi_{2}(\log_{2}N)+\bigOh(N^{-\xi}\log^{2}N)
\end{align*}
with $0<\xi<0.5604267891$, as the second largest eigenvalues of the
transition matrix are $-0.7718445063 \pm 1.1151425080\,i$. The
sequence \begin{math}\rho(n)\end{math} is asymptotically normally distributed. The functions
\begin{math}\Psi_{1}(x)\end{math} and \begin{math}\Psi_{2}(x)\end{math}
are \begin{math}1\end{math}-periodic and continuous.  The reset sequence of the transducer
is \begin{math}(00001)\end{math} (reading from right to left). The function \begin{math}\Psi_{1}(x)\end{math}
is nowhere differentiable and its Fourier series converges absolutely and uniformly. The first \begin{math}24\end{math} Fourier
 coefficients of \begin{math}\Psi_{1}(x)\end{math} are listed in
 Table~\ref{tab:paperfoldingCoeff}. In Figure~\ref{fig:paperfoldingFourier}, the trigonometric polynomial formed
with the first $2550$ Fourier coefficients is compared with the empirical
values of the function \begin{math}\Psi_{1}(x)\end{math} (see~\eqref{eq:emp-values}).
\end{example}
\begin{table}
  \centering\footnotesize
  \begin{equation*}\begin{array}{c|l|c|l}
    l&\hfill c_{l}\hfill&l&\hfill c_{l}\hfill\\\hline
    0  & \phantom{-}1.5308151288                 &12 & -0.0002297481+0.0009687657\,i\\
    1  & -0.0162585750+0.0478637218\,i&           13 & \phantom{-}0.0006425378+0.0006516706\,i\\
    2  & \phantom{-}0.0054521982+0.0075023586\,i& 14 & \phantom{-}0.0000413217-0.0003867709\,i\\
    3  & -0.0028294724+0.0086495903\,i&           15 & -0.0005632948-0.0001843541\,i\\
    4  & \phantom{-}0.0036818110+0.0021908312\,i& 16 & \phantom{-}0.0009051717-0.0000476354\,i\\
    5  & -0.0028244495+0.0014519078\,i&           17 & -0.0004621780-0.0000594551\,i\\
    6  & -0.0008962222+0.0030512180\,i&           18 & -0.0000127264-0.0003100798\,i\\
    7  & \phantom{-}0.0015033904+0.0013217107\,i& 19 & \phantom{-}0.0004112716+0.0001954204\,i\\
    8  & -0.0006766166-0.0015392566\,i&           20 & -0.0000011706+0.0004183253\,i\\
    9  & \phantom{-}0.0016074870-0.0000503663\,i& 21 & -0.0001027596+0.0004091624\,i\\
    10 & -0.0006908394+0.0018753575\,i&           22 & -0.0004725451+0.0004237489\,i\\
    11 & -0.0008974336+0.0007658455\,i&           23 & -0.0000596181+0.0002323317\,i  
  \end{array}\end{equation*}
  \caption{First $24$ Fourier coefficients of the abelian complexity
    function $\rho(n)$ of the paperfolding sequence.}
  \label{tab:paperfoldingCoeff}
\end{table}

\begin{figure}
  \centering
  \includegraphics{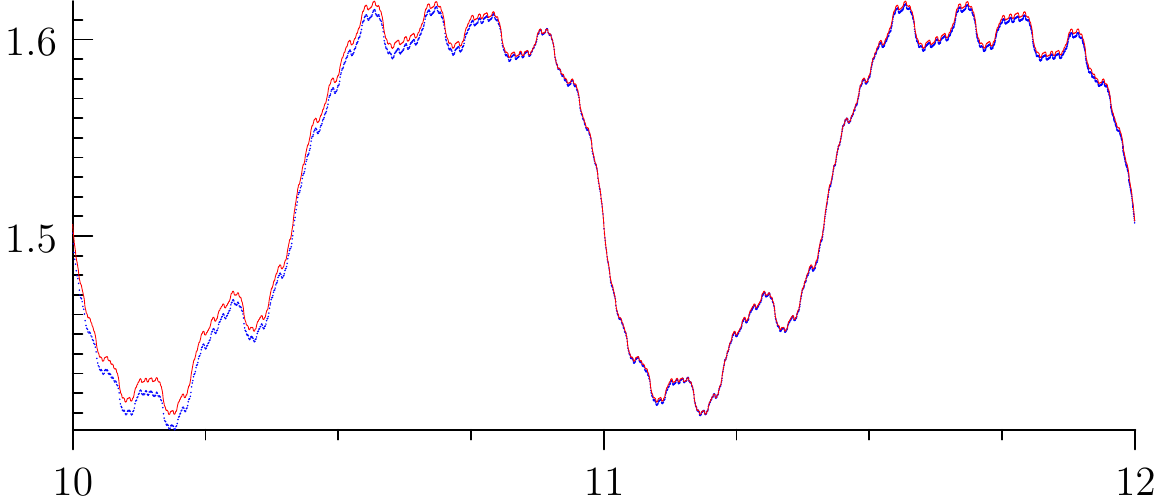}
  \caption{Partial Fourier series compared with the empirical values of
    $\Psi_1(x)$ of the abelian complexity function of the paperfolding sequence.}
  \label{fig:paperfoldingFourier}
\end{figure}

\section{Asymptotic Distribution --- Proof of
  Theorem~\ref{thm:asydist}}\label{sec:distribution}

This section contains some lemmas which will together imply Theorem~\ref{thm:asydist}.
Our plan is as follows: First, we give auxiliary lemmas
about the eigenvalues and eigenvectors of the transition matrix $M$ in
Section~\ref{sec:prel-from-lin-alg}. Section~\ref{sec:char-funct} contains an
asymptotic formula for the characteristic function of the random variable $\T(\bfn)$. We use this characteristic function to give
formulas for the expected value and the variance in Section~\ref{sec:moments}, and prove the continuity of
the periodic fluctuations in Section~\ref{sec:holder-continuity}. Finally, we prove the
central limit theorem in Section~\ref{sec:clt}. 

We use the notation \begin{math}(\bfeps_{L}\ldots\bfeps_{0})_{q}\end{math} for the standard
\begin{math}q\end{math}-ary joint digit representation of an integer vector
with \begin{math}\bfeps_{L}\neq0\end{math}. For a real number in
the interval \begin{math}[0,q)\end{math}, we
write \begin{math}(\eps_{0}\centerdot\eps_{1}\ldots)_{q}\end{math} for the \begin{math}q\end{math}-ary digit
representation choosing the representation ending on \begin{math}0^{\omega}\end{math} in the case
of ambiguity. Furthermore, we use Iverson's notation~\cite{Graham-Knuth-Patashnik:1994}: $[\mathit{expression}]$ is $1$ if
$\mathit{expression}$ is true and $0$ otherwise. All \begin{math}\bigOh\end{math}-constants depend only on
\begin{math}q\end{math}, \begin{math}d\end{math} and the number of states. 

\subsection{Transition Matrix and its Eigenvectors}\label{sec:prel-from-lin-alg}

This section contains the proofs of  some results on the eigenvalues, eigenvectors and eigenprojections
of the transition matrix $M$.

For the proof of Theorem~\ref{thm:asydist}, we use the following lemma which
describes the eigenvalues of a matrix in a similar way as the
Perron--Frobenius theorem (cf.~\cite{Godsil-Royle:2001:alggraphtheory}).

\begin{lemma}\label{lem:eigenvalues}
Let \begin{math}M\end{math} be a matrix with complex entries whose underlying directed graph is
\begin{math}p\end{math}-periodic and strongly connected. Then the set of non-zero eigenvalues of \begin{math}M\end{math} can be
partitioned into disjoint sets of cardinality \begin{math}p\end{math} where each set is invariant under
multiplication by \begin{math}e^{2\pi i/p}\end{math} and all eigenvalues in one set have
the same algebraic multiplicities.
\end{lemma}
\begin{proof}
Since the underlying directed graph of \begin{math}M\in\complexes^{n\times n}\end{math} is a strongly
connected, \begin{math}p\end{math}-periodic graph, we can write \begin{math}M\end{math} as
\begin{equation*}
M=
\begin{pmatrix}
0      &A_{2}  &0     &\cdots&0     \\
\vdots &\ddots&A_{3  }&\ddots&\vdots\\
\vdots &      &\ddots&\ddots&0     \\
0      &      &      &\ddots&A_{p}  \\
A_{1}   &0     &\cdots&\cdots&0
\end{pmatrix}
\end{equation*}
with block matrices \begin{math}A_{i}\end{math} by reordering the
vertices. Then
\begin{math}
  M-xI
\end{math}
is the product of the matrices
\begin{equation*}
\begin{pmatrix}
-xI    &0     &\cdots&\cdots&0     \\
0      &\ddots&\ddots&      &\vdots\\
\vdots &\ddots&\ddots&\ddots&\vdots\\
0      &\cdots&0     &-xI   &0    \\
A_{1}
&\frac1x\prod_{j=1}^{2}A_{j}&\cdots&\frac1{x^{p-2}}\prod_{j=1}^{p-1}A_{j}&\frac1{x^{p-1}}\prod_{j=1}^{p}A_{j}-xI\end{pmatrix}
\end{equation*} and 
\begin{equation*}
\begin{pmatrix}
I      &-\frac1xA_{2}&0     &\cdots&0     \\
0      &\ddots    &\ddots&\ddots&\vdots\\
\vdots &\ddots    &\ddots&\ddots&0     \\
\vdots &          &\ddots&I     &-\frac1xA_{p}\\
0      &\cdots    &\cdots&0     &I
\end{pmatrix}.
\end{equation*}
Let \begin{math}h(x)\end{math} be the characteristic polynomial of \begin{math}\prod_{j=1}^{p}A_{j}\in
\complexes^{m\times m}\end{math}. Thus
the characteristic polynomial of \begin{math}M\end{math} is \begin{math}x^{n-m-(p-1)m}h(x^{p})\end{math}. Therefore,
the eigenvalues of \begin{math}M\end{math} are either \begin{math}0\end{math} or any \begin{math}p\end{math}-th root of a non-zero
eigenvalue of \begin{math}\prod_{j=1}^{p}A_{j}\end{math}.
\end{proof}

With this lemma, we can prove Lemma~\ref{lem:eigenvalues-M} about the
eigenvalues of the matrix $M$:

\begin{proof}[Proof of Lemma~\ref{lem:eigenvalues-M}]
  First, consider the case
\begin{math}t=0\end{math}. By construction, \begin{math}q^{d}\end{math}
is an eigenvalue with right eigenvector \begin{math}\bfones\end{math} of
\begin{math}M\end{math}. As \begin{math}\|M\|_{\infty}\leq q^{d}\end{math},
where $\|\,\cdot\,\|_{\infty}$ denotes the row sum norm, \begin{math}q^{d}\end{math} is a dominant eigenvalue. 

Consider the strongly connected components of the underlying graph of
\begin{math}\T\end{math}. Each final strongly connected component \begin{math}C_{j}\end{math} induces a final transducer \begin{math}\T_{j}\end{math} which is strongly
connected, complete, deterministic
and \begin{math}p_{j}\end{math}-periodic. Thus, the adjacency matrix at \begin{math}t=0\end{math} of this final transducer  has
a dominant
eigenvalue \begin{math}q^{d}\end{math} with right
eigenvector \begin{math}\bfones\end{math}. By the
Perron--Frobenius theorem (cf.~\cite[Theorem~8.8.1]{Godsil-Royle:2001:alggraphtheory}), all dominant eigenvalues
of this final transducer are \begin{math}\{q^{d}e^{2\pi il/p}\mid
  l\in\calP \text{ with }p\mid lp_{j}\}\end{math}, each with
algebraic and geometric multiplicity one.

A non-final strongly connected component induces a transducer
\begin{math}\mathcal S\end{math} with the adjacency
matrix \begin{math}S\end{math}. This transducer is not
complete. Let \begin{math}\mathcal S^{+}\end{math} be the complete
transducer where loops are added to states of \begin{math}\mathcal S\end{math}
where necessary.  The adjacency matrix of \begin{math}\mathcal S^{+}\end{math}
is \begin{math}S^{+}\end{math}. Since \begin{math}\mathcal S^{+}\end{math} is
complete, deterministic and strongly connected,
\begin{math}\rho(S^{+})=q^{d}\end{math}. As \begin{math}S\leq S^{+}\end{math} but \begin{math}S\neq
S^{+}\end{math}, Theorem 8.8.1 in~\cite{Godsil-Royle:2001:alggraphtheory} implies \begin{math}\rho(S)<\rho(S^{+})=q^{d}\end{math}.

Thus, the dominant eigenvalues are $q^{d}e^{2\pi il/p}$ with an $l\in\calP$
such that there exists a $j\in\{1,\ldots,c\}$ with $p\mid lp_{j}$. We determine the geometric multiplicities of these dominant eigenvalues of $M$
in Lemma~\ref{lem:diagonalization}.

Now, fix a final strongly connected component $C_j$ and some $l\in\calP$ with $p\mid lp_{j}$. In a small neighborhood of \begin{math}t=0\end{math},
let \begin{math}\mu_{lj}(t)\end{math} be the eigenvalue of the submatrix
of \begin{math}M\end{math} corresponding to the complete transducer \begin{math}\T_{j}\end{math} with
\begin{math}\mu_{lj}(0)=q^{d}e^{2\pi i l/p}\end{math}. Because of Lemma~\ref{lem:eigenvalues} applied to
the final component $C_j$ separately, we have
\begin{math}\mu_{lj}(t)=e^{2\pi i l/p}\mu_{j}(t)\end{math}
where $\mu_j(t)$ is defined to be $\mu_{0j}(t)$.

All other moduli of eigenvalues of \begin{math}M\end{math} are less than
\begin{math}\min_{l, j}|\mu_{lj}(t)|\end{math} because of the continuity of
eigenvalues.

We prove the differentiability of the eigenvalues in Lemma~\ref{lem:diagonalization}.
\end{proof}

\begin{lemma}\label{lem:diagonalization}
Let $\mu_{j}(t)\exp(\frac{2\pi il}{p})$ be a dominant eigenvalue of the matrix
$M$. There exists a corresponding left eigenvector of $M$ with zero entries
except in coordinates corresponding to the final component $C_{j}$.

At $t=0$, the algebraic and geometric multiplicities of $q^{d}\exp(\frac{2\pi
  il}{p})$ coincide. 

Furthermore the eigenvalues and the eigenprojection corresponding to the
eigenvalues $\mu_{j}\exp(\frac{2\pi il}{p})$  are
analytic at $t=0$.
\end{lemma}
\begin{proof} Let $q^{d}\exp(\frac{2\pi il}{p})$ be a dominant eigenvalue of
  $M$. Its algebraic multiplicity at $t=0$ is $|\{j: p\mid lp_{j}\}|$.
We construct  exactly one
left eigenvector in the neighborhood of $t=0$ for each final component $C_{j}$ with $p\mid lp_{j}$:
Let $\T_{j}$ be the induced transducer of the final component $C_{j}$.
Let $\tilde\bfv^{\top}(t)$ be a left eigenvector of the adjacency matrix of $\T_{j}$ corresponding to the eigenvalue
$\mu_{j}(t)\exp(\frac{2\pi il}{p})$. As the algebraic multiplicity is $1$ in this final
component,
the choice of $\tilde\bfv^{\top}(t)$ is unique up to multiplication with a
scalar function in $t$.
 Then, we construct the left eigenvector $\bfv^{\top}(t)$ by padding
 $\tilde\bfv^{\top}(t)$ with zeros. 

These left eigenvectors are linearly independent because of the block
structure induced by the final components. Thus the
geometric and the algebraic multiplicities of $q^{d}\exp(\frac{2\pi il}{p})$
coincide.

Furthermore, $\mu_{j}(t)\exp(\frac{2\pi il}{p})$ is a simple eigenvalue of the
adjacency matrix of $\T_{j}$. Therefore, \cite[Chapter II]{Kato:1976:pertur-theor-linear-operat} implies the differentiability of
the eigenvalues and eigenprojections.
\end{proof}

From now on, we use the convention that the eigenspace corresponding to
$\mu_{j}(t)\exp(\frac{2\pi il}{p})$ is the null space if
$\mu_{j}(t)\exp(\frac{2\pi il}{p})$ is not an eigenvalue. Then its
eigenprojection is the constant null function.
  
\begin{definition}\label{def:eigenprojections-of-e1}
  Let $\bfw_{lj}^{\top}(t)$ be the eigenprojection of $\bfe_{1}^{\top}$ onto the left
  eigenspace corresponding to the possible eigenvalue
  $\mu_{j}(t)\exp(\frac{2\pi il}{p})$. The vector $\bfw_{lj}^{\top}(t)$ is thus a null vector or
  a left eigenvector of $M$ corresponding to the eigenvalue
  $\mu_{j}(t)\exp(\frac{2\pi il}{p})$.

Define
\begin{equation*}
  \bfw^{\top}(t)=\bfe_{1}^{\top}-\sum_{l\in\calP}\sum_{j=1}^{c}\bfw_{lj}^{\top}(t).
\end{equation*}

As an abbreviation, we write $\bfw_{lj}^{\top}$, $\bfw^{\top}$, $\bfw_{lj}'^{\top}$ and $\bfw'^{\top}$ for these projections and
their derivatives at $t=0$.
\end{definition}

\begin{remark}\label{rem:vanishingerror}
If there are only dominant eigenvalues, then $\bfw^{\top}(t)=0$. This will
imply that there is no error term in the asymptotic expansion of the expected
value and the variance. This occurs in the case of the sum of digits of the standard \begin{math}q\end{math}-ary
digit representation and other completely \begin{math}q\end{math}-additive
functions because the transducer has only one state.
\end{remark}

\begin{lemma}
  \label{lem:error-eigenvector}
  In a fixed neighborhood of $t=0$, let $\xi>0$ be as
  defined in~\eqref{eq:const-of-thm}, i.e., all non-dominant
  eigenvalues have modulus less than $q^{d-\xi}$. Then 
  \begin{equation*}
\Big\|\frac{d^{k}}{d t^{k}}\bfw^{\top}(t) M^{m}\Big\| =\bigOh(c_{k}^{(1)}q^{(d-\xi)(m-k)}m^{k})
\end{equation*}
for $m$, $k\geq 0$ and a constant $c_{k}^{(1)}$.
\end{lemma}
\begin{proof}
  Let $P$ be the matrix such that $x^{\top}\mapsto x^{\top}P$ is the sum of the eigenprojections onto the left eigenspaces corresponding
  to $\mu_{j}\exp(\frac{2\pi il}{p})$ for $j=1$, \dots, $c$
  and $l\in\calP$. Then
  $\bfw^{\top}=\bfe_{1}^{\top}(I-P)$ and
  \begin{equation*}
    \bfw^{\top}M^{m}=\bfe_{1}^{\top}((I-P)M)^{m}.
  \end{equation*}
  As the spectral radius of $(I-P)M$ is less than $q^{d-\xi}$, we obtain the stated estimates.
\end{proof}

With $\bfw_{l}^{\top}$ defined in Section~\ref{sec:prop-trans-matr}, we have
\begin{equation}\label{eq:w-l}
  \bfw_{l}^{\top}(t)=\sum_{j=1}^{c}\bfw_{lj}^{\top}(t).
\end{equation}

Note that left and right eigenvectors corresponding to different eigenvalues
annihilate each other.
Because of the block structure of the eigenvectors in
Lemma~\ref{lem:diagonalization} and because $\bfones$ is a right eigenvector to $q^{d}$, we have
\begin{equation}\label{eq:lambda-or-0}
  [l=0]\lambda_{j}=\bfw_{lj}^{\top}\bfones
\end{equation}
where $\lambda_{j}$ is
defined in Section~\ref{sec:prop-trans-matr}.
Furthermore, $\bfw^{\top}\bfones=0$ and
\begin{equation*}
\sum_{j=1}^{c}\lambda_{j}=\sum_{l\in\calP}\sum_{j=1}^{c}\bfw_{lj}^{\top}\bfones+\bfw^{\top}\bfones=\bfe_{1}^{\top}\bfones=1.\end{equation*}  

Denote by $\bfdelta$ the vector whose $s$-th component is the sum of the outputs of all
transitions leaving the state $s$. By the definition of the transition matrix
$M(t)$, $\bfdelta$ can be expressed  as
\begin{equation}\label{eq:bfdelta-is-derivative}
  i\bfdelta = \left.\frac{d}{dt}M(t)\bfones\right|_{t=0}.
\end{equation}
We now establish a relation between $\bfdelta$, the left eigenvector $\bfw_l^\top$
and its derivative at $t=0$. By definition of the left eigenvectors $\bfw_{lj}^\top(t)$ and \eqref{eq:w-l},
\begin{equation*}
\bfw_{l}^\top(t)
  M\bfones=\sum_{j=1}^{c}\mu_{j}(t)\exp\Bigl(\frac{2\pi
    il}{p}\Bigr)\bfw_{lj}^\top(t)\bfones.
\end{equation*}
Differentiation, \eqref{eq:lambda-or-0}, \eqref{eq:const-of-thm} and
\eqref{eq:w-l} yield
\begin{equation}\label{eq:bfw_bfdelta}
  \bfw_{l}^\top \bfdelta=[l=0]e_{\T}q^{d}-q^{d}\big(e^{\frac{2\pi i l}{p}}-1\big)i\bfw_{l}'^{\top}\bfones.
\end{equation}

To establish the interpretation of $\bfw_{0}^{\top}$ given at the end of
Section~\ref{sec:prop-trans-matr}, we consider
\begin{equation*}
  \hat\bfw_{k}^{\top}:=
 \lim_{m\rightarrow\infty}\bfe_{1}^{\top}M^{mp+k}q^{-d(mp+k)},
\end{equation*}
the stationary distribution on the state space of all states of the transducer
under the assumption that the input length is congruent to $k$ modulo $p$. Using
\eqref{eq:w-l} and Lemma~\ref{lem:error-eigenvector} yields
\begin{align*}
 \hat\bfw_{k}^{\top}
 &=\lim_{m\to\infty}\Big(\sum_{l\in\calP}\bfw_{l}^{\top}+\bfw^{\top}\Big)M^{mp+k}q^{-d(mp+k)}\\
 &=\lim_{m\to\infty}\sum_{l\in\calP}\exp\Big(\frac{2\pi ilk}{p}\Big)\bfw_{l}^{\top}+\bigOh(q^{-\xi(mp+k)})\\
 &=\sum_{l\in\calP}\exp\Big(\frac{2\pi
 ilk}{p}\Big)\bfw_{l}^{\top}.
\end{align*}
Summation leads to
$\frac
1p\sum_{k=0}^{p-1}\hat\bfw_{k}^{\top}=\bfw_{0}^{\top}$.
Thus, $\lambda_{j}$ is the hitting probability of the final component $C_{j}$
when starting in the initial state. As every state is accessible from the
initial state, $\lambda_{j}$ is
positive.

Finally, for $l=0$, \eqref{eq:bfw_bfdelta} reads $q^{-d}\bfw_0^\top
\bfdelta=e_\T$, which can be interpreted as the steady state analysis of the
expectation: the probability distribution $\bfw_0^\top$ is multiplied with the
expected output $q^{-d}\bfdelta$.

\subsection{Characteristic function}\label{sec:char-funct}

To obtain a central limit law in Section~\ref{sec:clt}, we compute an asymptotic formula for the
characteristic function in this section.

The next lemma can be proved by
induction on \begin{math}L\end{math}. It is a generalization of Lemma~3 in~\cite{Heuberger-Kropf:2013:analy}.
\begin{lemma}\label{lem:rec}
Let \begin{math}A_{\eps}\end{math}, \begin{math}\eps=0,\ldots,q-1\end{math} be matrices in \begin{math}\complexes^{n\times n}\end{math}, \begin{math}H_{\eps}\colon\naturals_{0}\to
\complexes^{n\times n}\end{math} be known functions
with \begin{math}H_{0}(0)=0\end{math}. Let \begin{math}G\colon\naturals_{0}\to\complexes^{n\times n}\end{math} be a
function which satisfies the recurrence relation
\begin{equation*} G(qN+\eps)=A_\eps G(N) + H_{\eps}(N) \end{equation*}
for \begin{math}N\ge
  0\end{math}, \begin{math}\eps\in\{0,\ldots,q-1\}\end{math} and $G(0)=0$. Then
\begin{equation*}G\left((\eps_{L}\ldots\eps_{0})_{q}\right)=\sum_{l=0}^L\biggl(\prod_{i=0}^{l-1}A_{\eps_i}\biggr)H_{\eps_{l}}\left((\eps_{L}\ldots\eps_{l+1})_{q}\right).\end{equation*}
\end{lemma}

The solution of this recursion finally leads to an asymptotic formula for the
characteristic function.

We choose the
branch \begin{math}-\pi+\frac{\pi}{p}<\arg z\leq\pi+\frac{\pi}{p}\end{math} of the complex logarithm. After
setting \begin{math}t=0\end{math}, we use only the logarithm of complex
numbers for which our branch coincides the principal
branch \begin{math}-\pi<\arg z\leq\pi\end{math}.

\begin{lemma}\label{lem:char-func}
  The characteristic function of the random variable $\T(\bfn)$ is
\begin{multline*}\mathbb E(\exp(it\T(\bfn)))=\\
\frac{1}{N^{d}}\sum_{l\in\calP}\sum_{j=1}^{c}\mu_{j}(t)^{\log_{q}N}\exp\Big(\frac{2\pi
      i
      l\log_{q}N}{p}\Big)\Psi_{lj}(\log_{q}N,t)+R(N,t)
\end{multline*}
with functions $\Psi_{lj}(x,t)$ (defined in~\eqref{eq:psi-jk}), which are arbitrarily
often differentiable in $t$ and $1$-periodic in $x$, and an error term $R(N,t)$. This error term $R(N,t)$
is arbitrarily often differentiable, too, and satisfies
$\frac{d^{k}}{d t^{k}}R(N,t)=\bigOh(c_{k}^{(2)}N^{-\xi}\log^{k}N)$, for
$k\geq 0$, a constant $c_{k}^{(2)}$ and the constant $\xi>0$ defined in
Section~\ref{sec:prop-trans-matr}, in a neighborhood of $t=0$. At $t=0$, we
have $R(N,0)=0$.
\end{lemma}
\begin{proof}
  For a transducer \begin{math}\T\end{math}, consider the characteristic function
\begin{equation}\label{eq:expsum2}
F(N)=\Expect(\exp(it\T(n)))=\frac{1}{N^{d}}\sum_{\bfn\in\Omega_{N}}e^{it\T(\bfn)}
\end{equation}
of the discrete random variable $\T(\bfn)$.

Then the summands in~\eqref{eq:expsum2} can be expressed as a matrix
product
\begin{equation*}
e^{it\T(\bfn)}=\bfe_{1}^{\top}\prod_{l=0}^{L}M_{\bfeps_{l}}\bfu
\end{equation*}
where \begin{math}(\bfeps_{L}\ldots\bfeps_{0})_{q}\end{math} is the standard \begin{math}q\end{math}-ary joint digit
representation of \begin{math}\bfn\end{math} with \begin{math}\bfeps_{L}\neq 0\end{math} and the vector \begin{math}\bfu\end{math}
has entries \begin{math}e^{itb(s)}\end{math} where \begin{math}b(s)\end{math}
is the final output of the state \begin{math}s\end{math}. Again, the
vector \begin{math}\bfe_{1}\end{math} is the indicator
vector of the initial state.

Let 
\begin{equation*}
g(\bfn)=\prod_{l=0}^{L}M_{\bfeps_{l}}\end{equation*} 
and
\begin{equation*}
G(N)=\sum_{\bfn\in\Omega_{N}}g(\bfn),\end{equation*}
hence
\begin{equation}
  \label{eq:char-func-matrix-prod}
  F(N)=\frac{1}{N^{d}}\bfe_{1}^{\top}G(N)\bfu.
\end{equation}
The function \begin{math}g(\bfn)\end{math} satisfies the recursion
\begin{equation}\label{eq:rec-g}
  g(q\bfn+\bfeps)=M_{\bfeps}g(\bfn)
\end{equation}
for $\bfeps\in\{0,1,\ldots,q-1\}^{d}$, $\bfn\geq 0$ with $q\bfn+\bfeps\neq 0$.

We define further functions 
\begin{equation}\label{eq:G_C-definition}
G_{C}(N)=\sum_{\substack{0\leq n_{i}<N\\i\not\in
    C}}\sum_{\substack{n_{i}=N\\i\in C}}g(\bfn)\end{equation}
where the coordinates \begin{math}n_{1}\end{math},
\dots, \begin{math}n_{d}\end{math} of \begin{math}\bfn\end{math} with indices
in the set \begin{math}C\subseteq\{1,\ldots,d\}\end{math} are fixed
to \begin{math}N\end{math}. This yields \begin{math}G(N)=G_{\emptyset}(N)\end{math}. 
Furthermore, we define the matrices
\begin{equation}\label{eq:matrices-M}
   M_{C,D}^{\eps}=\sum_{\substack{\beta_{i}=0\\ i\not\in C\cup
     D}}^{q-1}\sum_{\substack{\beta_{i}=0\\i\in
       D}}^{\eps-1}\sum_{\substack{\beta_{i}=\eps\\i\in C}}M_{\bfbeta}
 \end{equation}
for disjoint sets $C$, $D\subseteq\{1,\ldots,d\}$ and
$\eps\in\{0,1,\ldots,q-1\}$. In this definition, we restrict the $i$-th
coordinate $\beta_{i}$ of $\bfbeta$ to be $\eps$ or less than $\eps$ if $i\in C$
or $i\in D$, respectively. Otherwise, the $i$-th coordinate can be arbitrary. 
Then, $M=M_{\emptyset,\emptyset}^{\eps}$ holds independently of $\eps$.

 Then, \eqref{eq:rec-g} yields the following recursions
for \begin{math}G_{C}(N)\end{math}, \begin{math}\eps=0\end{math},
\dots, \begin{math}q-1\end{math}, \begin{math}N\geq0\end{math} and \begin{math}C\neq\{1,\ldots,d\}\end{math}:
\begin{equation}\label{eq:recursion-general}
\begin{aligned}
G_{C}(qN+\eps)&=\sum_{\substack{\beta_{i}=0\\i\not\in
      C}}^{q-1}\sum_{\substack{\beta_{i}=\eps\\i\in
      C}}\sum_{\substack{0\leq qm_{i}+\beta_{i}<qN+\eps\\i\not\in
      C}}\sum_{\substack{qm_{i}+\beta_{i}=qN+\eps\\i\in C}}g(q\boldsymbol{m}+\bfbeta)\\
&=[C=\emptyset \wedge qN+\eps\neq 0](I-M_{0})\\
&\quad+\sum_{\substack{\beta_{i}=0\\i\not\in
      C}}^{q-1}\sum_{\substack{\beta_{i}=\eps\\i\in
      C}}M_{\bfbeta}\sum_{\substack{0\leq m_{i}<N+\frac{\eps-\beta_{i}}q\\i\not\in
      C}}\sum_{\substack{m_{i}=N\\i\in C}}g(\boldsymbol{m})\\
&=[C=\emptyset \wedge qN+\eps\neq 0](I-M_{0})+\sum_{D\subseteq C^{c}}M_{C,D}^{\eps}G_{C\cup D}(N).
\end{aligned}
\end{equation}

This recursion for \begin{math}G_{C}\end{math} only depends on \begin{math}G_{C'}\end{math} for \begin{math}C'\supsetneq
C\end{math}. As
\begin{equation*}
G_{\{1,\ldots,d\}}(N)=g(N\bfones),\end{equation*}
we can recursively determine \begin{math}G_{C}\end{math} using
Lemma~\ref{lem:rec}. In particular, for \begin{math}G(N)\end{math}, this yields the recursion formula 
\begin{equation}\label{eq:recursion-G}
G(qN+\eps)=MG(N)+H_{\eps}(N)\end{equation}
for \begin{math}N\geq 0\end{math}, \begin{math}\eps\in\{0,\ldots
  q-1\}\end{math} where \begin{math}H_{\eps}\end{math} are known functions with
\begin{equation}\label{eq:H}H_{\eps}(N)=[qN+\eps\neq 0](I-M_{0})+\sum_{\emptyset\neq D\subseteq \{1,\ldots,d\}}M_{\emptyset,D}^{\eps}G_{D}(N).\end{equation}
		
Thus by Lemma~\ref{lem:rec}, we get
\begin{equation}\label{eqn:recursion}G((\eps_{L}\ldots\eps_{0})_{q})=\sum_{m=0}^LM^mH_{\eps_{m}}\left((\eps_{L}\ldots\eps_{m+1})_{q}\right).\end{equation}

By construction, \begin{math}\|M_{\bfeps}\|_{\infty}=1\end{math} for every \begin{math}\bfeps\in\{0,\ldots,q-1\}^{d}\end{math}. We conclude that
\begin{math}\|M_{C,D}^{\eps}\|_{\infty}\le q^{d-|C|-|D|}\eps^{|D|}\end{math}.
By the definition of \begin{math}G_{C}(N)\end{math}, the growth rates of the
functions \begin{math}G_{C}(N)\end{math}
and \begin{math}H_{\eps}(N)\end{math} are \begin{math}\|G_{C}(N)\|_{\infty}=\bigOh(N^{d-|C|})\end{math} and
\begin{math}\|H_{\eps}(N)\|_{\infty}=\bigOh(N^{d-1})\end{math}, respectively. For $k\geq 0$, the
$k$-th derivative of $H_{\eps}(N)$ at $t=0$ can be bounded by
$\bigOh(c_{k}^{(3)}N^{d-1}\log^{k} N)$ for a constant $c_{k}^{(3)}$. 

We define
\begin{equation*}
  R(N,t)=\frac{1}{N^{d}}\sum_{m=0}^{L}\bfw^{\top}M^{m}H_{\eps_{m}}((\eps_{L}\ldots\eps_{m+1})_{q})\bfu,
\end{equation*}
which constitutes an explicit expression for the error term contributed by the non-dominant eigenvalues. By Lemma~\ref{lem:error-eigenvector}, its derivatives satisfy \[\frac{d^{k}}{d t^{k}}R(N,t)=\bigOh(c_{k}^{(2)}N^{-\xi}\log^{k}N)\]
for $k\geq 0$. Because $\bfu(0)=\bfones$ and left and right eigenvectors corresponding to
different eigenvalues annihilate each other, we have $R(N,0)=0$.

By~\eqref{eq:char-func-matrix-prod},~\eqref{eqn:recursion} and $\bfe_{1}^{\top}=\sum_{l\in\calP}\sum_{j=1}^{c}\bfw_{lj}^{\top}+\bfw^{\top}$,
\begin{align*}
    F(N)&=\frac{1}{N^{d}}\sum_{l\in\calP}\sum_{j=1}^{c}\mu_{j}^{L}\exp\Big(\frac{2\pi
      ilL}{p}\Big)\\&\quad\cdot\sum_{m=0}^{L}\mu_{j}^{m-L}\exp\Big(\frac{2\pi
      il(m-L)}{p}\Big)\bfw_{lj}^{\top}H_{\eps_{m}}((\eps_{L}\ldots\eps_{m+1})_{q})\bfu\\&\quad+R(N,t)\\
&=\frac{1}{N^{d}}\sum_{l\in\calP}\sum_{j=1}^{c}\mu_{j}^{\log_{q}N}\exp\Big(\frac{2\pi
      il\log_{q}N}{p}\Big)\Psi_{lj}(\log_{q}N,t)+R(N,t)
\end{align*}
with 
\begin{multline}\label{eq:psi-jk}\Psi_{lj}(x,t)=\mu_{j}(t)^{-\{x\}}\exp\Big(-\frac{2\pi
    i l\{x\}}{p}\Big)\\\cdot\sum_{m=0}^{\infty}\mu_{j}(t)^{-m}\exp\Big(-\frac{2\pi ilm}{p}\Big)\bfw_{lj}^{\top}H_{x_{m}}((x_{0}\ldots
  x_{m-1})_{q})\bfu\end{multline}
and $q^{\{x\}}=(x_{0}\centerdot x_{1}\ldots)_{q}$, choosing the representation
ending on $0^{\omega}$ in the case of ambiguity.

The functions \begin{math}\Psi_{lj}(x,t)\end{math} are periodic in \begin{math}x\end{math} with period \begin{math}1\end{math} and well
defined for all \begin{math}x\in\reals\end{math} since they are dominated by
geometric series. Furthermore, they are arbitrarily often
differentiable in \begin{math}t\end{math}.
\end{proof}

\subsection{Moments}\label{sec:moments}
In this section we give the moments of the output sum $\T(\bfn)$.

\begin{lemma}\label{lem:exp-var}
  The expected value and the variance of $\T(n)$ are as stated in
  Theorem~\ref{thm:asydist} with constants given in~\eqref{eq:const-of-thm} and periodic functions given
  in Lemma~\ref{lemma:psi_1-explicit}
  and~\eqref{eq:psi2}.
\end{lemma}
\begin{proof}
  The derivative of \begin{math}\Expect(\exp(it\T(\bfn)))\end{math} with
  respect to \begin{math}t\end{math} at \begin{math}t=0\end{math} gives the expected value
of the sum of the output of the transducer
\begin{multline*}
\Expect(\T(n))=\frac1{N^{d}}\sum_{\bfn\in\Omega_{N}}\T(n)=\Psi_{0}(\log_{q}N)\log_{q}N+\Psi_{1}(\log_{q}N)\\+\bigOh(N^{-\xi}\log
N)\end{multline*}
with \begin{math}p\end{math}-periodic functions 
\begin{equation}\label{eq:psi1-general}
  \begin{aligned}
    \Psi_{0}(x)&=\sum_{l\in\calP}\sum_{j=1}^{c}a_{j}e^{\frac{2\pi
        ilx}{p}}\Psi_{lj}(x,0),\\
    \Psi_{1}(x)&=-i\sum_{l\in\calP}\sum_{j=1}^{c}e^{\frac{2\pi
        ilx}{p}}\Psi_{lj}'(x,0)
  \end{aligned}
\end{equation}
and constants $a_{j}$ defined in~\eqref{eq:const-of-thm}. Here, $\Psi_{lj}'$
denotes the derivative with respect to $t$.

We now compute
\begin{math}\Psi_0(x)\end{math} for some \begin{math}x\end{math}
with \begin{math}q^{\{x\}}=(x_0\centerdot x_1\ldots)_q\end{math}. To compute
$H_\eps(N)$, we use \eqref{eq:recursion-G} and the definition of $G(N)$ to obtain
\begin{equation}\label{eq:H_eps-explicit}
  H_\eps(N)\bfones = ((qN+\eps)^d-(qN)^d)\bfones
\end{equation}
for $t=0$, because $\bfones$ is a right eigenvector of $M_{\bfeps}$ for every $\bfeps$.
Together with \eqref{eq:psi-jk}, this results in
\begin{equation*}
  \Psi_{lj}(x,0)=q^{-d\{x\}}\exp\Big(-\frac{2\pi
    i l\{x\}}{p}\Big)\bfw_{lj}^{\top}\bfones D\Big(q^{d}e^{\frac{2\pi il}{p}}\Big)
\end{equation*}
with 
\begin{equation*}
  D(z)=\sum_{m=0}^\infty z^{-m}((x_0\ldots x_m)_q^d-(x_0\ldots x_{m-1}0)_q^d).
\end{equation*}
By \eqref{eq:lambda-or-0}, we have $\Psi_{lj}(x, 0)=0$ for 
\begin{math}l\neq 0\end{math}.

To compute $D(q^d)$, observe that
\begin{align*}
  D(q^d)&=\sum_{m=0}^\infty \bigl((x_0\centerdot x_1\ldots x_m)_q^d -
  (x_0\centerdot x_1\ldots x_{m-1})_q^d\bigr)\\
  &=\lim_{m\to\infty} (x_0\centerdot x_1\ldots x_m)_q^d=q^{d\{x\}}
\end{align*}
because $D(q^d)$ is a telescoping sum.

We conclude that
\begin{equation}\label{eq:psi-lambda}
  \Psi_{lj}(x,0)=\lambda_{j}[l=0]
\end{equation}
and therefore
\begin{equation*}
\Psi_{0}(x)=\sum_{j=1}^{c}a_{j}\lambda_{j}=e_{\T}\end{equation*}
by \eqref{eq:const-of-thm}. This completes the proof of the expectation as
given in \eqref{eq:expected}.

Using Lemma~\ref{lem:char-func} and \eqref{eq:psi-lambda}, the second derivative of \begin{math}\Expect(\exp(it\T(n)))\end{math} gives 
\begin{align*}
&\frac1{N^{d}}\sum_{\bfn\in\Omega_{N}}\T(n)^{2}=\\
&\qquad\quad\log_{q}^{2}N\sum_{j=1}^{c}a_{j}^{2}\lambda_{j}+v_{\T}\log_{q}N\\
&\qquad\quad-2i\log_{q}N\sum_{l\in\mathcal
P}\sum_{j=1}^{c}a_{j}\exp\Big(\frac{2\pi
il\log_{q}N}{p}\Big)\Psi_{lj}'(\log_{q}N,0)\\
&\qquad\quad+\Psi_{2}(\log_{q}N)+\bigOh(N^{-\xi}\log^{2}
N)\end{align*}
with $v_{\T}$ given in~\eqref{eq:const-of-thm} and
\begin{equation}\label{eq:psi2}
  \Psi_{2}(x)=-\sum_{l\in\calP}\sum_{j=1}^{c}e^{\frac{2\pi ilx}{p}}\Psi_{lj}''(x,0).
\end{equation}
Here, $\Psi_{lj}''$ denotes the second derivative with respect to $t$. Thus, by \eqref{eq:expected}, the variance is
\begin{equation}\label{eq:variance-general}
\begin{aligned}
  \Var(\T(n))&=\frac1{N^{d}}\sum_{\bfn\in\Omega_{N}}\T(n)^{2}-\left(\frac1{N^{d}}\sum_{\bfn\in\Omega_{N}}\T(n)\right)^{2}\\
&=\Big(\sum_{j=1}^{c}a_{j}^{2}\lambda_{j}-e_{\T}^{2}\Big)\log_{q}^{2}N\\
&\quad+\Big(v_{\T}-2i\sum_{l\in\mathcal
P}\sum_{j=1}^{c}a_{j}\exp\Big(\frac{2\pi
il\log_{q}N}{p}\Big)\Psi_{lj}'(\log_{q}N,0)\\
&\quad\qquad-2e_{\T}\Psi_{1}(\log_{q}N)\Big)\log_{q}N\\
&\quad+\Psi_{2}(\log_{q}N)-\Psi_{1}^{2}(\log_{q}N)+\bigOh(N^{-\xi}\log^{2} N).
\end{aligned}
\end{equation}

By Jensen's inequality, the coefficient of $\log_{q}^{2}N$ is zero if and only
if all $a_{j}$ are equal. If all $a_{j}$ are equal, then the coefficient of
$\log_{q}N$ in \eqref{eq:variance-general} simplifies by~\eqref{eq:psi1-general}, too, and we obtain~\eqref{eq:var-good}.
\end{proof}

For the computation of the Fourier coefficients and the proof of the H\"older
condition, we need an explicit expression for $\Psi_{1}$.

In analogy to the definition of $G_C$ in \eqref{eq:G_C-definition}, define
\begin{equation}\label{eq:bfB-def}
  \bfB_{C}(N)=\sum_{\substack{0\leq n_{i}<N\\i\not\in
      C}}\sum_{\substack{n_{i}=N\\i\in C}}\bfb(\bfn)
\end{equation}
for $C\subseteq \{1,\ldots, d\}$.

\begin{lemma}\label{lemma:psi_1-explicit}
  For \begin{math}q^{\{x\}}=(x_{0}\centerdot x_{1}\ldots)_{q}\end{math}, the
  fluctuation $\Psi_1(x)$ can be expressed as
  \begin{equation}\label{eq:psi-1-explicit}
\Psi_{1}(x)=-e_{\T}\{x\}-q^{-d\{x\}}\sum_{l\in\calP}\sum_{m=0}^{\infty}q^{-dm}e^{\frac{2\pi
  i l}{p}(\floor{x}-m)}f_{l}((x_{0}\ldots
x_{m})_{q})
\end{equation}
  with
  \begin{equation}\label{eq:coeff-f}
  \begin{aligned}
    f_l(r)&=[l=0]e_{\T}\big(\lfloor\log_{q}
    r\rfloor(r^{d}-(q\floor{rq^{-1}})^{d})+(q\floor{rq^{-1}})^{d}\big)\\&\quad+i\bfw_{l}'^{\top}\bfones\Big(r^{d}-\exp\Big(\frac{2\pi i l}{p}\Big)(q\floor{rq^{-1}})^{d}\Big)\\&\quad-\bfw_{l}^{\top}\bfB_{\emptyset}(r)+q^{d}\exp\Big(\frac{2\pi
      il}{p}\Big)\bfw_{l}^{\top}\bfB_{\emptyset}(\floor{rq^{-1}}).
  \end{aligned}
  \end{equation}
  The estimate $f_{l}(r)=\bigOh(r^{d-1}\log r)$ holds.
\end{lemma}
\begin{proof}
  From \eqref{eq:psi1-general}, \eqref{eq:psi-jk}, \eqref{eq:const-of-thm},
  \eqref{eq:psi-lambda} and \eqref{eq:lambda-or-0} and the
  absolute convergence of \begin{math}\Psi_{lj}\end{math}, we obtain
  \eqref{eq:psi-1-explicit} with
  \begin{multline*}
f_{l}(r)=[l=0]e_{\T}\lfloor\log_{q} r\rfloor(r^{d}-(q\floor{rq^{-1}})^{d})\\+i\left.\frac{d}{d
  t}\bfw_{l}^{\top}(t)H_{r\bmod q}(\floor{rq^{-1}})\bfu(t)\right\vert_{t=0}.
\end{multline*}

  From the combinatorial interpretation of $\bfb(\bfn)$ and $g(\bfn)\bfu(t)$,
  we obtain
  \begin{equation}\label{eq:bfb-is-derivative}
    i\bfb(\bfn)=\left.\frac{d}{dt}g(\bfn)\bfu(t)\right|_{t=0},
  \end{equation}
  in analogy to \eqref{eq:bfdelta-is-derivative}. As the range of summation of
  $G_C$ and $\bfB_C$ coincides, we immediately get
  \begin{equation}\label{eq:bfB_C-is-derivative}
    i \bfB_C(N) = \left.\frac{d}{dt}G_C(N)\bfu(t)\right|_{t=0}.
  \end{equation}

  By \eqref{eq:H_eps-explicit} and by differentiating $H_\eps(N)\bfu(t)$ using
  \eqref{eq:recursion-G}, \eqref{eq:bfB_C-is-derivative} and
  \eqref{eq:bfdelta-is-derivative},
\begin{multline*}
f_{l}(r)=[l=0]e_{\T}\lfloor\log_{q} r\rfloor(r^{d}-(q\floor{rq^{-1}})^{d})\\
+i\bfw_l'^\top\bfones \bigl(r^d-(q\lfloor rq^{-1}\rfloor)^d \bigr)\\
- \bfw_l^\top \bigl(\bfB_\emptyset(r)-M\bfB_\emptyset(\lfloor rq^{-1}\rfloor) -
\lfloor rq^{-1}\rfloor^d \bfdelta\bigr).
\end{multline*}
  The fact that $\bfw_l^{\top}$ is a left eigenvector of $M$ and
  \eqref{eq:bfw_bfdelta} establish \eqref{eq:coeff-f}.

  For the growth estimate of $f_l(r)$, we use the explicit definition of
  $H_{\eps}$ in \eqref{eq:H},
  \eqref{eq:bfB_C-is-derivative} and the trivial estimate $\|\bfb(\bfn)\|=\bigOh(\log
  \|\bfn\|)$.
\end{proof}

To formulate $\T(\bfn)$ as a $q$-regular sequence, we first define output vectors.
The \begin{math}s\end{math}-th entry of the vector
\begin{math}\bfdelta_{\bfeps}\end{math} is the output label of the transition from state \begin{math}s\end{math}
 with input label \begin{math}\bfeps\end{math}. 
By~\eqref{eq:rec-g}, \eqref{eq:bfb-is-derivative}, and 
\begin{equation}
  \label{eq:delta-eps-is-derivative}
  \frac{d}{dt}M_{\bfeps}\bfones\Big\lvert_{t=0}=i\bfdelta_{\bfeps},
\end{equation}
we have
\begin{equation}
  \label{eq:recursion-b}
  \bfb(q\bfn+\bfeps)=M_{\bfeps}\bfb(\bfn)+\bfdelta_{\bfeps}.
\end{equation}
\begin{remark}
  \label{rem:q-reg}
We can use the matrices
\begin{equation*}
  V_{\bfeps}=
  \begin{pmatrix}
    M_{\bfeps}&\bfdelta_{\bfeps}&[\bfeps=0]I\\
    0&1&0\\
    0&0&[\bfeps=0]I
  \end{pmatrix}
\end{equation*}
and $\bfv(\bfn)=(\bfb(\bfn),1,[\bfn=0](\bfb(0)-M_{0}\bfb(0)-\bfdelta_{0}))^{\top}$ in the definition of a $q$-regular sequence~\eqref{eq:q-reg} to
realize that the output sum of a transducer is $q$-regular. If $d>1$, then
this is a multidimensional $q$-regular sequence (cf.~\cite{Allouche-Shallit:2003:autom}).
\end{remark}

\subsection{H\"older Continuity}\label{sec:holder-continuity}
In this section, we prove the continuity of the fluctuations $\Psi_{1}$ and
$\Psi_{2}$ as well as the  H\"older continuity of $\Psi_1$. This will be used
to establish the convergence of the Fourier series.

\begin{lemma}\label{lem:cont-psi}
The functions $\Psi_{1}(x)$ and, if all $a_{j}$ are equal, $\Psi_{2}(x)$ are continuous for
$x\in\mathbb R$.
\end{lemma}
\begin{proof}
First note that continuity of $\Psi_{1}$ for
\begin{math}x\in\reals\end{math} with \begin{math}x=\log_q y\end{math} where \begin{math}y\end{math} has no finite \begin{math}q\end{math}-ary expansion
follows from the definitions~\eqref{eq:psi-jk} and~\eqref{eq:psi1-general}. To prove it for \begin{math}x=\log_qy\end{math}
with \begin{math}0\le x<p\end{math} where 
\begin{math}y\end{math} has a finite \begin{math}q\end{math}-ary expansion, observe that the two
one-sided limits exist due to the definition. Next, we prove that they are the
same. Consider the two integer sequences \begin{math}N_k=yq^{pk}\end{math}
and \begin{math}\tilde N_k=N_k-1\end{math}
for \begin{math}k\end{math} large enough such that \begin{math}N_k\end{math} is an integer. 
For a real number \begin{math}z\end{math}, we
write \begin{math}\{z\}_p=p\{z/p\}\end{math} for the unique real number in the interval \begin{math}[0,p)\end{math}
such that \begin{math}z-\{z\}_p\end{math} is an integer multiple of \begin{math}p\end{math}.

This yields
  \begin{align*}
    \lim_{k\to\infty}\{\log_q N_k\}_p&=\lim_{k\to\infty}\{\log_q y+pk\}_{p}=\{x\}_{p}=\lim_{z\to x^+}\{z\}_p,\\
    \lim_{k\to\infty}\{\log_q \tilde{N}_k\}_p&=\lim_{k\to\infty}\{\log_{q}N_{k}+\log_{q}(1-N_{k}^{-1})\}_{p}\\&=\lim_{k\to\infty}\{x+\log_{q}(1-N_{k}^{-1})\}_{p}=\lim_{z\to x^-}\{z\}_p.
  \end{align*}
If we insert the two sequences \begin{math}N_k\end{math} and \begin{math}\tilde N_k\end{math} in 
\begin{equation*}
  \sum_{\bfn\in\Omega_N} \T(\bfn)= e_\T N^d\log_q N+N^d \Psi_1(\log_q N)+
  \bigOh(N^{d-\xi}\log N)
\end{equation*}
(cf.~\eqref{eq:expected}) and take the difference, we get
\begin{equation*}
\mathcal O(N_{k}^{d-1}\log N_k)=N_k^{d}\Psi_1(\log_qN_k)-\tilde
N_k^{d}\Psi_1(\log_q\tilde N_k)+\mathcal O(N_k^{d-\xi}\log N_k).
\end{equation*}
Because \begin{math}\Psi_1(x)\end{math} is bounded by a geometric series by definition, we have
\begin{equation*}
  \Psi_1(\log_q N_k)-\Psi_1(\log_q \tilde N_k)=O(N_k^{-\xi}\log N_k)
\end{equation*}
and in particular
 \begin{equation*}
\lim_{k\rightarrow\infty}\Psi_1(\{\log_q
N_k\}_p)=\lim_{k\rightarrow\infty}\Psi_1(\{\log_q\tilde N_k\}_p).\end{equation*} Therefore,
\begin{math}\Psi_1\end{math} is continuous in \begin{math}x\end{math}.

The continuity of \begin{math}\Psi_2(x)\end{math} at $x=\log_q(y)$ for $y$ with
infinite $q$-ary expansion again follows from the definition of $\Psi_2$. If
all $a_j$ are equal, the
continuity of the fluctuation $-\Psi_1^2+\Psi_2$ of the variance \eqref{eq:var-good} follows as above, where $\log
N_k$ has to be replaced by $\log^2 N_k$ in the error terms. Thus $\Psi_2$ is
also continuous in this case.
\end{proof}

\begin{lemma}\label{lem:hoelder}
  The function $\Psi_{1}$ satisfies a H\"older condition of order $\alpha$ for
  all $\alpha\in(0,1)$.
\end{lemma}
\begin{proof}
  Let $0<\alpha<1$ be any constant.
  We want to prove that there exists a positive constant $C$ such that
  \begin{equation}
    \label{eq:hoelder-ineq}
    |\Psi_{1}(y)-\Psi_{1}(x)|\leq C |y-x|^{\alpha}
  \end{equation}
  holds for all $x$, $y\in\reals$.
  
  For $x=y$, the left-hand side of \eqref{eq:hoelder-ineq} is $0$ and the inequality  is obviously
  satisfied. From now on, assume that $x<y$. By the periodicity of $\Psi_1$,
  it is sufficient to prove \eqref{eq:hoelder-ineq} for $0\le x< p$.
  
First, we prove \eqref{eq:hoelder-ineq} for the case $0\leq x<y$ and
sufficiently small $y-x<1$.

  Fix such $x$ and $y$ and choose the integer $k$ such that
  \begin{equation*}
    q^{-k-1}\leq\lvert q^y-q^x\rvert<q^{-k}.
  \end{equation*}

  Note that the continuous differentiability of $z\mapsto q^z$ on the compact
  interval $[0, p+1]$ implies that $q^y-q^x= \bigOh(|y - x|)$ and therefore
  \begin{equation}\label{eq:Hoelder-q_-k-bound}
    q^{-k}=\bigOh(|y - x|).
  \end{equation}

  We prove \eqref{eq:hoelder-ineq} in three steps.

  \begin{statement}\label{statement:Hoelder-1}
    Let $a$, $b\in \reals$ with $x\le a< b\le y$ and $\floor{a}=\floor{b}$ such
    that the first $k+1$ digits of the expansions
    \begin{equation*}
      q^{\{a\}}=(a_{0}\centerdot a_{1}\ldots)_{q},\quad
      q^{\{b\}}=(b_{0}\centerdot b_{1}\ldots)_{q}
    \end{equation*}
    coincide, i.e., $a_i=b_i$ for $0\leq i\leq k$. Then
    \begin{equation*}
      \rvert\Psi_{1}(b)-\Psi_{1}(a)\lvert=\bigOh(\lvert y - x\rvert^{\alpha}).
    \end{equation*}
  \end{statement}
  \begin{proof}
    Lemma~\ref{lemma:psi_1-explicit} yields
    \begin{align*}
      \lvert \Psi_{1}(b)-\Psi_{1}(a)\rvert&\leq\lvert e_{\T}\rvert\lvert\fpart
      b-\fpart a\rvert+q^{-d\fpart b}\\
      &\hspace{2em}\cdot\sum_{l\in\calP}\sum_{m\geq 0}q^{-dm}\lvert
      f_{l}((b_{0}\ldots b_{m})_{q})-f_{l}((a_{0}\ldots
      a_{m})_{q})\rvert\\
      &\quad+\lvert q^{-d\fpart b}-q^{-d\fpart
        a}\rvert\sum_{l\in\calP}\sum_{m\geq 0}q^{-dm}\lvert f_{l}((a_{0}\ldots a_{m})_{q})\rvert\\
      &\leq \lvert
      e_{\T}\rvert\lvert\fpart{b}-\fpart{a}\rvert\\
      &\hspace{1em}+\sum_{l\in\calP}\sum_{m>k}q^{-dm}(\lvert
      f_{l}((b_{0}\ldots b_{m})_{q})\rvert+\lvert f_{l}((a_{0}\ldots a_{m})_{q})\rvert)\\
      &\hspace{1em}+\lvert q^{-d\fpart{b}}-q^{-d\fpart{a}}\rvert
      \sum_{l\in\calP}\sum_{m\geq 0} q^{-dm}\lvert f_{l}((a_{0}\ldots a_{m})_{q})\rvert
    \end{align*}
    because the summands for $m\leq k$ cancel in the first sum as the first $k+1$
    digits coincide. By using the estimates
    \begin{align*}
      \lvert\fpart{b}-\fpart{a}\rvert&\le
      \lvert\fpart{b}-\fpart{a}\rvert^{\alpha}=\lvert b - a\rvert^{\alpha},\\
      \lvert q^{-d\fpart{b}}-q^{-d\fpart{a}}\rvert&=\bigOh(\lvert
      b - a\rvert^{\alpha}),\\
      \lvert f_{l}((b_{0}\ldots b_{m})_{q})\rvert&= \bigOh(q^{(d-1)m}m)
    \end{align*}
    (see Lemma~\ref{lemma:psi_1-explicit} for the last estimate),
    we obtain
    \begin{align*}
      \lvert \Psi_{1}(b)-\Psi_{1}(a)\rvert&=\bigOh\biggl(\lvert b - a\rvert^\alpha + \sum_{m>k}mq^{-m} +
      \lvert b - a\rvert^\alpha\biggr) \\
      &= \bigOh(\lvert b - a\rvert^\alpha + kq^{-k}) =
      \bigOh(\lvert b - a\rvert^\alpha + q^{-\alpha k})\\
      &=\bigOh(\lvert b - a\rvert^\alpha +
      \lvert y - x\rvert^\alpha)  = \bigOh(\lvert y - x\rvert^\alpha).
    \end{align*}
    Here, \eqref{eq:Hoelder-q_-k-bound} has been used in the penultimate step.
  \end{proof}

  We now use the continuity of $\Psi_1$ and Statement~\ref{statement:Hoelder-1}
  to remove the condition on coinciding digits from
  Statement~\ref{statement:Hoelder-1}.

  \begin{statement}\label{statement:Hoelder-2}
    Let $a$, $b\in \reals$ with $x\le a< b\le y$ and $\floor{a}=\floor{b}$. Then
    \begin{equation*}
      \rvert\Psi_{1}(b)-\Psi_{1}(a)\lvert=\bigOh(\lvert y - x\rvert^{\alpha}).
    \end{equation*}
  \end{statement}
  \begin{proof}
    We write the expansions of $q^{\{a\}}$ and $q^{\{b\}}$ as 
    \begin{equation*}
      q^{\{a\}}=(a_{0}\centerdot a_{1}\ldots)_{q},\quad
      q^{\{b\}}=(b_{0}\centerdot b_{1}\ldots)_{q}.
    \end{equation*}
    This yields
    \begin{equation*}
      0< q^{\{b\}} - q^{\{a\}} = \frac1{q^{\lfloor a\rfloor}} (
      q^b-q^a)\le 
      q^b-q^a\le 
      q^y-q^x< q^{-k}.
    \end{equation*}
    Thus 
    \begin{equation*}
      0\le (b_{0}\ldots b_{k})_{q} - (a_{0}\ldots a_{k})_{q}\le 1.
    \end{equation*}
    If $(b_{0}\ldots b_{k})_{q} = (a_{0}\ldots a_{k})_{q}$, the result follows
    immediately from Statement~\ref{statement:Hoelder-1}. Otherwise, we have
    \begin{equation}\label{eq:Hoelder-difference-1}
      (b_{0}\ldots b_{k})_{q} = (a_{0}\ldots a_{k})_{q}+1.
    \end{equation}
    For $m\ge 0$, define $z$ and $z_m$ by
    $\floor{z}=\floor{z_{m}}=\floor{a}=\floor{b}$ and
    \begin{align*}
      q^{\{z\}}&=(b_{0}\centerdot b_{1}\ldots b_{k})_{q},\\
      q^{\{z_{m}\}}&=(a_{0}\centerdot a_{1}\ldots a_{k}(q-1)^{m})_{q}.
    \end{align*}
    Then $\lim_{m\rightarrow \infty}z_{m}=z$ because of
    \eqref{eq:Hoelder-difference-1}.

    By construction of $z$ and $z_{m}$, we have $a<z_{m}<z\le b$ for sufficiently
    large $m$.

    By continuity of $\Psi_1$,
    \begin{equation}\label{eq:psi-cont-estimate-2}
      \lvert\Psi_{1}(z)-\Psi_{1}(z_{m})\rvert\leq \lvert y - x\rvert^{\alpha}
    \end{equation}
    holds for sufficiently large $m$.

    This yields
    \begin{align*}
      \lvert\Psi_{1}(b)-\Psi_{1}(a)\rvert&\leq\lvert\Psi_{1}(b)-\Psi_{1}(z)\rvert+\lvert\Psi_{1}(z)-\Psi_{1}(z_m)\rvert\\
      &\quad+\lvert\Psi_{1}(z_m)-\Psi_{1}(a)\rvert.
    \end{align*}
    The third summand can be bounded by Statement~\ref{statement:Hoelder-1}
    (for $a$ and $z_m$) and the second by \eqref{eq:psi-cont-estimate-2}.
    The first summand is either $0$ or can be bounded by Statement~\ref{statement:Hoelder-1} (for $z$ and $b$).
  \end{proof}

  To finally prove \eqref{eq:hoelder-ineq} for sufficiently small $y-x<1$ , we only have to remove the
  assumption $\floor{a}=\floor{b}$ from Statement~\ref{statement:Hoelder-2}.
  We use the idea of the proof of Statement~\ref{statement:Hoelder-2} once
  more.

  Assume that $\floor{y}>\floor{x}$. By our assumption $y<x+1$, this amounts to
  $\floor{y}=\floor{x}+1$. For $m\ge 0$, define $z$ and $z_m$ by $z=\lfloor y\rfloor$,
  $\floor{z_{m}}=\floor{x}$ and $q^{\{z_{m}\}}=((q-1)\centerdot
  (q-1)^{m})_{q}$.
  Then $\lim_{m\rightarrow\infty}z_{m}=z$.  By
  continuity of $\Psi_1$, we have 
  \begin{equation}\label{eq:psi-cont-estimate}
    \lvert\Psi_{1}(z)-\Psi_{1}(z_{m})\rvert\leq \lvert y - x\rvert^{\alpha}
  \end{equation}
  and $x<z_m<z\le y$ for sufficiently large $m$.

  Then, this yields
    \begin{align*}
      \lvert
      \Psi_{1}(y)-\Psi_{1}(x)\rvert&\leq\lvert\Psi_{1}(y)-\Psi_{1}(z)\rvert+\lvert\Psi_{1}(z)-\Psi_{1}(z_{m})\rvert\\
      &\quad+\lvert\Psi_{1}(z_m)-\Psi_{1}(x)\rvert.
    \end{align*}
    The third summand can be bounded by Statement~\ref{statement:Hoelder-2} for
    $x$ and $z_m$ and
    the second by \eqref{eq:psi-cont-estimate}. The first vanishes or can be
    bounded by Statement~\ref{statement:Hoelder-2} for $z$ and $y$.

    This yields
    \begin{equation*}
      \lvert\Psi_{1}(y)-\Psi_{1}(x)\rvert =\bigOh(\lvert y-x\rvert^{\alpha}).
    \end{equation*}
    Therefore, \eqref{eq:hoelder-ineq}  is satisfied with a suitable positive
    constant $C$ for $y-x<\eps$ for some $\eps>0$.

 Assume $y-x\ge \eps$. As
  $\Psi_{1}$ is continuous and periodic,
  $\lvert\Psi_{1}(y)-\Psi_{1}(x)\rvert$ is bounded. Thus, \eqref{eq:hoelder-ineq}
  holds for a suitable positive constant $C$ for $|y-x|\ge \eps$.

 Therefore, the function $\Psi_{1}$ is H\"older continuous of order $\alpha<1$.
\end{proof}

\subsection{Limiting distribution}\label{sec:clt}

Finally, we can prove the parts of Theorem~\ref{thm:asydist} concerning the
approximation of the distribution function and the central limit theorem.

\begin{proof}
  To prove that the distribution function can be approximated by a Gaussian
mixture, we use the Berry-Esseen inequality (cf., for instance,
\cite[Theorems~IX.5]{Flajolet-Sedgewick:ta:analy}) to estimate the difference
between distribution functions. The proof follows the proof of Hwang's Quasi-Power Theorem~\cite{Hwang:1998}. First, we describe the two corresponding characteristic functions.

Let $\hat g_{N}(t)$ be the characteristic function of a mixture of Gaussian or
degenerate distributions with weights
$\lambda_{j}$, means $a_{j}\sqrt{\log_{q}N}$ and variances $b_{j}$ for
$j=1,\ldots,c$, that is
\begin{equation*}
  \hat g_{N}(t)=\sum_{j=1}^{c}\lambda_{j}\exp\Big(ia_{j}\sqrt{\log_{q}N}t-\frac{b_{j}}{2}t^{2}\Big)
\end{equation*}
with $a_{j}$, $b_{j}$ and $\lambda_{j}$ defined in~\eqref{eq:const-of-thm}.

By Lemma~\ref{lem:char-func}, the characteristic function $\hat f_{N}(t)$ of $\T(\bfn)/\sqrt{\log_{q}N}$ is
\begin{align*}
  \hat
  f_{N}(t)&=\sum_{j=1}^{c}\exp\Big(ia_{j}\sqrt{\log_{q}N}t-\frac{b_{j}}{2}t^{2}+\bigOh\Big(\frac{t^{3}}{\sqrt{\log
      N}}\Big)\Big)\\
  &\quad\cdot\sum_{l\in\mathcal
  P}e^{\frac{2\pi
    il}{p}\log_{q}N}\Psi_{lj}\Big(\log_{q}N,\frac{t}{\sqrt{\log_{q}N}}\Big)+R\Big(N,\frac{t}{\sqrt{\log
  N}}\Big)
\end{align*}
for $t\log_{q}^{-\frac 12}N$ in a fixed neighborhood of $0$.

Because of~\eqref{eq:psi-lambda} and $R(N,0)=0$ (see Lemma~\ref{lem:char-func}), we have
\begin{align*}
  \hat
  f_{N}(t)&=\sum_{j=1}^{c}\exp\Big(ia_{j}\sqrt{\log_{q}N}t-\frac{b_{j}}{2}t^{2}\Big)\exp\Big(\bigOh\Big(\frac{t^{3}}{\sqrt{\log
    N}}\Big)\Big)\\
&\quad\cdot\Big(\lambda_{j}+\bigOh\Big(\frac{t}{\sqrt{\log
    N}}\Big)\Big)+\bigOh\Big(N^{-\xi}t\sqrt{\log
  N}\Big).
\end{align*}

Now we use the
inequality $\lvert e^{w}-1\rvert\leq \lvert w\rvert e^{\lvert w\rvert}$, valid for
all complex numbers $w$, to obtain
\begin{multline}\label{eq:approximation-characteristic-function}
  \Big\lvert\frac{1}{t}(\hat f_{N}(t)-\hat
  g_{N}(t))\Big\rvert=\\\sum_{j=1}^{c}\bigOh\Big(\Big(\frac{t^{2}+1}{\sqrt{\log
    N}}\Big)\exp\Big(-\frac{b_{j}}{2}t^{2}+\bigOh\Big(\frac{t^{3}}{\sqrt{\log
    N}}\Big)\Big)\Big)\\
\quad+\bigOh(N^{-\xi}\log^{-\frac 12}N)
\end{multline}
for $t\log_{q}^{-\frac 12}N$ in a small neighborhood of $0$.

From now on, we assume that $b_j\neq 0$.
There is a small neighborhood of $0$ for $t\log_{q}^{-\frac 12}N$ such that
\begin{equation*}
\bigOh\Big(\exp\Big(-\frac{b_{j}}{2}t^{2}+\bigOh\Big(\frac{t^{3}}{\sqrt{\log
    N}}\Big)\Big)\Big)= \bigOh\Big(\exp\Big(-\frac {b_{j}}{4}t^{2}\Big)\Big)
\end{equation*}
holds.

This yields
\begin{align*}
\Big\lvert\frac{1}{t}(\hat f_{N}(t)-\hat
  g_{N}(t))\Big\rvert&=\sum_{j=1}^{c}\bigOh\Big(\exp\Big(-\frac {b_{j}}{4}t^{2}\Big)\frac{t^{2}+1}{\sqrt{\log N}}\Big)+\bigOh(N^{-\xi}\log^{\frac 12}N).
\end{align*}

Now, the Berry-Esseen inequality with $T=c\sqrt{\log_q N}$ for a small constant
$c>0$ (cf., for instance,
\cite[Theorem~IX.5]{Flajolet-Sedgewick:ta:analy}) implies that 
\begin{equation*}
  \sup_{x\in\mathbb R}\,\lvert F_{N}(x)-G_{N}(x)\rvert=
  \bigOh\Big(\frac{1}{\sqrt{\log N}}\Big)
\end{equation*}
where $F_{N}$ is the cumulative distribution function of $\T(\bfn)$ and $G_{N}$ is the
cumulative distribution function of the mixture of Gaussian distributions.

If all $a_{j}$ are equal and $b_j\ge 0$, $G_{N}$ is the
distribution function of a mixture of normal (or degenerate) distributions with mean $e_{\T}\sqrt{\log_{q}N}$
and variances $b_{j}\geq 0$. After subtracting the mean, 
\eqref{eq:approximation-characteristic-function} converges to $0$.
 Thus, 
\begin{equation*}
  \frac{\T(\bfn)-\Expect(\T(\bfn))}{\sqrt{\log_{q}N}}
\end{equation*}
converges in distribution. If all $b_j>0$, then the same estimates as above yield
the speed of convergence.
\end{proof}

This completes the
proof of Theorem~\ref{thm:asydist}.

\section{Fourier Coefficients --- Proof of Theorem~\ref{thm:fourier}}
\label{sec:four-coeff-gener}
This section contains the proof of the theorem about the Fourier coefficients. First, we investigate some Dirichlet series which we will use later. Then,
we prove the formulas given in Theorem~\ref{thm:fourier}. We use the H\"older
condition for $\Psi_{1}$ to prove that its Fourier
series converges.

\begin{lemma}\label{lem:dirichlet-L}
  The Dirichlet series 
  \begin{equation*}
     L(z)=\sum_{r\geq 1}\lfloor\log_{q}r\rfloor(r^{d}-(r-1)^{d}) r^{-z}
  \end{equation*}
is meromorphic in $\Re z>d-1$ with poles in $z=d+\frac{2\pi il}{\log q}$ for
$l\in\mathbb Z$. The main part at $z=d$ is
\begin{align*}
  &\frac{d}{(z-d)^{2}\log q} - \frac{d}{2(z-d)}\end{align*}
and, for $l\neq 0$, the residue at $z=d+\frac{2\pi il}{\log q}$ is
  $\frac{d}{2\pi il}$.

\end{lemma}

\begin{proof}
First, we use the binomial theorem to obtain
\begin{equation}\label{eq:L-after-binom-thm}
  L(z)=dL_{1}(z-d+1)-\sum_{j=0}^{d-2}\binom{d}{j}(-1)^{d-j}L_{1}(z-j)
\end{equation}
with $L_{1}=\sum_{r\geq 1}\floor{\log_{q}r}r^{-z}$.
The Dirichlet series $L_{1}(z)$ is holomorphic for $\Re z>1$. Thus, the second
summand in \eqref{eq:L-after-binom-thm} is holomorphic for $\Re z>d-1$. To
obtain the expansion of $L(z)$ at $z$ with $\Re z=d$, we investigate the
Dirichlet series $L_{1}(z)$ at $\Re z=1$.

Let $k\ge 0$ be an integer. We use Euler-Maclaurin
  summation with $f(x)=kx^{-z}$
  to obtain
  \begin{align*}
    \sum_{q^k\le r<q^{k+1}}\frac{\floor{\log_q
        r}}{r^z}&=\int_{q^k}^{q^{k+1}}kx^{-z}\,dx-\frac{k}2(q^{-(k+1)z}-q^{-kz})\\
    &\quad-kz\int_{q^k}^{q^{k+1}}B_1(\{x\})x^{-z-1}\,dx\\
    &=\frac{1}{1-z}(kq^{(k+1)(1-z)}-kq^{k(1-z)})\\
    &\quad-\frac{1}2(kq^{-(k+1)z}-kq^{-kz})\\
    &\quad-z\int_{q^k}^{q^{k+1}}B_1(\{x\})x^{-z-1}\floor{\log_q(x)}\,dx
  \end{align*}
where $B_{1}(x)$ is the first Bernoulli polynomial.
  For $\Re z>1$, summation over $k\ge 0$ yields
  \begin{align*}
    L_{1}(z)&=\frac{1}{1-z}\sum_{k\ge 1}q^{k(1-z)}((k-1)-k)-\frac12\sum_{k\ge
      1}q^{-zk}((k-1)-k)\\
    &\quad
    -z\int_1^{\infty}B_1(\{x\})x^{-z-1}\floor{\log_q(x)}\,dx\\
    &=\frac{1}{z-1}\frac{1}{q^{z-1}-1}+\frac12\frac{1}{q^z-1}
    -z\int_1^{\infty}B_1(\{x\})x^{-z-1}\floor{\log_q(x)}\,dx.
  \end{align*}
  The second summand and the integral are clearly holomorphic for
  $\Re z>0$. Thus, $L_{1}(z)$ can be continued meromorphically to $\Re z>0$ with
  poles coming from the first summand.

  The expansion around $z=1$ is
  \begin{align*}
    \frac{1}{z-1}\frac{1}{q^{z-1}-1}+O(1)
\ifdetails{
&=\frac{1}{z-1}\frac{1}{\log q (z-1)+\frac12 \log^2q (z-1)^2+O((z-1)^3)}+O(1)\\
    &=\frac{1}{\log q(z-1)^2}\frac{1}{1+\frac12 \log q (z-1) + O((z-1)^2)}+O(1)\\
    &=\frac{1}{\log q(z-1)^2}\left(1-\frac12 \log q(z-1)+ O((z-1)^2)\right)\\}\fi
    &=\frac{1}{ (z-1)^2\log q}-\frac{1}{2(z-1)}+O(1).
  \end{align*}

Thus,  by \eqref{eq:L-after-binom-thm}, we obtain the main part and the residues of $L(z)$ at $z=d+\frac{2\pi
  il}{\log q}$ for $l\in\mathbb Z$ as stated in the
lemma.
\end{proof}

\begin{lemma}\label{lem:dirichlet-Z}
  The Dirichlet series
  \begin{equation*}
    Z(z)=\sum_{r\geq1}(r^{d}-(r-1)^{d})r^{-z}
  \end{equation*}
  is meromorphic in $\mathbb C$ with simple poles in $z=j$,
  $j\in\{1,\ldots,d\}$ with residues $\binom{d}{j-1}(-1)^{d-j}$.
\end{lemma}
\begin{proof}
  The binomial theorem yields
  \begin{equation*}
    Z(z)=\sum_{j=0}^{d-1}\binom{d}{j}(-1)^{d-j+1}\zeta(z-j),
  \end{equation*}
where $\zeta$ is the Riemann $\zeta$-function.
 The result follows from the
unique pole of $\zeta(z)$ at $z=1$ with residue $1$.
\end{proof}

Denote by
  $\zeta(z,\alpha)$ the Hurwitz $\zeta$-function
  \begin{align*}
    \zeta(z,\alpha)=\sum_{r>-\alpha}(r+\alpha)^{-z}.
  \end{align*}
Furthermore $\psi$ is the digamma function.

\begin{lemma}\label{lem:dirichlet-J}
  For $0\leq\alpha< 1$ and and an integer $0\leq j\leq d-1$, the Dirichlet series
  \begin{equation*}
    J(z,\alpha, j)=\sum_{r\geq1}r^{j}(r+\alpha)^{-z}
  \end{equation*}
  is analytic for $\Re z>j+1$. For $j=d-1$, it is meromorphic for $\Re z>d-1$ with a simple pole at $z=d$ with expansion
  \begin{equation}\label{eq:J-expansion}
    \begin{aligned}
      J(z,\alpha, d-1)&=\frac{1}{z-d}-\psi(\alpha+[\alpha=0])-[\alpha>0\wedge
      d=1]\alpha^{-1}\\
      &\quad+\sum_{k=0}^{d-2}\binom{d-1}{k}(-\alpha)^{d-1-k}\zeta(d-k,\alpha)+\bigOh(z-d).
    \end{aligned}
  \end{equation}
\end{lemma}
\begin{proof}
  As $r^{j}(r+\alpha)^{-z}=\bigOh(r^{j-\Re z})$, $J$ is analytic for $\Re
  z>j+1$. Now, let $j=d-1$.

  The binomial theorem yields
  \begin{align*}
    J(z, \alpha, d-1) &= \sum_{r\ge 1}(r+\alpha - \alpha)^{d-1}(r+\alpha)^{-z}\\
    &= \sum_{k=0}^{d-1}\binom{d-1}k(-\alpha)^{d-1-k}\sum_{r\ge 1}(r+\alpha)^{-(z-k)}\\
    &= \sum_{k=0}^{d-1}\binom{d-1}k(-\alpha)^{d-1-k}\big(\zeta(z-k, \alpha)-[\alpha>0]\alpha^{-z+k}\big)\\
    &= \zeta(z-d+1,
    \alpha)+\sum_{k=0}^{d-2}\binom{d-1}k(-\alpha)^{d-1-k}\zeta(z-k,
    \alpha)\\
    &\quad -[\alpha>0\wedge d=1]\alpha^{-z}.
  \end{align*}
  Using the expansion (cf. \cite[p.\ 271]{Whittaker-Watson:1963})
  \begin{equation*}
    \zeta(z, \alpha)= \frac{1}{z-1} -\psi(\alpha+[\alpha=0])+\bigOh(z-1)
  \end{equation*} 
yields \eqref{eq:J-expansion}.
\end{proof}

\begin{lemma}\label{lem:dirichlet-B}
  Let $k\in\mathbb Z$. The Dirichlet series
  \begin{equation*}
B(z)=\bfw_{k}^{\top}\sum_{r=1}^{\infty}\left(\bfB_{\emptyset}(r+1)-2\bfB_{\emptyset}(r)+\bfB_{\emptyset}(r-1)\right)r^{-z}
  \end{equation*}
  is analytic for $\Re z>d-1$.
\end{lemma}
\begin{proof}
 By the definition \eqref{eq:bfB-def}, we have
  \begin{equation}\label{eq:B-sum-of-Bs}
\bfB_{\emptyset}(r+1)-\bfB_{\emptyset}(r)=\sum_{\emptyset\neq
    C\subseteq\{1,\ldots, d\}}\bfB_{C}(r),
\end{equation}
which can be bounded by $\|\bfB_{C}(r)\|=\bigOh(r^{d-1}\log r)$. Thus,
\begin{equation*}
  B(z)=\bfw_{k}^{\top}\sum_{\emptyset\neq C\subseteq\{1,\ldots, d\}}\sum_{r\geq 1}(\bfB_{C}(r)-\bfB_{C}(r-1))r^{-z}
\end{equation*}
which converges for $\Re z>d-1$ by \cite[Theorem~8.1]{Apostol:1976:modul-funct}.
\end{proof}

The
vector-valued functions \begin{math}\bfH_{C}(z)\end{math} are defined by the Dirichlet series
\begin{equation}\label{eq:dirichlet-H}
\bfH_{C}(z)=\sum_{r\geq 1}\bfB_{C}(r)r^{-z}.
\end{equation}
By \eqref{eq:dirichlet-H-definition} and \eqref{eq:B-sum-of-Bs}, this yields
\begin{equation}\label{eq:H-sum-H-C}
\bfH(z)=\sum_{\emptyset\neq C\subseteq\{1,\ldots,d\}}\bfH_{C}(z)=\sum_{r\geq1}(\bfB_{\emptyset}(r+1)-\bfB_{\emptyset}(r))r^{-z}.
\end{equation}

Next, we investigate the Dirichlet series \begin{math}\bfH_{C}\end{math}. In
particular, we determine its behavior at \begin{math}z=d+\chi_k\end{math} and provide an infinite
functional equation to compute its residues at these points. This will finally
give us the residues of $\bfH$ in~\eqref{eq:gen-fourier-coeff}. We use a
similar method as Grabner and Hwang in~\cite{Grabner-Hwang:2005:digit}. 

For this infinite recursion, define
 \begin{equation}\label{eq:bfdelta-C-D}
   \bfdelta_{C,D}^{\eps}=\sum_{\substack{\beta_{i}=0\\i\not\in C\cup
     D}}^{q-1}\sum_{\substack{\beta_{i}=0\\i\in
       D}}^{\eps-1}\sum_{\substack{\beta_{i}=\eps\\i\in C}}\bfdelta_{\bfbeta},
 \end{equation}
in analogy to the definition
 of $M_{C,D}^{\eps}$. As before, the $s$-th entry of $\bfdelta_{\bfeps}$ is the output
 label of the transition starting in $s$ with input label $\bfeps$.
Then, $\bfdelta=\bfdelta_{\emptyset,\emptyset}^{\eps}$ holds independently of
$\eps$. Furthermore,
$\bfdelta_{C,D}^{\eps}=\left.\frac{d}{dt}M_{C,D}^{\eps}\bfones\right|_{t=0}$
by \eqref{eq:delta-eps-is-derivative}.
\begin{lemma}\label{lem:inf-recursion}
Let $C\neq \emptyset$.
For \begin{math}\Re z>d\end{math} and $C\neq\emptyset$, the Dirichlet series \begin{math}\bfH_{C}(z)\end{math} satisfies the following infinite
recursion
\begin{equation}\label{eq:infrecursion-general}
  \begin{aligned}
    &\Big(1-q^{-z}\sum_{\eps=0}^{q-1}M_{C,\emptyset}^{\eps}\Big)\bfH_{C}(z)=\\
    &\qquad\qquad\sum_{\eps=1}^{q-1}\bfB_{C}(\eps)\eps^{-z}+q^{-z}\sum_{\emptyset\neq
      D\subseteq C^{c}}\sum_{\eps=0}^{q-1}M_{C,D}^{\eps}\bfH_{C\cup
      D}(z)\\
    &\qquad\qquad+q^{-z}\sum_{D\subseteq
      C^{c}}\sum_{\eps=0}^{q-1}\bfdelta_{C,D}^{\eps}J\Big(z,\frac{\eps}{q},d-|D|-|C|\Big)\\
    &\qquad\qquad+\sum_{D\subseteq C^{c}}\sum_{m\geq
      1}\binom{-z}{m}q^{-z-m}\sum_{\eps=0}^{q-1}M_{C,D}^{\eps}\eps^{m}\bfH_{C\cup
      D}(z+m).
  \end{aligned}
\end{equation}

It is analytic for $\Re z>d-|C|+1$. For
$|C|=1$ and $k\neq 0$, $\bfw_{k}^{\top}\bfH_{C}$ has a possible simple pole in
$z=d+\chi_{k}$ with residue the right-hand side
of~\eqref{eq:infrecursion-general} evaluated at $z=d+\chi_{k}$ and divided by $\log q$.
For $|C|=1$, $\bfw_{0}^{\top}\bfH_{C}$ has a possible double pole with main
part
\begin{equation*}
   \frac{e_{\T}}{\log q}\frac{1}{(z-d)^{2}}+\Big(\frac{e_{\T}}{2}+\frac{h_{C}}{\log
  q}\Big)\frac{1}{z-d}
\end{equation*}
where $h_{C}$ is given in~\eqref{eq:res-rhs}.
\end{lemma}
\begin{remark}
  The infinite recursion~\eqref{eq:infrecursion-general} can be used to
  numerically compute the values of $\bfH_{C}$ and its residues at
  $z=d+\chi_{k}$ with arbitrary precision. It numerically converges fast
if the first terms of the Dirichlet series $\bfH_{C}$ are computed explicitly.
\end{remark}
\begin{proof}As $\bfB_{C}(r)=\bigOh(r^{d-\lvert C\rvert}\log r)$, the
  Dirichlet series $\bfH_{C}$ is analytic for $\Re z>d-\lvert C\rvert+1$.

By multiplying \eqref{eq:recursion-general} with $\bfu(t)$, differentiating
with respect to $t$ at $t=0$ and using~\eqref{eq:bfB_C-is-derivative},
\eqref{eq:matrices-M} and \eqref{eq:bfdelta-C-D}, we obtain the recursion
\begin{equation}
  \label{eq:recursion-B-C}
  \bfB_{C}(qr+\eps)=\sum_{D\subseteq
    C^{c}}M_{C,D}^{\eps}\bfB_{C\cup D}(r)+\bfdelta_{C,D}^{\eps}r^{d-|D|-|C|}
\end{equation}
for $C\neq\emptyset$, $\{1,\ldots, d\}$ and $qr+\eps\geq
0$. By \eqref{eq:rec-g}, this recursion is
also valid for $C=\{1,\ldots, d\}$ and $qr+\eps>0$.

By \eqref{eq:recursion-B-C}, we have
\begin{align*}
\bfH_{C}(z)&=\sum_{\eps=1}^{q-1}\bfB_{C}(\eps)\eps^{-z}+\sum_{\eps=0}^{q-1}\sum_{r\geq
  1}\bfB_{C}(qr+\eps)(qr+\eps)^{-z}\\
&=\sum_{\eps=1}^{q-1}\bfB_{C}(\eps)\eps^{-z}+\\
&\quad \sum_{D\subseteq C^{c}}\sum_{\eps=0}^{q-1}\sum_{r\geq
  1}(M_{C,D}^{\eps}\bfB_{C\cup D}(r)+\bfdelta_{C,D}^{\eps}r^{d-\lvert
  D\rvert-\lvert C\rvert})\\&\qquad\qquad\qquad\qquad\cdot q^{-z}r^{-z}\Big(1+\frac{\eps}{qr}\Big)^{-z}
\end{align*}
for  $C\neq \emptyset$. Expanding \begin{math}(1+\eps/(qr))^{-z}\end{math} as a binomial
series yields
\begin{align*}
\bfH_{C}(z)&=\sum_{\eps=1}^{q-1}\bfB_{C}(\eps)\eps^{-z}\\
&\quad+\sum_{D\subseteq
  C^{c}}\sum_{\eps=0}^{q-1}\sum_{r\geq 1}\sum_{m\geq
  0}\binom{-z}{m}M_{C,D}^{\eps}\eps^{m}q^{-z-m}\bfB_{C\cup D}(r)r^{-z-m}\\
&\quad+q^{-z}\sum_{D\subseteq
C^{c}}\sum_{\eps=0}^{q-1}\bfdelta_{C,D}^{\eps}
J\Big(z,\frac{\eps}{q},d-|D|-|C|\Big)\\
&=\sum_{\eps=1}^{q-1}\bfB_{C}(\eps)\eps^{-z}+q^{-z}\sum_{D\subseteq
  C^{c}}\sum_{\eps=0}^{q-1}M_{C,D}^{\eps}\bfH_{C\cup D}(z)\\
&\quad+q^{-z}\sum_{D\subseteq
C^{c}}\sum_{\eps=0}^{q-1}\bfdelta_{C,D}^{\eps}
J\Big(z,\frac{\eps}{q},d-|D|-|C|\Big)
\\
&\quad+\sum_{D\subseteq C^{c}}\sum_{m\geq
  1}\binom{-z}{m}q^{-z-m}\sum_{\eps=0}^{q-1}M_{C,D}^{\eps}\eps^{m}\bfH_{C\cup D}(z+m)
\end{align*}
for
\begin{math}\Re z> d\end{math} and $C\neq \emptyset$. Collecting \begin{math}\bfH_{C}(z)\end{math} on the left-hand side results in~\eqref{eq:infrecursion-general}.

To compute the residues of $\bfw_{k}^{\top}\bfH_{C}$ for $|C|=1$ at
$z=d+\chi_{k}$, note that $\sum_{\eps=0}^{q-1}M_{C,\emptyset}^{\eps}=M$
holds
independently of $C$.

We multiply~\eqref{eq:infrecursion-general}
with the left eigenvector \begin{math}\bfw_{k}^\top\end{math} which results in 
    \begin{equation}\label{eq:infrecursion-again-general}
      \begin{aligned}
        &\Big(1-q^{d-z}\exp\Big(\frac{2\pi
          ik}{p}\Big)\Big)\bfw_{k}^{\top}\bfH_{C}(z)=\\
        &\hspace{4em}\bfw_{k}^{\top}\sum_{\eps=1}^{q-1}\bfB_{C}(\eps)\eps^{-z}\\
        &\hspace{4em}+q^{-z}\bfw_{k}^{\top}\sum_{\emptyset\neq
          D\subseteq C^{c}}\sum_{\eps=0}^{q-1}M_{C,D}^{\eps}\bfH_{C\cup
          D}(z)\\
        &\hspace{4em}+q^{-z}\bfw_{k}^{\top}\sum_{D\subseteq
          C^{c}}\sum_{\eps=0}^{q-1}\bfdelta_{C,D}^{\eps}
        J\Big(z,\frac{\eps}{q},d-|D|-1\Big) \\
        &\hspace{4em}+\bfw_{k}^{\top}\sum_{D\subseteq
          C^{c}}\sum_{m\geq
          1}\binom{-z}{m}q^{-z-m}\sum_{\eps=0}^{q-1}M_{C,D}^{\eps}\eps^{m}\bfH_{C\cup
          D}(z+m).
      \end{aligned}
\end{equation}

As $|C\cup D|\geq 2$ or $\Re z+m>d$, all $\bfH_{C\cup D}$ used on right-hand side
of~\eqref{eq:infrecursion-again-general} are well defined for $\Re z>
d-1$. The Dirichlet series $J$ have simple poles at $z=d$ for $|C|=1$ and
$D=\emptyset$ (Lemma~\ref{lem:dirichlet-J}). Thus the right-hand
side of~\eqref{eq:infrecursion-again-general} is meromorphic for $\Re z> d-1$ with a simple
pole at $z=d$.

The factor $1-q^{d-z}\exp(\frac{2\pi ik}{p})$ has a zero exactly for
$z=d+\chi_{k}$, $k\in\mathbb Z$. Thus for $k\neq 0$, $\bfw_{k}^{\top}\bfH_{C}$
has a possible simple pole at $z=d+\chi_{k}$. Its residue is the right-hand
side of \eqref{eq:infrecursion-again-general} evaluated at $z=d+\chi_{k}$
divided by $\log q$.

If $k=0$, we have $z=d$. In this case the
expansion of the right-hand side of~\eqref{eq:infrecursion-again-general} is
\begin{equation*}
  \frac{e_{\T}}{z-d}+h_{C}+\bigOh(z-d)
\end{equation*}
with
\begin{equationaligned}
  h_{C}&=-e_{\T}\log
  q-q^{-d}\bfw_{0}^{\top}\sum_{\eps=0}^{q-1}\bfdelta_{C,\emptyset}^{\eps}\psi\Big(\frac{\eps}{q}+[\eps=0]\Big)\label{eq:res-rhs}\\
  &\quad-[d=1]\bfw_{0}^{\top}\sum_{\eps=1}^{q-1}\bfdelta_{C,\emptyset}^{\eps}\eps^{-1}\OBHnotag\\
&\quad+q^{-d}\bfw_{0}^{\top}\sum_{\eps=0}^{q-1}\bfdelta_{C,\emptyset}^{\eps}\sum_{k=0}^{d-2}\binom{d-1}{k}\Big(-\frac{\eps}{q}\Big)^{d-1-k}\zeta\Big(d-k,\frac{\eps}{q}\Big)\OBHnotag\\
&\quad+\bfw_{0}^{\top}\sum_{\eps=1}^{q-1}\bfB_{C}(\eps)\eps^{-d}+q^{-d}\bfw_{0}^{\top}\sum_{\emptyset\neq
D\subseteq C^{c}}\sum_{\eps=0}^{q-1}M_{C,D}^{\eps}H_{C\cup D}(d)\OBHnotag\\
&\quad+q^{-d}\bfw_{0}^{\top}\sum_{\emptyset\neq D\subseteq
  C^{c}}\sum_{\eps=0}^{q-1}\bfdelta_{C,D}^{\eps}J\Big(d,\frac{\eps}{q},d-|D|-1\Big)\OBHnotag\\
&\quad+\bfw_{0}^{\top}\sum_{D\subseteq
  C^{c}}\sum_{m\geq1}\binom{-d}{m}q^{-d-m}\sum_{\eps=0}^{q-1}M_{C,D}^{\eps}\eps^{m}\bfH_{C\cup
D}(d+m)\OBHnotag
\end{equationaligned}
where we used the expansion of $J$ in Lemma~\ref{lem:dirichlet-J},
$\bfdelta=\sum_{\eps=0}^{q-1}\bfdelta_{C,\emptyset}^{\eps}$ and \eqref{eq:bfw_bfdelta}.
\end{proof}
From the previous lemma and \eqref{eq:H-sum-H-C}, the residues of the
Dirichlet function $\bfH$ follow. Only $\bfH_{C}$ with $|C|=1$ contribute as
all other summands are holomorphic.
\begin{lemma}\label{lem:dirichlet-H}
  The Dirichlet function $\bfH$ is meromorphic in $\Re z>d-1$ with possible
  simple poles at $z=d+\chi_{k}$, $k\neq 0$ and a possible double pole at
  $z=d$. 

The residue at $z=d+\chi_{k}$, $k\neq 0$ is
  \begin{align*}
    &\frac1{\log
      q}\sum_{j=1}^{d}\Bigg(\sum_{\eps=1}^{q-1}\bfB_{\{j\}}(\eps)\eps^{-d-\chi_k}\\
&+q^{-d-\chi_k}\sum_{\emptyset\neq
      D\subseteq \{j\}^{c}}\sum_{\eps=0}^{q-1}M_{\{j\},D}^{\eps}\bfH_{\{j\}\cup
      D}(d+\chi_k)\\
    &+q^{-d-\chi_k}\sum_{D\subseteq
      \{j\}^{c}}\sum_{\eps=0}^{q-1}\bfdelta_{\{j\},D}^{\eps}J\Big(d+\chi_{k},\frac{\eps}{q},d-|D|-1\Big)\\
    &+\sum_{D\subseteq \{j\}^{c}}\sum_{m\geq
      1}\binom{-d-\chi_k}{m}q^{-d-m-\chi_k}\\&\qquad\qquad\qquad\qquad\cdot\sum_{\eps=0}^{q-1}M_{\{j\},D}^{\eps}\eps^{m}\bfH_{\{j\}\cup
      D}(d+m+\chi_k)\Bigg).
  \end{align*}

The main part at $z=d$ is
\begin{equation*}
   \frac{e_{\T}d}{\log q}\frac{1}{(z-d)^{2}}+\Big(\frac{e_{\T}d}{2}+\sum_{j=1}^{d}\frac{h_{\{j\}}}{\log
  q}\Big)\frac{1}{z-d}
\end{equation*}
where $h_{\{j\}}$ is defined in \eqref{eq:res-rhs}.
\end{lemma}

Now we can prove the formulas for the Fourier coefficients.
\begin{proof}[Proof of Theorem~\ref{thm:fourier}]
  The periodic fluctuation \begin{math}\Psi_{1}\end{math} of the expected
  value is a \begin{math}p\end{math}-periodic function. We use the explicit
  expression of $\Psi_{1}$ given in Lemma~\ref{lemma:psi_1-explicit}.

  Due to absolute convergence, the \begin{math}k\end{math}-th Fourier
  coefficient of \begin{math}\Psi_{1}(x)\end{math} is
  \begin{align*}
    c_{k}&=\frac1p\int_{0}^{p}\Psi_{1}(x)e^{-\frac{2\pi i k}{p}x}\,dx\\
    &=-\frac{e_{\T}}{p}\int_{0}^{p}\{x\}e^{-\frac{2\pi i k}{p}
      x}\,dx-\sum_{l\in\calP}\sum_{m=0}^{\infty}q^{-dm}e^{-\frac{2\pi i
        lm}{p}}I_{l,m}\end{align*} with
  \begin{align*}
    I_{l,m}&=\frac 1p\int_{0}^{p}q^{-d\{x\}}\exp\Big(\frac{2\pi i l}{p}\lfloor
    x\rfloor-\frac{2\pi i k}{p} x\Big)f_{l}((x_{0}\ldots x_{m})_{q})\,dx
  \end{align*}
  and $q^{\{x\}}=(x_{0}\centerdot x_{1}\ldots)_{q}$.  The value of the first
  integral is given by \begin{math}-\frac{e_{\T}}{2}\end{math}
  for \begin{math}k=0\end{math}, and
  \begin{math}[k\equiv 0\bmod p]\frac{e_{\T}}{\chi_{k}\log q}\end{math}
  otherwise. Thus, we focus on the second integral $I_{l,m}$.

  First, we partition the interval $[0,p)$ into intervals $[r,r+1)$ for $r=0$,
  \dots, $p-1$. After simplifying the sum of $p$-th roots of unity, we
  obtain
  \begin{equation*}
    I_{l,m}= [k\equiv l\bmod p]\int_{0}^{1}q^{-dx}f_{l}((x_{0}\ldots
    x_{m})_{q})e^{-\frac{2\pi i k}{p} x}\,dx.
  \end{equation*}

  After partitioning the interval \begin{math}[0,1)\end{math} into the intervals
  \begin{math}[\log_{q}r-m,\allowbreak\log_{q}(r+1)-m)\end{math}
  for \begin{math}r=q^{m}\end{math}, \dots, \begin{math}q^{m+1}-1\end{math},
  the function \begin{math}f_{l}((x_{0}\ldots x_{m})_{q})\end{math} is
  constant on the interval of integration. Therefore, we obtain

\begin{multline*}
  \sum_{l\in\calP}\sum_{m=0}^{\infty}q^{-md}e^{-\frac{2\pi i l
      m}{p}}I_{l,m}=\\\frac{1}{(d+\chi_{k})\log
    q}\sum_{r=1}^{\infty}f_{k\bmod
    p}\left(r\right)\left(r^{-d-\chi_{k}}-(r+1)^{-d-\chi_{k}}\right).
\end{multline*}

Next, consider the function
\begin{equation*}
  A(z)=\sum_{r=1}^{\infty}f_{k\bmod p}\left(r\right)\left(r^{-z}-(r+1)^{-z}\right).
\end{equation*}
We know that $f_{l}(r)=\bigOh(r^{d-1}\log r)$. Thus, $A(z)$ is analytic for $\Re z>d-1$.

By summation by parts, we can rearrange the series for \begin{math}\Re
  z>d\end{math} and obtain a sum of Dirichlet series
\begin{equation}\label{eq:dirichlet-A-sum}
  A(z)=[p\mid k]e_{\T}S_{1}(z)+i\bfw_{k}'^{\top}\bfones S_{2}(z)-S_{3}(z)+q^{d}\exp\Big(\frac{2\pi ik}{p}\Big)S_{4}(z)
\end{equation} 
with
coefficients $s_{1}(r)$, $s_{2}(r)$, $s_{3}(r)$ and $s_{4}(r)$ respectively. These coefficients
are differences of the four summands in $f_{k\bmod p}(r)$ and $f_{k\bmod
  p}(r-1)$ in~\eqref{eq:coeff-f}, respectively, e.g., 
\begin{align*}
  s_{1}(r)&=\floor{\log_{q}(r)}(r^{d}-(q\floor{r/q})^{d})+(q\floor{r/q})^{d}\\
&\quad-[r>1]\big(\floor{\log_{q}(r-1)}((r-1)^{d}-(q\floor{(r-1)/q})^{d})\\
&\hspace{8em}-(q\floor{(r-1)/q})^{d}\big).
\end{align*}

After some simplifications using
$\floor{\frac{r-1}q}=\floor{\frac rq}-[q\mid r]$ and
 $\floor{\log_{q}(r-1)}=\floor{\log_{q}r}-[r\text{ is a power of }q]$ (for $r\geq2$),
 we obtain
\begin{equation}
  \label{eq:summand-dirichlet-A}
  \begin{aligned}
    s_{1}(r)&=\floor{\log_{q}r}(r^{d}-(r-1)^{d})\\
    &\quad -[q\mid r]q^{d}\floor{\log_{q}rq^{-1}}((rq^{-1})^{d}-(rq^{-1}-1)^{d})\\
    &\quad +[r\neq 1 \text{ is a power of }q]((r-1)^{d}-(r-q)^{d}),\\
    s_{2}(r)&=r^{d}-(r-1)^{d}-[q\mid r]q^{d}\exp\Big(\frac{2\pi ik}{p}\Big)((rq^{-1})^{d}-(rq^{-1}-1)^{d}),\\
    s_{3}(r)&=\bfw_{k}^{\top}(\bfB_{\emptyset}(r)-\bfB_{\emptyset}(r-1)),\\
    s_{4}(r)&=[q\mid
    r]\bfw_{k}^{\top}(\bfB_{\emptyset}(rq^{-1})-\bfB_{\emptyset}(rq^{-1}-1)).
  \end{aligned}
\end{equation}

For \begin{math}\Re z>d\end{math}, we can split up the summation into the
different cases in~\eqref{eq:summand-dirichlet-A}. This yields
\begin{align*}
  S_{1}(z)&=(1-q^{d-z})L(z)+\sum_{j=0}^{d-1}\binom{d}{j}(-1)^{d-j}\frac{1-q^{d-j}}{q^{z-j}-1},\\
  S_{2}(z)&=\Big(1-q^{d-z}\exp\Big(\frac{2\pi ik}{p}\Big)\Big)Z(z),\\
  S_{3}(z)&=\bfw_{k}^{\top}\bfH(z)-B(z),\\
  S_{4}(z)&=q^{-z}\bfw_{k}^{\top}\bfH(z)-q^{-z}B(z)
\end{align*}
where we used~\eqref{eq:B-sum-of-Bs}, \eqref{eq:H-sum-H-C}  and the Dirichlet series defined
in Lemmas~\ref{lem:dirichlet-L}, \ref{lem:dirichlet-Z} and \ref{lem:dirichlet-B}.

Thus, in~\eqref{eq:dirichlet-A-sum}, we obtain
\begin{equation}\label{eq:dirichletseries-general}
  \begin{aligned}
    A(z)&=[p\mid
    k]e_{\T}\sum_{j=0}^{d-1}\binom{d}{j}(-1)^{d-j}\frac{1-q^{d-j}}{q^{z-j}-1}\\
    &\quad+i\bfw_{k}'^{\top}\bfones
    \big(1-q^{d-z}e^{\frac{2\pi ik}{p}}\big)Z(z)\\
    &\quad-\big(1-q^{d-z}e^{\frac{2\pi i
        k}{p}}\big)\bfw_{k}^{\top}\bfH(z)\\
    &\quad+[p\mid
    k]e_{\T}(1-q^{d-z})L(z)\\
    &\quad+\big(1-q^{d-z}e^{\frac{2\pi ik}{p}}\big)B(z).
  \end{aligned}
\end{equation}

We want to evaluate $A$ at $z=d+\chi_{k}$. The factors $1-q^{d-z}e^{\frac{2\pi ik}{p}}$ are zero if and only if
$z=d+\chi_{k}$. Thus, the following  Dirichlet series contribute to~\eqref{eq:dirichletseries-general}:
\begin{itemize}
\item  The Dirichlet series $Z$ only contributes if $k=0$ (Lemma~\ref{lem:dirichlet-Z}).

\item  The Dirichlet series $\bfw_{k}^{\top}\bfH$ has poles at
  $z=d+\chi_{k}$ for $k\in\mathbb Z$  (Lemma~\ref{lem:dirichlet-H}). The possible
  double pole at $z=d$ cancels with the one of $L$. 

\item The residue of the Dirichlet series $L$ contributes to the Fourier
  coefficients (Lemma~\ref{lem:dirichlet-L}). The possible double pole at $z=d$ cancels with that of $\bfw_{0}^{\top}\bfH$.

\item As the Dirichlet series \begin{math}B\end{math} converges
  for \begin{math}\Re z>d-1\end{math} (Lemma~\ref{lem:dirichlet-B}), it does not contribute to the Fourier
  coefficients.
\end{itemize}

As the second order poles of $\bfw_{0}^{\top}\bfH$ and $L$ cancel, the right-hand side of~\eqref{eq:dirichletseries-general} is well
defined for the limit \begin{math}z\rightarrow d+\chi_{k}\end{math}. After computing the limit
and simplifying the
summation, we obtain~\eqref{eq:gen-fourier-coeff}.

Then Lemma~\ref{lem:hoelder} and Bernstein's theorem (cf.~\cite[p.~240]{Zygmund:2002:trigon}) imply the absolute and uniform
convergence of the Fourier series.
\end{proof}

Now we use Theorem~\ref{thm:fourier} to prove Corollary~\ref{cor:delange}.
\begin{proof}[Proof of Corollary~\ref{cor:delange}]
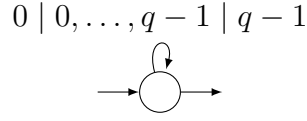
\begin{figure}
    \centering
    \begin{tikzpicture}[auto, initial text=, >=latex, accepting text=,
      accepting/.style=accepting by arrow, every state/.style={minimum
    size=1.3em}]
      \node[state, initial, accepting] (v0) at (3.000000, 0.000000) {};
      \path[->] (v0) edge[loop above] node {$0\mid 0, \ldots, q-1\mid q-1$} ();
    \end{tikzpicture}
    \caption{Transducer for the $q$-ary sum-of-digits function.}
    \label{fig:q-ary-trans}
  \end{figure}
The transducer in Figure~\ref{fig:q-ary-trans} computes the $q$-ary sum-of-digits function
$s_{q}(n)$ and we can use Theorem~\ref{thm:fourier}.

  We transform the Dirichlet series
  \begin{equation*}
    D(z)=\sum_{m\geq 1}(s_{q}(m)-s_{q}(m-1))m^{-z}
  \end{equation*}
in two different ways. This series is absolutely convergent for $\Re z>1$.

First,  we can rearrange the summation of the Dirichlet series $D(z)$ such
that the Dirichlet series $H(z)=\sum_{m\geq1}s_{q}(m)m^{-z}$, defined
in~\eqref{eq:dirichlet-H}, appears. We have
\begin{equation}
  \begin{aligned}
    \lvert H(z)-1\rvert &= \bigOh\Big(2^{-\Re z}+\sum_{m\geq 3} m^{-\Re z}\log
    m\Big)\\
    &= \bigOh\Big(2^{-\Re z}+\int_{2}^{\infty}x^{-\Re z}\log x \,
    dx\Big)\\
    &=\bigOh(2^{-\Re z})
  \end{aligned}
\label{eq:estimate-H-delange}
\end{equation}
for $\Re z> 1$.
By partial summation, we obtain
\begin{align*}
  D(z)&=1-2^{-z}+\sum_{m\geq2}s_{q}(m)(m^{-z}-(m+1)^{-z})\\
  &=1-2^{-z}+\sum_{m\geq 2}s_{q}(m)m^{-z}\big(1-\big(1+m^{-1}\big)^{-z}\big).\\
\end{align*}
Expanding the binomial series yields
\begin{equation}
\begin{aligned}
D(z)&=1-2^{-z}-\sum_{m\geq 2}s_{q}(m)m^{-z}\sum_{l\geq 1}\binom{-z}{l}m^{-l}\\
&=1-2^{-z}-\sum_{l\geq 1}\binom{-z}{l}(H(z+l)-1).
\end{aligned}\label{eq:D-infinite-sum}
\end{equation}

By~\eqref{eq:D-infinite-sum}, we have
\begin{equation*}
  D(z)=1-2^{-z}+zH(z+1)-z-\sum_{l\geq2}\binom{-z}{l}(H(z+l)-1)
\end{equation*}
which is equivalent to
\begin{equation*}
  H(z+1)=\frac 1z D(z)+\frac 1z(2^{-z}-1)+1-\sum_{l\geq2}\frac 1l\binom{-z-1}{l-1}(H(z+l)-1)
\end{equation*}
for $\Re z>1$.
The sum on the right-hand side is holomorphic at $\Re z=0$ because
of~\eqref{eq:estimate-H-delange}. By meromorphic continuation, this equation
also holds for $\Re z= 0$. This yields
\begin{equation}\label{eq:res-H-delange}
  \Res_{z=1+\chi_{k}}H(z)=\Res_{z=\chi_{k}}H(z+1)=\Res_{z=\chi_{k}}\frac 1zD(z).
\end{equation}

On the other hand, we split up the summation in the definition of $D(z)$ into the $q$ equivalence classes modulo $q$
and we use the recursions\footnote{Actually, these recursions are \eqref{eq:recursion-b}.}
\begin{align*}
  s_{q}(qm+\eps)&=s_{q}(m)+\eps
\end{align*}
for $0\leq\eps<q$. This results in
\begin{equation*}
  s_q(m)-s_q(m-1)=1 + [q\mid m]\left(s_q\left(q^{-1}m\right)-s_q\left(q^{-1}m - 1\right) - q\right)
\end{equation*}
for $m \ge 1$. Thus we obtain
\begin{align*}
  D(z)&=\sum_{m\ge 1}\left(1 + [q\mid m]\left(s_q\left(q^{-1}m\right)-s_q\left(q^{-1}m - 1\right) - q\right)\right)m^{-z}\\
  &=\zeta(z)+q^{-z}D(z)-q^{1-z}\zeta(z).
\end{align*}
Thus, we obtain\footnote{Note that this well-known identity can also be derived from $s_q(m)-s_{q}(m-1)=1-(q-1)v_q(m)$, where $v_{q}(m)$
is the $q$-adic valuation of $m$.} 
\begin{equation}\label{eq:dirichlet-D}
  D(z)=\frac{1-q^{1-z}}{1-q^{-z}}\zeta(z).
\end{equation}

This formula yields
\begin{equation}\label{eq:res-D-delange}
  \Res_{z=\chi_{k}}D(z)=-\frac{q-1}{\log q}\zeta(\chi_{k}).
\end{equation}
For $k=0$, we further use the expansion
\begin{equation*}
  \zeta(z)=-\frac 12-\frac 12\log (2\pi)z+\bigOh(z^{2})
\end{equation*}
(cf.\ \cite[\href{http://dlmf.nist.gov/25.6.E1}{25.6.1} and \href{http://dlmf.nist.gov/25.6.E11}{25.6.11}]{NIST:DLMF})
and~\eqref{eq:dirichlet-D} to obtain
\begin{equation}\label{eq:D-exp-delange}
  D(z)=\frac{q-1}{2 z \log{q}}+\frac{(q-1) \log{(2 \pi)}}{2 \log{q}}-\frac{q+1}{4}+\bigOh(z).
\end{equation}

Thus, by \eqref{eq:estimate-H-delange} and \eqref{eq:res-D-delange}, we obtain
\begin{equation*}
 \Res_{z=1+\chi_{k}}H(z)=\frac{1}{\chi_{k}}\Res_{z=\chi_{k}}D(z)=-\frac{q-1}{\chi_{k}\log q}\zeta(\chi_{k})
\end{equation*}
for $k\neq 0$.
For $k=0$, \eqref{eq:D-exp-delange} and \eqref{eq:res-H-delange} yield
\begin{equation*}
  \Res_{z=1}H(z)=\frac{(q-1) \log{(2 \pi)}}{2 \log{q}}-\frac{q+1}{4}.
\end{equation*}

Now,~\eqref{eq:gen-fourier-coeff} with $e_{\T}=\frac{q-1}{2}$ and
$\bfw_{0}'^{\top}=0$ yields~\eqref{eq:delange-fou-coeff}.
\end{proof}

\section{Non-Differentiability --- Proof of
  Theorem~\ref{thm:nondiff}}\label{appendix:non-differentiability}
In this section, we give the proof of the non-differentiability of
$\Psi_{1}(x)$. We follow the method presented by
Tenenbaum~\cite{Tenenbaum:1997:non-derivabilite}, see also Grabner and
Thuswaldner~\cite{Grabner-Thuswaldner:2000:sum-of-digits-negative}.

\begin{proof}[Proof of Theorem \ref{thm:nondiff}]
  Let \begin{math}r=(r_{m-1}\ldots r_{0})_{q}\end{math} be the value of the
  reset sequence \begin{math}(r_{m-1}\ldots r_{0})\end{math} leading to
  state \begin{math}\nu\end{math}.

  Assume that \begin{math}\Psi_{1}\end{math} is differentiable at
  \begin{math}x\in[0,1)\end{math}. Let
  \begin{math}q^{x}=(\eps_{0}\centerdot\eps_{1}\ldots)_{q}\end{math} be the
  standard
  \begin{math}q\end{math}-ary digit expansion choosing the representation
  ending on \begin{math}0^{\omega}\end{math} in the case of
  ambiguity. Further, let \begin{math}x_{k}\end{math} be such that
  \begin{math}q^{x_{k}}=(\eps_{0}\centerdot\eps_{1}\ldots\eps_{k})_{q}\end{math}. Thus,
  we have \begin{math}\lim_{k\rightarrow\infty}x_{k}=x\end{math}. For
  \begin{math}f\in\{0,1\}\end{math}, the
  function \begin{math}L_{f}\colon\integers\rightarrow\integers\end{math} is
  defined as \begin{math}L_{f}(k)=ck+f\end{math} with
  \begin{math}c\end{math} a positive integer such
  that \begin{math}c>\frac1{\xi}-1\end{math}. Define \begin{math}N_{k}=q^{x_{k}+k+L_{f}(k)}\end{math} and
  \begin{math}h(k)=\lfloor q^{ck+\frac
      c{c+1}x_{k}-m-2}\rfloor\end{math}. Let \begin{math}y_{k}\end{math}
  and \begin{math}z_{k}\end{math} be such
  that \begin{math}N_{k}+q^{ck-m-1}r=q^{y_{k}+k+L_{f}(k)}\end{math} and
  \begin{math}N_{k}+q^{ck-m-1}r+h(k)=q^{z_{k}+k+L_{f}(k)}\end{math}.

  From these definitions, we know that
  \begin{align*}
    \frac{h(k)}{N_{k}}&=\Theta(q^{-k}),\\
    N_{k}^{1-\xi}\log N_{k}&=o(h(k))
  \end{align*}
  for \begin{math}k\rightarrow\infty\end{math}. Apart from
  \begin{math}x_k\end{math}, also, \begin{math}y_{k}\end{math}
  and \begin{math}z_{k}\end{math} converge to \begin{math}x\end{math} and
  satisfy the following bounds:
  \begin{align*}
    z_{k}-y_{k}&=\frac1{\log
      q}\frac{h(k)}{N_{k}}+\bigOh\left(\frac{h(k)^{2}}{N_{k}^{2}}\right),\\
    |y_{k}-x_{k}|&=\bigOh(q^{-k}),\\
    x-x_{k}&=\bigOh(q^{-k}).
  \end{align*}

  Now, we compute
  \begin{equation}\label{eq:3}\frac1{h(k)}\sum_{n\in\mathcal
      N_{k}}\T(n)\end{equation}
  in two different ways where \begin{math}\mathcal N_{k}=\{n\in\integers\mid N_{k}+q^{ck-m-1}r\leq n<N_{k}+q^{ck-m-1}r+h(k)\}\end{math}.

  First, observe that \begin{math}q^{ck-1}\mid N_{k}\end{math}
  and \begin{math}h(k)<q^{ck-m-1}\end{math}. Thus, the digit representations
  of the three summands in \begin{math}N_{k}+q^{ck-m-1}r+n\end{math} are not
  overlapping at non-zero digits for \begin{math}n<h(k)\end{math}. Since the
  digit expansion of \begin{math}r\end{math} is a reset sequence, we have
  \begin{equation*}
    \T(N_{k}+q^{ck-m-1}r+n)=\bfe_{\nu}^{\top}\bfb(N_{k}q^{-ck+1})+\T(q^{ck-m-1}r+n)-b(\nu)\end{equation*}
  where \begin{math}\bfe_{\nu}^{\top}\bfb(N)\end{math} is the output of the transducer when starting in state
  \begin{math}\nu\end{math} with input \begin{math}N\end{math}
  and \begin{math}b(\nu)\end{math} is the final output at
  state \begin{math}\nu\end{math}.

  Thus, we have
  \begin{align*}\frac1{h(k)}\sum_{n\in\mathcal
      N_{k}}\T(n)&=\frac1{h(k)}\sum_{0\leq
      n<h(k)}\T(N_{k}+q^{ck-m-1}r+n)\\
    &=\bfe_{\nu}^{\top}\bfb(N_{k}q^{-ck+1})-b(\nu)+\frac1{h(k)}\sum_{n<h(k)}\T(q^{ck-m-1}r+n)
  \end{align*}
  where only the first summand depends on \begin{math}L_{f}(k)\end{math} and
  hence on \begin{math}f\end{math}.

  Taking the difference in \eqref{eq:expected}, there is a second way of computing the sum
  in~(\ref{eq:3}).  Using the periodicity and continuity of
  \begin{math}\Psi_{1}(x)\end{math} yields
  \begin{equation}\label{eq:4}
    \begin{aligned}
      \sum_{n\in\mathcal N_{k}}\T(n)
      &=(N_{k}+q^{ck-m-1}r)e_{\T}(z_{k}-y_{k})+h(k)e_{\T}(x+k+L_{f}(k))\\
      &\quad+(N_{k}+q^{ck-m-1}r)(\Psi_{1}(z_{k})-\Psi_{1}(y_{k}))\\
      &\quad+h(k)\Psi(x)+o(h(k)).
    \end{aligned}
  \end{equation}

  Next, we use our assumption that \begin{math}\Psi_{1}\end{math} is
  differentiable at \begin{math}x\end{math} to replace the difference by the
  derivative
  \begin{equation*}
    \Psi_{1}(z_{k})-\Psi_{1}(y_{k})=\Psi_{1}'(x)(z_{k}-y_{k})+o(|z_{k}-x|)+o(|x-y_{k}|).\end{equation*}

  Now, we insert this into~\eqref{eq:4}, divide by \begin{math}h(k)\end{math}
  and obtain
  \begin{equation*}
    \frac1{h(k)}\sum_{n\in\mathcal N_{k}}\T(n)=\frac
    {e_{\T}}{\log q}+e_{\T}(x+k+L_{f}(k))+\frac 1{\log q}\Psi_{1}'(x)+\Psi_{1}(x)+o(1).
  \end{equation*}

  Thus, we have the following equality
  \begin{multline*}
    \bfe_{\nu}^{\top}\bfb(N_{k}q^{-ck+1})-b(\nu)+\frac1{h(k)}\sum_{n<h(k)}\T(q^{ck-m-1}r+n)=\\\frac
    {e_{\T}}{\log q}+e_{\T}(x+k+L_{f}(k))+\frac 1{\log
      q}\Psi_{1}'(x)+\Psi_{1}(x)+o(1)
  \end{multline*}
  twice, for \begin{math}f\in\{0,1\}\end{math}. Subtracting these two from
  each other yields
  \begin{equation*}
    \bfe_{\nu}^{\top}\bfb(q^{x_{k}+k+2})-\bfe_{\nu}^{\top}\bfb(q^{x_{k}+k+1})=e_{\T}+o(1).\end{equation*}
  Since the left-hand side is an integer, but the right-hand side is not for \begin{math}k\end{math}
  large enough, this contradicts our assumption that \begin{math}\Psi_{1}\end{math} is
  differentiable at \begin{math}x\end{math}.
\end{proof}
\section{Recursions --- Proof of Theorem~\ref{thm:recursion}}
\label{sec:recursion-proof}
In this section, we construct a transducer associated to the sequence defined by
the recursion in~\eqref{eq:recursions-d}. All inequalities, maxima and minima
in this section are considered coordinate-wise.

Define the function  $A\colon\mathbb N_0^{d}\to\mathbb N_{0}^{d}\cup\{\infty\}$ by
\begin{equation*}
A(q^{\kappa}\bfn+\bflambda)=
\begin{cases}
  q^{\kappa_{\bflambda}}\bfn+\bfr_{\bflambda}&\text{if
  }q^{\kappa_{\bflambda}}\bfn+\bfr_{\bflambda}\geq 0,\\
\infty&\text{else}
\end{cases}
\end{equation*}
for $0\leq \bflambda<q^{\kappa}\bfones$ and $\bfn\ge 0$.
So, if $A(\bfn)<\infty$, then the recursion
\eqref{eq:recursions-d} can be used for this argument because the argument on the right-hand
side is non-negative, i.e., $a(\bfn)=a(A(\bfn))+ t_{\bfn \bmod q^{\kappa}}$.

First, we construct a non-deterministic transducer $\Ttilde$. A priori, it has an
infinite number of states; later, we will prove that only finitely many of them
are accessible. We then simplify it to obtain a finite, deterministic,
subsequential, complete transducer $\T$.

The set of states of $\Ttilde$ is
\begin{equation*}
  \{ (\bfl, j)_F \mid \bfl\in\integers^d, j\in \naturals_0\} \cup
  \{ (\bfl, j)_N \mid \bfl\in\integers^d, j\in \naturals_0\}.
\end{equation*}
The initial state is $(0, 0)_F$; all states $(\bfl, j)_F$ are final states with
final output $a(\bfl)$ if $\bfl\ge 0$ and final output $0$
otherwise\footnote{In fact, we will prove that a path with valid input will
  never end in a state $(\bfl, j)_F$ with $\bfl\not\ge 0$, but the framework of
  subsequential transducers requires us to specify a final output even in that
  case. The non-final states $(\bfl, j)_N$ will disappear in the reduction to
  $\T$ anyway.}.
As an abbreviation, we will frequently speak about ``a state $(\bfl, j)$'' if
we do not want to distinguish between $(\bfl, j)_F$ and $(\bfl, j)_N$. We
call $\bfl$ the \emph{carry} and $j$ the \emph{level} of the state $(\bfl, j)$.
A state $(\bfl, j)_F$ is called \emph{simple}, if it is final, $\bfl\ge 0$ and
$j\le \kappa$.

There are two types of transitions in $\Ttilde$, \emph{recursion transitions}
and \emph{storing transitions}. Each state is either the origin of one recursion
transition or of $q^d$ storing transitions.

There is a recursion transition leaving $(\bfl,
j)$ if 
\begin{itemize}
\item $j\ge \kappa$ and
\item $A(q^j\bfn + \bfl)<\infty$ for all $\bfn\ge 0$ with $\bfn\neq 0$.
\end{itemize}
In that case, we write  $\bfl=q^{\kappa}\bfs+\bflambda$ for a $0\leq
\bflambda<q^{\kappa}\bfones$ and the transition leads to the state $(\bfl', j')_N$
with $j'=\kappa_{\bflambda}+j-\kappa$ and
$\bfl'=q^{\kappa_{\bflambda}}\bfs+\bfr_{\bflambda}$. The input label is empty,
the output label is $t_{\bflambda}$. Thus
\begin{equation}\label{eq:recursion-transition-simple-invariant}
  A(q^j\bfn + \bfl) = q^{j'}\bfn + \bfl'
\end{equation}
for $\bfn\ge 0$ with $\bfn\neq 0$. Note that 
\eqref{eq:recursion-transition-simple-invariant} holds for $\bfn=0$ if and only
if $\bfl\ge 0$ and $\bfl'\ge 0$.

Otherwise, there are storing transitions from $(\bfl, j)$ to $(q^j\bfeps+\bfl,
j+1)_F$ with input $\bfeps$ and output $0$ for all $0\le \bfeps<q \bfones$.

We now define the classes $F_1$, \ldots, $F_K$
announced in Section~\ref{sec:rec}. For each accessible cycle in $\Ttilde$ with simple
states and input $0$, the carries of its states form one of these classes. The
other classes are the singletons of those carries $\bfl\ge 0$ in the accessible part
of $\Ttilde$ with $A(\bfl)=\infty$. These sets will turn out to be disjoint by Lemma~\ref{lemma:transducer-bijection}
and the finiteness of $K$ will follow from the
finiteness of the accessible part of $\Ttilde$ (Lemma~\ref{lemma:transducer-is-finite}).

\begin{remark}\label{remark:combinatorial-characterization-outdegree-zero-recursion-digraph}We
  also give a combinatorial description of those classes $F_{1}$, \dots,
  $F_{K}$ which do not come from cycles in $\Ttilde$:
  Let $\bfl\ge 0$ be a carry of an accessible state of $\Ttilde$. Then
  $A(\bfl)=\infty$ if and only if there is a recursion transition from some
  $(\bfl, j)$ to some $(\bfl', j')$ with $\bfl'\not\ge 0$.
\end{remark}
\begin{proof}
  Let $(\bfl, j_0)$ be any accessible state with carry $\bfl$. We use the
  longest path with input $0$ using storing transitions only to arrive in some
  state $(\bfl, j)$---again, finiteness of this process will follow from the
  finiteness of the accessible part and the fact that the levels increase
  along storing transitions. As there is no storing transition leaving $(\bfl, j)$ by
  construction, there is a recursion transition from $(\bfl, j)$ to some
  $(\bfl', j')$. By the remark following
  \eqref{eq:recursion-transition-simple-invariant}, $\bfl'=A(\bfl)$ or
  $\bfl'\not\ge 0$.
\end{proof}

As usual, if reaching a state which is the origin of a transition with empty
input, the process may stay in that state or may continue to the destination
state writing the output of the transition without reading an input. This is
the reason why the transducer is non-deterministic.

Note that in our case, transitions with empty input (i.e., recursion
transitions) lead to non-final states and transitions with non-empty input
(i.e., storing transitions) lead to final states. Combined with the fact that
each state is either the origin of one recursion transition or of $q^d$ storing
transitions, processing an input is in fact deterministic: For every admissible
input---we do not allow leading zeros---, there exists exactly one path leading
from the initial state to a final state with the given input. This will enable
us to simplify the transducer $\Ttilde$ to a deterministic transducer $\T$
later on.

We need the property that the carries of accessible states are not ``too negative'':
\refstepcounter{lemma}\label{lemma:transducer-non-negativity}\addtocounter{lemma}{-1}

\begin{lemma}
  \begin{enumerate}
  \item If $(\bfl, j)$ is an accessible state, then
    \begin{equation}\label{eq:transducer-non-negativity-condition}
      q^j \bfn + \bfl\ge 0
    \end{equation}
    holds for all $\bfn\ge 0$ with $\bfn\neq 0$.
  \item If $d\ge 2$ and $(\bfl, j)$ is an accessible state, then
    \begin{equation*}
      \bfl\ge 0.
    \end{equation*}
  \item Any accessible transition with input $\bfeps\neq 0$ leads
    to a state $(\bfl,j)$ with $\bfl\geq 0$.
  \item If $d=1$ and $(l, j)$ is an accessible state, then
    \begin{equation*}
      l\ge \lmin=\min_\lambda\Biggl\{0, \frac{-1+\frac{r_\lambda}{q^{\kappa_\lambda}}}{\frac{1}{q^{\kappa_\lambda}}-\frac1{q^\kappa}}\Biggr\}.
    \end{equation*}
  \end{enumerate}
\end{lemma}
\begin{proof}The first assertion is easily shown by induction and \eqref{eq:recursion-transition-simple-invariant}. The second
  assertion follows by induction and from the assumption that $r_{\bflambda}\ge
  0$ holds for all $\bflambda$. To prove the third assertion, we use \eqref{eq:transducer-non-negativity-condition}
  on the originating state of the transition.

  The last assertion is shown by induction. It is clearly valid in the initial state.
  For storing transitions, the value of $l$ is non-decreasing. If there is a recursion transition from some $(l, j)$ to some $(l', j')_N$, we have
  \begin{align*}
    l'&=q^{\kappa_\lambda}\left\lfloor\frac{l}{q^\kappa}\right\rfloor+r_\lambda \ge
    q^{\kappa_\lambda}\left(\frac{l}{q^{\kappa}}-1+\frac{r_\lambda}{q^{\kappa_\lambda}}\right)\\
    &\ge q^{\kappa_\lambda}\left(\frac{\lmin}{q^{\kappa}} + \lmin\left(\frac1{q^{\kappa_\lambda}}-\frac1{q^\kappa}\right)\right)=\lmin.
  \end{align*}
\end{proof}

As leading zeros are not allowed, the last transition in the computation path
of any valid input has input $\bfeps\neq0$ and thus leads to a state with a
non-negative carry.

For our further investigations and finally the correctness proof, we need a suitable
invariant:

\begin{lemma}
  Consider a path from $(\bfl,j)$ to $(\bfl',j')$ with input label
  $\bfeps_{m-1}\ldots\bfeps_{0}$, output label $\delta_{m'-1}\ldots\delta_0$ using
  $L$ recursion transitions and $\bfn\ge 0$. Thus $m'$ is the number of
  transitions and $m=m'-L$ is the number of storing transitions.

  If $\bfn\neq 0$ or if the last
  transition is a storing transition with non-zero input $\bfeps_{m-1}$, then
  \begin{align}
    A^L(q^{j}(q^m\bfn + (\bfeps_{m-1}\ldots\bfeps_{0})_q )+\bfl)&=
    q^{j'}\bfn + \bfl',\label{eq:transducer-invariant-A}\\
\intertext{and, if the recursion \eqref{eq:recursions-d} is well-posed,}
    a(q^{j}(q^m\bfn + (\bfeps_{m-1}\ldots\bfeps_{0})_q )+\bfl)&=
    a(q^{j'}\bfn + \bfl') + \sum_{k=0}^{m'-1}\delta_k\label{eq:transducer-invariant-a}.
  \end{align}
\end{lemma}
\begin{proof}First consider the case that the path consists of a single
  transition. If it is a storing transition, then $L=0$, $m=1$, and all
  assertions follow from the definition and Lemma~\ref{lemma:transducer-non-negativity}.
  On the other hand, if
  the transition is a recursion transition, we have $L=1$, $m=0$, and all
  assertions again follow from the definition, Lemma~\ref{lemma:transducer-non-negativity} and \eqref{eq:recursion-transition-simple-invariant}.

  By induction on the length of the path, we obtain
  \eqref{eq:transducer-invariant-A} and \eqref{eq:transducer-invariant-a}.
\end{proof}

We are now able to prove the finiteness of the accessible part.
 
\begin{lemma}\label{lemma:transducer-is-finite}
  The transducer has a finite number of accessible states.
\end{lemma}
\begin{proof}
  For a recursion transition from $(\bfl,j)$ to $(\bfl',j')_{N}$, we have
  $j>j'$. Thus, there are no infinite paths consisting only of recursion
  transitions. In particular, there exist no cycles of recursion transitions.

 For $d=1$, let $J\ge\kappa$ be
minimal such that $q^{J-\kappa}\geq-\bigl\lfloor\frac{\lmin}{q^\kappa}\bigr\rfloor-\min_{\lambda}q^{-\kappa_{\lambda}}r_{\lambda}$. Then $A(q^j+l)<\infty$ holds for all accessible states $(l, j)$ with $j\ge J$.
This implies $j\leq J$ for all accessible states $(l,j)$. For $d\geq 2$, we have $j\leq \kappa=:J$ for all accessible states
$(\bfl,j)$. Thus there are at most $J$ consecutive recursion transitions.

To prove that only finitely many states are accessible, we introduce the
  notion of heights of states: The height of a
  state \begin{math}(\bfl,j)\end{math} is defined to be \begin{math}\bfh=\bfl
    q^{-j}\end{math}.  If there exists a storing transition
  from \begin{math}(\bfl,j)\end{math} of height \begin{math}\bfh\end{math} to
  \begin{math}(\bfl',j')_F\end{math} of height \begin{math}\bfh'\end{math}, we
  have $\frac 1q\bfh\leq \bfh'\leq\frac 1q\bfh+\bfones$.  If there exists a recursion transition
  from \begin{math}(\bfl,j)\end{math} of height \begin{math}\bfh\end{math} to
  \begin{math}(\bfl',j')_{N}\end{math} of height \begin{math}\bfh'\end{math}, we
  have $\bfh+\bfs^{-}-\bfones\leq \bfh'\leq\bfh+\bfs^{+}$ where \begin{math}\bfs^{+}=\max_{\bflambda}
    \{\bfr_{\bflambda} q^{-\kappa_{\bflambda}},0\}\end{math}
  and \begin{math}\bfs^{-}=\min_{\bflambda} \{\bfr_{\bflambda}
    q^{-\kappa_{\bflambda}},0\}\end{math}.

Assume that there is a path from $(\bfl,j)$ of height $\bfh$ to
$(\bfl',j')$ of height $\bfh'$ with $L\leq J$ recursion transitions and one
storing transition (in this order). Then we have
\begin{equation*}
\frac
1q\bfh+\frac{J}{q}(\bfs^{-}-\bfones)\leq \bfh'<\frac 1q\bfh+\frac{J}{q}\bfs^{+}+\bfones.
\end{equation*}

We can subdivide every path in the transducer starting with the initial state
into a sequence of such paths and a final path consisting of only recursion
transitions. Let $\bfh_{m}$ be the sequence of heights of the states where the
subpaths starts. Then, we have
\begin{equation*}
\frac
1q\bfh_{m}+\frac{J}{q}(\bfs^{-}-\bfones)\leq \bfh_{m+1}<\frac 1q\bfh_{m}+\frac{J}{q}\bfs^{+}+\bfones.
\end{equation*}
Iteration leads to
\begin{equation*}
  \frac{J(\bfs^{-}-\bfones)}{q-1}\le \bfh_{m}\le \frac{J\bfs^{+}+q\bfones}{q-1}
\end{equation*}
for all $m$.
Therefore, the height $\bfh$ of an
accessible state is bounded. Since \begin{math}0\leq
    j\leq J\end{math} is also bounded, the integer
  carry \begin{math}\bfl=q^{j}\bfh\end{math} of an accessible
  state \begin{math}(\bfl,j)\end{math} can only take finitely many different
  values. The accessible part of the transducer is thus finite.
\end{proof}

\begin{lemma}\label{lemma:simple-state-path}
  Let $\calP$ be an infinite path with input zero starting at some state of level $j$
  such that all of its states have non-negative carries. Then, after at most
  $j$ transitions, it reaches a state $(\bfl_0, \kappa)$. From that point on,
  it only passes through simple states, namely
  \begin{align*}
    (\bfl_0, \kappa), {}&(\bfl_1, j_1)_N, (\bfl_1, j_1+1)_F, \ldots, (\bfl_1, \kappa)_F,\\
    &(\bfl_2, j_2)_N,  (\bfl_2, j_2+1)_F,
    \ldots, (\bfl_2, \kappa)_F,\\ &(\bfl_3, j_3)_N,  (\bfl_3, j_3+1)_F, \ldots,
    (\bfl_3, \kappa)_F,\\
    &\ldots
  \end{align*}
  where  $\bfl_i=A(\bfl_{i-1})$ and $j_{i}=\kappa_{\bfl_{i-1}\bmod q^{\kappa}}$
  for $i\ge 1$.
\end{lemma}
\begin{proof}
  Denote the first state of $\calP$ by $(\bfl, j)$. 

  First, assume that $j\ge \kappa$. As storing transitions always
  increase the level and the levels are bounded by
  Lemma~\ref{lemma:transducer-is-finite}, the path has to contain at least one
  recursion transition. Thus the path starts with $k\ge 0$ storing transitions
  leading from $(\bfl, j)$ to $(\bfl, j+k)$, followed by a recursion transition
  from $(\bfl, j+k)$ to $(\bfl', j')$. By assumption, we have
  $\bfl\ge 0$ and $\bfl'\ge 0$. Thus 
  $A(\bfl)=\bfl'\neq\infty$ by
  \eqref{eq:recursion-transition-simple-invariant}. Therefore, there is a
  recursion transition leaving $(\bfl, j)$, i.e., there were no leading storing
  transitions. Recall that $j'<j$ holds for any recursion transition.
  We repeat the argument at most $j-\kappa$ times until we reach a
  simple state.

  If we are in a simple state $(\bfl', j')$ with $j'<\kappa$, the next $\kappa-j'$
  steps will be storing transitions, leading to $(\bfl', \kappa)$. This means
  that after at most $j$ steps, we reach a state $(\bfl_0, \kappa)$.

  We now apply the argument of the second paragraph again. Thus a recursion
  transition leads to $(\bfl_1, j_1)$ with $\bfl_1=A(\bfl_0)$ and
  $j_1=\kappa_{\bfl_0 \bmod q^\kappa}$.

  The remainder of the lemma follows by induction.
\end{proof}

As an auxiliary structure for deciding the well-posedness of the recursion, we
introduce the \emph{recursion digraph} $\calR$. It has set of vertices $\naturals_0^d$
and arcs $(\bfn, A(\bfn))$ with label $t_{\bfn\bmod q^\kappa}$ for all $\bfn\in \naturals_0^d$ with
$A(\bfn)<\infty$. Thus $a(\bfn)$ can be computed from the successor of $\bfn$
in $\calR$ using the recursion \eqref{eq:recursions-d}. By definition, each
vertex of $\calR$ has out-degree $1$ or $0$. Each component of $\calR$ is a
functional digraph or a rooted tree (oriented towards the root).

If
\begin{equation*}
  \|\bfn\|_{\infty}> \max_{\bflambda}\frac{\|\bflambda\|_{\infty}+\|\bfr_{\bflambda}\|_{\infty}}{q^\kappa-
   q^{\kappa_{\bflambda}}},
\end{equation*}
we have
\begin{equation*}
  q^{\kappa}\|\bfn\|_\infty - \|\bflambda\|_\infty > q^{\kappa_{\bflambda}}
  \|\bfn\|_\infty + \|\bfr_{\bflambda}\|_\infty 
\end{equation*}
and therefore
\begin{equation*}
  \|q^{\kappa}\bfn+\bflambda\|_{\infty} 
  > \|q^{\kappa_{\bflambda}} + \bfr_{\bflambda}\|_\infty 
\end{equation*}
for all $0\leq \bflambda<q^{\kappa}\bfones$. Thus we have
$\|\bfn'\|_\infty<\|\bfn\|_\infty$ for all but finitely many arcs $(\bfn,
\bfn')$ of $\calR$.

Thus for every vertex of $\calR$, there is a unique path starting in this
vertex and leading to a vertex with out-degree $0$ or a finite cycle.

From this description, it is clear that the recursion is well-posed if and only
if
\begin{itemize}
\item the sum of the labels of each cycle in $\calR$ is $0$ and
\item the set $\calI$ consists of one element for every cycle in $\calR$ as well as of the
  vertices with out-degree $0$ in $\calR$.
\end{itemize}

We now prove the essential connection between the recursive digraph and the
transducer $\Ttilde$. This also implies that the classes $F_{1}$, \dots,
$F_{K}$ are disjoint.

\begin{lemma}\label{lemma:transducer-bijection}
  There exists a bijection between cycles in the recursive digraph $\calR$ and
  accessible cycles in the transducer $\Ttilde$ with input $0$ and simple
  states. Corresponding cycles under this bijection have the same output sum
  and sum of labels.
\end{lemma}

\begin{proof}
 Let $\bfn_{0}$, \dots, $\bfn_{L}=\bfn_{0}$ be a cycle in the recursive
 digraph with $\bfn_{R}\geq 0$ for all $0\leq R<L$.

 Let $k_0$ be the length of the path $\calP_0$ in $\Ttilde$ starting in the
 initial state and reading the $q$-ary
 expansion of $\bfn_0$.

 We determine the destinations of certain paths in the transducer associated with
 the cycle in the recursive digraph.

 \begin{statement}\label{statement:transducer-path-destination}
   Let $k\ge k_0$ and $\calP$ be the path from the initial state $(0, 0)$ to
   $(\bfl, j)$ of length $k$ whose input label is the $q$-ary expansion of
   $\bfn_0$, padded with leading zeros. Assume that the number of recursion
   transitions in this path is $LQ+R$ for some $Q\ge 0$ and $0\le R<L$. Then
   $\bfl=\bfn_R\ge 0$.
 \end{statement}
 \begin{proof}[Proof of Statement~\ref{statement:transducer-path-destination}]
   Let $k'=k-(LQ+R)$ be the number of storing transitions of $\calP$.
   By \eqref{eq:transducer-invariant-A}, we have
   \begin{equation}\label{eq:1}
     A^{LQ+R}(q^{k'}\bfn+\bfn_{0})=q^{j}\bfn+\bfl
   \end{equation}
   for $\bfn\geq 0$, $\bfn\neq 0$.

   Note that for $M\ge \kappa$ and $\bfn\equiv\bfn'\pmod{q^M}$ with
   $A(\bfn)<\infty$ and $A(\bfn')<\infty$, the definition of $A$ implies
   $A(\bfn)\equiv A(\bfn') \pmod{q^{M-\kappa}}$. 
   
   Together with the definitions of $\bfn_R$ and the recursive digraph $\calR$
   as well as \eqref{eq:1}, this implies
   \begin{align*}
     \bfn_R&=A^{LQ+R}(\bfn_0)\equiv
     A^{LQ+R}(q^{k'+M}\bfones+\bfn_0)\\
     &=q^{j+M}\bfones+\bfl \pmod{q^{k'+M-(LQ+R)\kappa}}
   \end{align*}
   for sufficiently large $M$. Coarsening yields
   \begin{equation*}
     \bfn_R\equiv \bfl \pmod{q^{M-(LQ+R)\kappa}},
   \end{equation*}
   still valid for sufficiently large $M$. As $\bfl$ is bounded by
   Lemma~\ref{lemma:transducer-is-finite}, this implies $\bfn_R=\bfl$.
 \end{proof}
 Now, we conclude the proof of Lemma~\ref{lemma:transducer-bijection}.

 Let $\calP$ be the infinite path in $\Ttilde$ starting at the destination of $\calP_0$ and
 reading zeros. By Lemma~\ref{lemma:simple-state-path} applied to $\calP$
 together with Statement~\ref{statement:transducer-path-destination} applied to
 $\calP_0$ concatenated with prefixes of $\calP$, $\calP$ leads to a cycle in
 $\Ttilde$.  Its states are simple and have carries $\bfn_0$, \ldots,
 $\bfn_{L-1}$ and levels determined by $\bfn_0$, \ldots, $\bfn_{L-1}$ as in
 Lemma~\ref{lemma:simple-state-path}.

 This construction defines a map from the cycles of the recursive digraph $\calR$ to
 the accessible cycles with input $0$ in the transducer with simple states.
 This map is injective by construction. Under this map, the sum of the labels
 of the cycle in $\calR$ equals the sum of output labels of the cycle in
 $\Ttilde$ by construction.

 On the other hand, let
  \begin{align*}
    &(\bfn_0, j_0),  (\bfn_0, j_0+1), \ldots, (\bfn_0, \kappa),\\
    &(\bfn_1, j_1), (\bfn_1, j_1+1), \ldots, (\bfn_1, \kappa),\ldots\\
    &(\bfn_{L-1}, j_{L-1}),  (\bfn_{L-1}, j_{L-1}+1), \ldots, (\bfn_{L-1},
    \kappa),\\
    &(\bfn_0, j_0)
  \end{align*}
 be an accessible cycle of simple states in the transducer with input $0$. 
 Lemma~\ref{lemma:simple-state-path} yields
 $A(\bfn_{R})=\bfn_{R+1\bmod L}\geq 0$ for $0\leq R<L$. Thus, this cycle in the transducer is
 the image of the cycle $\bfn_{0}$, \dots,
 $\bfn_{L}=\bfn_{0}$ in the recursive digraph. Thus the map is surjective.
\end{proof}

To use Theorem~\ref{thm:asydist}, we simplify $\Ttilde$ to obtain the deterministic transducer $\T$, that is
one without transitions with empty input.
As a first step, we
remove all non-accessible states. By
Lemma~\ref{lemma:transducer-is-finite},
this leaves us with finitely many states.

By Lemma~\ref{lemma:transducer-is-finite} and the fact that recursion
transitions decrease the level, the length of paths consisting of recursion
transitions only is bounded.
As a recursion transition always leads to a non-final state, processing an input
never ends with a recursion transition. 

Consider a recursion transition from $(\bfl,
j)$ to $(\bfl', j')_N$ with output $t$ such that no recursion transition
originates in $(\bfl', j')_N$. For each transition originating in
$(\bfl', j')_N$, say to some $(\bfl'',j'')_F$ with input $\bfeps$ and output $t'$, we
insert a storing transition from $(\bfl, j)$ to $(\bfl'', j'')_F$ with input
$\bfeps$ and output $t+t'$. Then, the recursion transition from $(\bfl, j)$ to $(\bfl',
j')_{N}$ is removed. The number of recursion transitions decreased by one and the new
transducer generates the same output as the old transducer.
We repeat this process until there are no more recursion transitions. Then, all
non-final states are inaccessible and are removed.

\begin{proof}[Proof of Theorem~\ref{thm:recursion}]
  By Lemma~\ref{lemma:transducer-bijection} and the characterization of
  well-posedness via the recursive digraph, the recursion
  \eqref{eq:recursions-d} is well-posed if and only if $\calI$ consists of
  exactly one representative of each of the sets $F_j$, 
  $1\le j\le K$, and if $\Ttilde$ has no cycle with simple states, input $0$
  and non-vanishing output sum.

  We now show that the cycles of simple states with input $0$ in $\T$ are exactly the
  reductions of the cycles of simple states with input $0$ in $\Ttilde$.
  As a cycle with simple states and input $0$ in $\Ttilde$ does not have
  consecutive recursion transitions (cf.\ Lemma~\ref{lemma:simple-state-path}),
  it is reduced to a cycle with simple states in $\T$. On the other hand,
  consider a cycle
  of $\Ttilde$ with input $0$ containing a non-simple state. If there is a
  state of level $>\kappa$, the state with largest level is final and is not
  removed. If all states have level $\le \kappa$, then there are no two
  consecutive recursion transitions, so no negative carry is completely removed
  from the cycle in the reduction to $\T$. Therefore, such a cycle is not
  reduced to a cycle with simple states and input $0$ in $\T$.

  Therefore, the assertion on well-posedness is proved.

  To prove correctness of the transducer, we
  use~\eqref{eq:transducer-invariant-A} with $(\bfl, j)=(0, 0)$, the joint
  $q$-ary expansion of $\bfn$ as input leading to some state $(\bfl', j')_F$ with
  output $\delta_{m'-1}\ldots \delta_0$. By
  Lemma~\ref{lemma:transducer-non-negativity}, we have $\bfl'\ge 0$ because the
  last transition is a storing transition with non-zero input. Thus by \eqref{eq:transducer-invariant-a},
  $a(\bfn)=a(\bfl')+\sum_{k=0}^{m'-1} \delta_k$. As the final output of
  $(\bfl', j')_F$
  is defined to be $a(\bfl')$, we obtain
  $\T(\bfn)=a(\bfl')+\sum_{k=0}^{m'-1} \delta_k=a(\bfn)$, as requested.
\end{proof}

\bibliographystyle{amsplain}
\bibliography{bib}

\end{document}
